\documentclass[a4paper,oneside]{amsart}
\usepackage[T2A,T1]{fontenc}
\usepackage[utf8]{inputenc}
\usepackage[english]{babel}
\usepackage[textheight=500pt,textwidth=344pt]{geometry}
\usepackage{amsmath}
\usepackage{amssymb}
\usepackage{amsthm}
\usepackage{mathtools}
\usepackage{xcolor}
\usepackage{tikz}
\usetikzlibrary{shapes,calc,patterns}
\usepackage{url}
\usepackage{enumerate}
\usepackage{braket}
\usepackage{mleftright}
\usepackage{array}
\usepackage{tabularx}%
\usepackage{booktabs}%

\theoremstyle{definition}
\newtheorem{definition}{Definition}[section]

\theoremstyle{plain}
\newtheorem{theorem}[definition]{Theorem}
\newtheorem{proposition}[definition]{Proposition}
\newtheorem{lemma}[definition]{Lemma}
\newtheorem{corollary}[definition]{Corollary}

\theoremstyle{remark}
\newtheorem{remark}[definition]{Remark}

\numberwithin{equation}{section}

\DeclareMathOperator{\Pol}{\mathsf{Pol}}
\DeclareMathOperator{\Inv}{\mathsf{Inv}}
\DeclareMathOperator{\Alg}{Alg}
\DeclareMathOperator{\Clo}{Clo}
\DeclareMathOperator{\POL}{Pol}
\DeclareMathOperator{\Con}{Con}
\DeclareMathOperator{\Hom}{Hom}
\DeclareMathOperator{\Id}{Id}
\newcommand{\variety}[1]{\mathfrak{#1}}
\newcommand{\Bcl}[1]{\mathsf{#1}}
\newcommand{\cloneL}{\mathcal{L}}%
\newcommand{\aP}{\mathbf{aP}}%
\newcommand{\AP}{\mathbf{AP}}%
\newcommand{\SL}{\mathbf{SL}}%
\newcommand{\Ltwo}{\mathbf{L}_2}%
\newcommand{\TN}{\mathbf{TN}}%
\newcommand{\aipii}{\mathbf{a}_{\infty}\pi_\infty}%
\newcommand{\Aipii}{\mathbf{A}_{\infty}\pi_\infty}%
\newcommand{\fipii}{f_\pi^\infty}%
\DeclareMathOperator{\id}{id}
\DeclareMathOperator{\minority}{mnr}
\DeclareMathOperator{\majority}{maj}
\DeclareMathOperator{\ps}{ps}
\DeclareMathOperator{\rght}{right}
\DeclareMathOperator{\lft}{left}
\DeclareMathOperator{\plus}{plus}
\newcommand*{\pluszero}{\plus_0}
\newcommand{\eni}[2][n]{e\arii{#1}_{#2}}
\newcommand{\cna}[2][n]{c\arii{#1}_{#2}}
\newcommand{\ca}[1][a]{c_{#1}}
\newcommand{\dup}{\mathsf{dup}^3}%
\newcommand{\negate}[1]{\overline{#1}}
\newcommand{\ab}[1]{{\mathbf{#1}}}
\newcommand{\type}[2]{\mathtt{typ}(#1,#2)}
\newcommand{\Type}[1]{\mathtt{typ}(#1)}
\newcommand{\Cg}[2][\ab{A}]{\mleft\langle#2\mright\rangle_{\Con{#1}}}%
\newcommand{\Sg}[2][\ab{A}]{\mleft\langle#2\mright\rangle_{#1}}%
\newcommand{\algop}[2]{({#1}; {#2})}
\newcommand*{\algeq}{\sim_{\text{alg}}}
\newcommand{\interval}[2]{\mathtt{I}[#1,#2]}
\newcommand{\deltarelation}[1][A]{\Delta_{#1}^{(4)}}
\newcommand{\deltathree}[1][A]{\Delta_{#1}^{(3)}}
\newcommand{\clonelat}{\mathfrak{L}}%
\newcommand{\bottom}[1]{0_{#1}}
\newcommand{\uno}[1]{1_{#1}}

\let\numberset=\mathbb%
\newcommand{\N}{\numberset{N}}
\newcommand{\Z}{\numberset{Z}}

\newcommand{\finset}[1]{[{#1}]}
\renewcommand{\vec}[1]{{\boldsymbol{#1}}}
\let\tp=\vec
\newcommand{\ari}[1]{_{#1}}
\newcommand{\arii}[1]{^{[#1]}}
\newcommand{\relset}[1]{\mathcal{#1}}
\newcommand{\funset}[1]{\mathcal{#1}}
\newcommand{\potenza}[1]{\mathcal{P}(#1)} 
\newcommand{\restrict}[1]{\rvert_{#1}}
\newcommand{\defeq}{\mathrel{\mathop:}=}
\newcommand{\eqdef}{\mathrel{\mathopen={\mathclose:}}}
\let\subs=\subseteq
\renewcommand{\set}[1]{\mleft\{#1\mright\}}%
\DeclareMathOperator{\GF}{\mathsf{GF}}
\DeclareMathOperator{\rad}{rad}
\DeclareMathOperator{\hw}{wt}
\newcommand{\w}[1]{\hw(#1)}
\DeclareMathOperator{\arity}{ar}
\newcommand{\crd}[1]{\mleft\lvert#1\mright\rvert}
\newcommand{\pair}[1]{\bigl(\begin{smallmatrix}#1\end{smallmatrix}\bigr)}
\providecommand{\tablename}{Tab.}%

\newcolumntype{C}[1]{>{\centering}p{#1}}

\title{On when the union of two algebraic sets is algebraic}
\author{Erhard Aichinger}
\address[Erhard Aichinger]{
Institut f\"ur Algebra,
Johannes Kepler Universit\"at Linz, Altenberger Str.~69, 4040 Linz, Austria}
\email{erhard@algebra.uni-linz.ac.at}

\author{Mike Behrisch}
\address[Mike Behrisch]{
Institut f\"ur Diskrete Mathematik und Geometrie,
Technische Universit\"at Wien, Wiedner Hauptstr.~8--10, 1040 Wien, Austria}
\address{Institut f\"ur Algebra,
Johannes Kepler Universit\"at Linz, Altenberger Str.~69, 4040 Linz, Austria}
\email{behrisch@logic.at}

\author{Bernardo Rossi}
\address[Bernardo Rossi]{
Institut f\"ur Algebra,
Johannes Kepler Universit\"at Linz, Altenberger Str.~69, 4040 Linz, Austria}
\email{bernardo.rossi@jku.at}

\subjclass[2020]{08A40, (08A62, 08B05, 08B10, 20F70)}
\keywords{
universal algebraic geometry, algebraic set, equational domain,
equational additivity, equationally additive clone}
\thanks{Supported by the Austrian Science Fund (FWF):~P33878.}
\begin{document}
\renewcommand{\figurename}{Fig.}%

\date{\today}

\begin{abstract}
  In universal algebraic geometry, an algebra is called an \emph{equational
  domain} if the union of two algebraic sets is algebraic.
  We characterize equational domains, with respect to polynomial equations,
  inside congruence permutable varieties, and with respect to
  term equations, among all algebras of size two and all algebras of
  size three with a cyclic automorphism.
  Furthermore, for each size at least three, we prove that, modulo term
  equivalence, there is a continuum of equational domains of
  that size.
\end{abstract}
\maketitle

\section{Introduction} \label{sec:intro}

A basic fact in classical algebraic geometry is that the union of
two algebraic sets is again algebraic. In universal algebraic
geometry, which studies the algebraic sets over
an arbitrary algebraic structure~\cite{Plo98,BauMyaRem99,DanMyaRem12},
this is no longer true in general. In~\cite{DanMyaRem10},
algebras with the property
that the union of two algebraic sets is algebraic have
been called \emph{equational domains};
for example, a commutative ring with unity
is an equational domain if and only it is an integral domain.
In such algebras, the non-empty algebraic sets coincide with the
non-empty closed sets of a topology, which is called \emph{Zariski
topology}~\cite{DanMyaRem10} as in classical algebraic geometry.

In this paper, we seek to characterize equational domains.
A first observation is that every equational domain is finitely subdirectly
irreducible. For a more detailed study, we need to specify
whether the equations defining algebraic sets involve \emph{term functions}
or \emph{polynomial functions}; the difference lies in whether
constants from the algebra are allowed. In~\cite{DanMyaRem10},
an algebraic set is defined as
the solution set of a system of \emph{term equations};
we follow this viewpoint and treat polynomial equations
by passing from an algebra to its expansion with all constant operations.
From the fundamental result~\cite[Theorem~2.5]{DanMyaRem10}, we see
that if an algebra is an equational domain, then so is its expansion
with all constant operations; in other words, if it is an equational domain
with respect to term functions, then it is an equational domain with
respect to polynomial functions. The converse is not true, as witnessed, e.g.,
by the alternating group~$A_5$ (cf.~\cite[Proposition~2.19,
Claim~2.22(4), Corollary~2.31]{DanMyaRem10}).
For algebras in congruence permutable varieties, we obtain a structural
description of those algebras that are equational domains with respect
to polynomial equations.
Our description uses a generalization of the ideal product in rings to
universal algebra, the binary commutator~\cite{Smith, McKMcnTay88, FreMck87}.
The equational domains inside congruence permutable varieties,
with respect to polynomials,
are then those algebras~$\ab{A}$ with at least two elements that
satisfy $[\alpha, \beta] > 0_A$ for all congruences  $\alpha, \beta > 0_A$
(Theorem~\ref{teor:equ_additive_cofinite_non_abelianity});
for a finite algebra this is equivalent
to saying that the algebra is subdirectly irreducible with non-Abelian
monolith. In each of these finite algebras,
every subset of~$A^n$ is algebraic, hence they all have the same
collection of algebraic sets, in other words,
they are \emph{algebraically equivalent} (cf.~\cite{Pin17a}).

One can view being an equational domain as a property of the clone
of term functions of an algebra. Following~\cite{Pin17a} we say that a
clone~$\funset{C}$ on~$A$
is \emph{equationally additive} if $\algop{A}{\funset{C}}$ is an equational domain.
When~$\funset{C}$ does not contain all constant operations,
we do not have a complete description of equationally additive
clones, even when they contain a Ma\v{l}cev operation.
One difficulty in finding a structural description is explained by
the fact that every finite algebra is weakly
isomorphic to an algebra that is polynomially equivalent to
the quotient of an equational domain
modulo its monolith (Theorem~\ref{teor:Erhardex13b}).
However, we obtain a complete description of two-element equational domains:
A two-element algebra is an equational domain if and only if
it generates a congruence distributive variety;
in Section~\ref{section:equatio_add_boolean_clones}
the order filter of equationally additive clones on a two-element
set is described in detail
(cf.\ Theorem~\ref{thm:char-Boolean-eqn-additive}).
Part of this description carries over to all
\emph{E-minimal} algebras; these are finite algebras in which every
idempotent unary polynomial function is either bijective or constant, which
is the case, e.g., for all finite $p$-groups.
Again, an E-minimal algebra is an equational domain if and only if
it generates a congruence distributive variety
(Theorem~\ref{teor:the_TCT_type_of_E_minimal_algebras}).
A similar description can be obtained for all clones on a three-element
set that are contained
in the clone of self-dual operations (cf.~\cite{Zhu15}): such a
clone~$\funset{C}$ is
equationally additive if and only if $\algop{A}{\funset{C}}$
generates a congruence distributive variety
(Theorem~\ref{teor:self-dual_operations_char}).

Finally, we investigate the number of equationally additive clones
on a finite set. Modulo algebraic equivalence, this number is
finite~\cite[Theorem~3]{Pin17a},
but as we will see, there can be infinitely many equationally additive
clones on a finite set that all induce the same algebraic sets.
We prove that  on each finite set with at least three elements,
there is a continuum of equationally additive clones
(Theorem~\ref{teor:countinously_many_eq_additive_clones_on_the_4_element_set}),
and we determine for which finite Abelian groups
the number of equationally additive clones above the clone of polynomial
functions is infinite (Theorem~\ref{thm:char-eqn-add-poly-expansions-ab-grp}).

\section{Notation and Preliminaries}\label{sec:preliminaries}
We write~$\N$ for the set of positive integers and, for $n\in \N$, let
$\finset{n}\defeq \{1,\dots, n\}$.
For a set~$A$, the $i$-th component of $\vec{a}\in A^n$ is denoted
by~$a_i$ and~$\vec{a}(i)$, and $\potenza{A}$ is the power set of~$A$.
A \emph{relation} on~$A$ is an element of $\bigcup_{n\in\N}
\potenza{A^n}$, and
an \emph{operation} on~$A$ is an element of $\bigcup_{n\in\N}A^{A^n}$.
For $n\in\N$ and $i\in\finset{n}$, the $i$-th $n$-ary \emph{projection}
is the operation~$\eni{i}\colon A^n\to A$ that is given by
$\eni{i}(x_1,\dotsc,x_n)\defeq x_i$ for all $x_1,\dotsc,x_n\in A$; we
also abbreviate $\id_A\defeq\eni[1]{1}$.
For $a\in A$ the $n$-ary \emph{constant} with value~$a$ is the
operation $\cna{a}\colon A^n\to A$ given for all $x_1,\dotsc,x_n\in A$
by $\cna{a}(x_1,\dotsc,x_n)\defeq a$. We sometimes
write~$\cna[1]{a}$ as~$\ca$.
For a set~$B$ and a function $f\colon B\to A$ we denote the image
of~$f$ by $f[B]$.
For a partial order~$\leq$ on~$A$, we write $a<b$ if
$a\leq b$ and $a\neq b$. Moreover, we write $a\prec b$ if $a<b$
and there is no $x\in A$ with $a<x<b$.
For basic notions from universal algebra and lattice theory we refer
to~\cite{BurSan81,McKMcnTay88}.
In particular we will use \emph{lattices} as defined
in~\cite[Chapter~1, p.~16]{McKMcnTay88}.
For a lattice~$\ab{L}$ and $a,b\in L$ such that $a\leq b$, we define
$\interval{a}{b}$ to be the set
$\{l\in L\mid a\leq l \text{ and } l\leq b\}$.
If $a\prec b$, we say that $\interval{a}{b}$ is a
\emph{prime quotient} of~$\ab{L}$.
A \emph{clone} on a set~$A$ is a set of operations on~$A$ that is
closed under composition and contains all the
projections (cf.~\cite[1.1.2, 1.1.3]{PosKal79},
\cite[Section~6.1]{Bod21},
and~\cite[Definition~4.1]{McKMcnTay88}).
A clone on~$A$ is \emph{constantive} if it includes all
operations~$\ca$ where $a\in A$.
For the definition of an \emph{algebra}~$\ab{A}$ on~$A$ we refer
to~\cite[Definition~1.1]{McKMcnTay88}.

We will fix some notation.
For a set of relations~$\relset{R}$ on~$A$, $\Pol\relset{R}$ is the
clone of \emph{polymorphisms} of~$\relset{R}$, and for a set of
functions~$\funset{F}$ on~$A$, $\Inv\funset{F}$ is the relational clone
of \emph{invariant relations} of~$\funset{F}$
(cf.~\cite[Section~E2]{PosKal79}).
For an algebra~$\ab{A}$, $\Clo\ab{A}$ is the clone of \emph{term
operations} of~$\ab{A}$ (cf.~\cite[Definition~4.2]{McKMcnTay88}), while
$\POL\ab{A}$ is the clone of \emph{polynomial operations} of~$\ab{A}$
(cf.~\cite[Definition~4.4]{McKMcnTay88}).
An element~$e$ of $\POL\ari{1}\ab{A}$ is called \emph{idempotent}
if it satisfies $e\circ e=e$.
We write $\Con\ab{A}$ for the congruence lattice of~$\ab{A}$,
and we denote its bottom element by~$\bottom{A}$ and
its top element by~$\uno{A}$.
For an algebra~$\ab{B}$ of the same signature as~$\ab{A}$,
we denote the set of all homomorphisms from~$\ab{B}$ to~$\ab{A}$
by $\Hom(\ab{B},\ab{A})$.
We say that~$\ab{A}$ is \emph{finitely subdirectly irreducible}
if for all $\alpha, \beta\in \Con\ab{A}\setminus\{\bottom{A}\}$
we have $\alpha\cap \beta\neq \bottom{A}$.
Clearly, a finite algebra is
subdirectly irreducible (cf.~\cite[Definition~4.39]{McKMcnTay88}) if
and only if it is finitely subdirectly irreducible and has at least
two elements.
For a subset~$X$
of~$A^2$ we denote by $\Cg{X}$ the congruence generated by~$X$
in~$\ab{A}$ (cf.~\cite[Definition~1.19]{McKMcnTay88}).
For $n\in\N$ and for a subset~$X$ of~$A^n$ we denote by
$\Sg[{\ab{A}^n}]{X}$ the subalgebra of~$\ab{A}^n$ generated by~$X$
(cf.~\cite[Definition~1.8]{McKMcnTay88}).
For a group~$\ab{G}$ and $g\in G$, $\braket{g}$ is the subgroup
of~$\ab{G}$ generated by~$\{g\}$.
For $n\in\N$, $\vec{a},\vec{b}\in A^n$ and $\alpha\in \Con\ab{A}$, we
write $\vec{a}\equiv_\alpha \vec{b}$ if
$(\vec{a}(i),\vec{b}(i))\in\alpha$
for all $i\in\finset{n}$; for $n=1$ and $(a,b)\in \alpha$ we will sometimes
just write $a\mathrel{\alpha}b$.
Given a clone~$\funset{C}$ on~$A$, the symbol~$\funset{C}\arii{n}$
denotes the set of all $n$-ary functions in~$\funset{C}$.
For a congruence~$\theta$ of the algebra $\algop{A}{\funset{C}}$ and a
function $f\in \funset{C}\arii{n}$,
$f_\theta$ is the function from $(A/\theta)^n$ to~$A/\theta$
defined by
$f_\theta(a_1/\theta, \dots, a_n/\theta)= f(a_1,\dots, a_n)/\theta$
for all $a_1,\dots, a_n\in A$
(cf.~\cite[Definition~1.15]{McKMcnTay88}).
We observe that for an algebra~$\ab{A}$ and for $\theta\in \Con\ab{A}$
we have
\begin{equation}\label{eq:the_clone_of_actions_of_function_on_a_quotient_algebra}
\POL(\ab{A}/ \theta)=\{f_{\theta}\mid f\in \POL\ab{A}\}.
\end{equation}

A clone~$\funset{C}$ on a set~$A$ is called a \emph{Ma\v{l}cev clone}
if there exists $d\in\funset{C}\arii{3}$ such that for all $a,b\in A$
the equality $d(a,b,b)=d(b,b,a)=a$ holds.
An algebra~$\ab{A}$
is called a \emph{Ma\v{l}cev algebra} if $\Clo\ab{A}$ is a Ma\v{l}cev
clone.
Moreover, we say that~$\ab{A}$ has a \emph{Ma\v{l}cev
polynomial} if $\POL\ab{A}$ is a Ma\v{l}cev clone.
Note that for a group~$\ab{G}$ the term function~$t$,
defined by $t(x_1,x_2,x_3)=x_1-x_2+x_3$ for all $x_1,x_2,x_3\in G$,
is a Ma\v{l}cev term.
We will often use the following basic facts on Ma\v{l}cev algebras.
\begin{lemma}[{cf.~\cite[Theorem~4.70(iii)]{McKMcnTay88}}]\label{lemma:generating_congruences_in_malcev_with_polynomials}
Let~$\ab{A}$ be an algebra with a Ma\v{l}cev polynomial,
let $k\in\N$ and let $a_1,\dots,a_k,b_1,\dots,b_k\in A$. Then
\[
\Cg{\{(a_1,b_1),\dots ,(a_k, b_k)\}}=\{(p(a_1,\dots,
a_k),p(b_1,\dots, b_k))\mid p\in \POL\ari{k}\ab{A}\}.
\]
\end{lemma}

Later we shall also use the following observation.
\begin{lemma}[{cf.~\cite[Lemma~5.22]{HobbMcK}}]\label{lem:Malcev-weak-con-is-con}
Every reflexive subuniverse of the square of
an algebra~$\ab{A}$ with a Ma\v{l}cev polynomial is a
congruence of~$\ab{A}$.
\end{lemma}

We will use the notions of \emph{centralizing relation} and
\emph{commutator}
as defined in \cite[Section~4.13]{McKMcnTay88}.
To aid the reader we give the definitions explicitly.
Following~\cite{Aic06}, for an algebra~$\ab{A}$, $m,n\in \N$
and $\alpha, \beta,\eta \in\Con\ab{A}$, we say that
$C(m,n,\alpha,\beta,\eta)$ holds if for all $p\in\POL\ari{m+n}\ab{A}$,
for all $\vec{a},\vec{b}\in A^m$, $\vec{u},\vec{v}\in A^n$ with
$\vec{a}\equiv_\alpha \vec{b}$, $\vec{u}\equiv_\beta \vec{v}$ and
$p(\vec{a},\vec{u})\mathrel{\eta} p(\vec{a},\vec{v})$ we have
$p(\vec{b},\vec{u})\mathrel{\eta} p(\vec{b},\vec{v})$.
We say that \emph{$\alpha$ centralizes~$\beta$ modulo~$\eta$}, and
write $C(\alpha,\beta;\eta)$, if $C(1,k,\alpha,\beta,\eta)$ is
satisfied for all $k\in\N$.
Note that this definition of the centralizing relation is proved to be
equivalent to~\cite[Definition~4.148]{McKMcnTay88}
in~\cite[Proposition~2.1]{Aic06}.
Following~\cite[Definition~4.150]{McKMcnTay88}, for
$\alpha,\beta\in\Con\ab{A}$ we
define their \emph{commutator}, denoted by $[\alpha,\beta]$, to be the
smallest congruence~$\eta$ of~$\ab{A}$ for which
$C(\alpha,\beta;\eta)$.
The fact that there is such a smallest congruence is a consequence
of~\cite[Lemma~4.149]{McKMcnTay88}.
Given an algebra~$\ab{A}$ and $\theta\in\Con\ab{A}$,
we say that~$\theta$ is \emph{Abelian}~\cite{FreMck87} if
$[\theta, \theta]=\bottom{A}$, and we say that~$\ab{A}$ is
\emph{Abelian} if
$[\uno{A}, \uno{A}]=\bottom{A}$. Groups are Abelian if and only
if they
are commutative.

An algebra~$\ab{V}$ that has a group reduct is called an
\emph{expanded group}.
If the group reduct is $\ab{G}=\algop{V}{+,-,0}$, we
will say that~$\ab{V}$ is an \emph{expansion} of~$\ab{G}$.
A subset~$I$ of~$V$ is an \emph{ideal} if it is
a normal subgroup of~$\ab{G}$ and for all $n\in\N$,
for each $n$-ary basic operation~$f$ of~$\ab{V}$,
for all $\vec{i}\in I^n$ and for all $\vec{v}\in V^n$ we have
$f(\vec{v}+\vec{i})-f(\vec{v})\in I$.
We denote the lattice of ideals of an expanded group~$\ab{V}$
by~$\Id\ab{V}$.
We remark that the function $\psi\colon \Id\ab{V}\to \Con\ab{V}$
defined by $\psi(I)=\{(a_1,a_2)\mid a_1-a_2\in I\}$ for all
$I\in\Id\ab{V}$,
induces a lattice isomorphism between $\Con\ab{V}$ and $\Id\ab{V}$.
On $\Id\ab{V}$ we define a binary operation, the \emph{ideal
commutator} (cf.~\cite{Sco97}), as follows:
For $A,B\in \Id\ab{V}$ we let $[A,B]$ be the ideal generated by
\[
\{p(a,b)\mid a\in A, b\in B, p\in \POL\ari{2}\ab{V} \text{ and }
\forall v\in V\colon p(v,0)=p(0,v)=0\}.
\]
The lattice $\Id\ab{V}$ expanded with the ideal commutator is
isomorphic, via the isomorphism~$\psi$, to $\Con\ab{V}$
expanded with the commutator
operation defined for congruences
above.
A proof can be found in~\cite[Section~2]{AicMay07} and
in~\cite[Section~4]{AicMud09}.
Thus, for two ideals $M,N$ of an expanded group~$\ab{V}$, their ideal
commutator $[M,N]$ is the ideal $\psi^{-1}([\psi(M),\psi(N)])$.

In~\cite{HobbMcK}, Hobby and McKenzie developed a structure theory
for finite algebras called \emph{tame congruence theory} (TCT).
The central notions of this theory are that of \emph{minimal set}
(cf.~\cite[Definition~2.5]{HobbMcK}),
and that of \emph{minimal algebra}
(cf.~\cite[Definition~2.14]{HobbMcK}).
Each minimal algebra has one of five types (cf.~\cite[Definition~4.10,
Corollary~4.11]{HobbMcK}).
To denote the five TCT-types we will use bold numbers:
$\tp{1},\tp{2},\tp{3},\tp{4},\tp{5}$.
Tame congruence theory associates to each prime quotient of $\Con\ab{A}$
a set of minimal algebras that have the same type.
The type of a prime quotient is then defined as the type of these
minimal algebras
(cf.~\cite[Definition~5.1]{HobbMcK}).
We will denote the type of a prime quotient $\interval{\alpha}{\beta}$
by $\type{\alpha}{\beta}$.

\section{Algebraic consequences of equational additivity}
Let~$A$ be a set and let~$\funset{C}$ be a clone on~$A$.
Following~\cite{Pin15a}, for $n\in \N$ and for $X\subseteq A^n$
we say that~$X$ is \emph{algebraic with respect to~$\funset{C}$},
or that~$X$ is $\funset{C}$-algebraic,
if there exist an index set~$I$ and two families
$(p_i)_{i\in I}$, $(q_i)_{i\in I}$
of operations in~$\funset{C}\arii{n}$ such that
$X=\{\vec{x}\in A^n\mid \forall i\in I\colon
p_i(\vec{x})=q_i(\vec{x})\}$.
We define $\Alg\ari{n}\funset{C}$ to be the
collection of all the subsets of~$A^n$
that are algebraic with respect to~$\funset{C}$,
and we define the \emph{algebraic geometry} of~$\funset{C}$ by
$\Alg\funset{C}\defeq\bigcup_{n\in\N}\Alg\ari{n}\funset{C}$.
For an algebra~$\ab{A}$ we set $\Alg\ab{A}\defeq\Alg\Clo(\ab{A})$
(cf.~\cite{BauMyaRem99, Pin17c}).
We provide a lemma that will be useful to assess whether a set~$X$
is algebraic with respect to a clone~$\funset{C}$.
\begin{lemma}\label{lemma:equivalent_definition_algebraic_set}
Let~$A$ be a set, let~$\funset{C}$ be a clone on~$A$, let $n\in\N$
and let $X\subseteq A^n$. Then $X\in \Alg\ari{n}\funset{C}$
if and only if for all $\vec{a}\in A^n\setminus X$
there exist $f_\vec{a},g_\vec{a}\in\funset{C}\arii{n}$
such that $f_\vec{a}(\vec{a})\neq g_\vec{a}(\vec{a})$ and
$f_\vec{a}(\vec{x})=g_\vec{a}(\vec{x})$ for all $\vec{x}\in X$.
\end{lemma}
\begin{proof}
Let us assume that $X\in\Alg\ari{n}\funset{C}$. Then there exists an
index set~$I$ and
$\{p_i\mid i\in I\},\{q_i\mid i\in I\}\subseteq\funset{C}\arii{n}$
such that
$X=\{\vec{a}\in A^n\mid \forall i\in I \colon
                                      p_i(\vec{a})=q_i(\vec{a})\}$.
Let $\vec{a}\in A^n\setminus X$. Clearly,
there exists $i\in I$ such that $p_i(\vec{a})\neq q_i(\vec{a})$. Thus,
it suffices to set $f_\vec{a}=p_i$ and $g_\vec{a}=q_i$.

Let us assume that for all $\vec{a}\in A^n\setminus X$ there exist
$f_\vec{a},g_\vec{a}\in \funset{C}\arii{n}$ such that
$f_\vec{a}(\vec{a})\neq g_\vec{a}(\vec{a})$ and
$f_\vec{a}(\vec{x})=g_\vec{a}(\vec{x})$ for all $\vec{x}\in X$.
Then we can obtain~$X$ in the form
$X=\{\vec{x}\in A^n\mid \forall \vec{a}\in A^n\setminus X\colon
f_\vec{a}(\vec{x})=g_\vec{a}(\vec{x})\}$.
\end{proof}
We report the definition of equationally additive clone as given
in~\cite{Pin17a}.
\begin{definition}[Equationally additive]
A clone~$\funset{C}$ on a set~$A$ is called
\emph{equationally additive} if for all $n\in\N$ and for all $A,B\in
\Alg\ari{n}\funset{C}$ we have $A\cup B\in \Alg\ari{n}\funset{C}$.
An algebra~$\ab{A}$ is an \emph{equational
domain}~\cite[Definition~1]{DanMyaRem10} if
$\Clo\ab{A}$ is equationally additive.
\end{definition}
To each set~$A$ we associate the following quaternary relation
\[\deltarelation\defeq\{(x_1,x_2,x_3, x_4)\in A^4\mid
                        x_1=x_2 \text{ or } x_3=x_4\}.\]
We observe that $\deltarelation=\pi_4(A)$ as
defined in~\cite[Lemma~1.3.1]{PosKal79}.
Next, we shall state a theorem by Daniyarova, Myasnikov and
Remeslennikov that characterizes equationally additive clones in
terms of their quaternary algebraic sets.
\begin{lemma}[{cf.~\cite[proof of
Theorem~2.5]{DanMyaRem10}}]\label{lemma:how_does_deltarelation_really_works}
Let~$\funset{C}$ be a clone on a set~$A$ and $n\in\N$.
Suppose that~$\deltarelation$ and $B,C\subs A^n$ are algebraic
over~$\funset{C}$, expressed as
\begin{alignat*}{2}
\deltarelation
&=\{\vec{a}\in A^4\mid\forall i\in I\colon&
p_i(\vec{a})&=q_i(\vec{a})\}\\
B&=\{\vec{a}\in A^n\mid \forall j\in J \colon&
f_j(\vec{a})&=g_j(\vec{a})\}\\
C&=\{\vec{a}\in A^n\mid\forall k\in K\colon&
h_k(\vec{a})&=t_k(\vec{a})\}
\end{alignat*}
for some index sets $I, J, K$ and operations
$\{p_i\mid i\in I\}, \{q_i\mid i\in I\}\subseteq  \funset{C}\arii{4}$
$\{f_j\mid j\in J\}, \{g_j\mid j\in J\}, \{h_k\mid k\in K\},
 \{t_k\mid k\in K\}\subseteq \funset{C}\arii{n}$.
Then we have
\begin{align*}
B\cup C=\{&\vec{a}\in A^n\mid \forall (i,j,k)\in I\times J\times K
\colon \\
&p_i(f_j(\vec{a}),g_j(\vec{a}),h_k(\vec{a}),t_k(\vec{a}))=q_i(
f_j(\vec{a}),g_j(\vec{a}),h_k(\vec{a}),t_k(\vec{a}))\}.
\end{align*}
\end{lemma}

\begin{theorem}[{\cite[Theorem~2.5]{DanMyaRem10}}]\label{teor:eq_additive_iff_delta_four_algebraic}
A clone~$\funset{C}$ on a set~$A$ is equationally additive if and
only if
$\deltarelation\in \Alg\ari{4}\funset{C}$.
\end{theorem}
\begin{proof}
If $\deltarelation \in \Alg\ari{4}\funset{C}$, then
Lemma~\ref{lemma:how_does_deltarelation_really_works} yields that the
union of two $\funset{C}$-algebraic sets is always a
$\funset{C}$-algebraic set. If~$\funset{C}$ is equationally additive,
then $\deltarelation\in \Alg\ari{4}\funset{C}$ since it is the
union of
two algebraic sets, namely
\[\deltarelation =\{\vec{a}\in A^4 \mid
a_1=a_2\}\cup \{\vec{a}\in A^4 \mid a_3=a_4\}.\qedhere\]
\end{proof}

\begin{corollary}\label{cor:equational_add_and_inclusion}
Let~$\funset{C}$ and~$\funset{D}$ be clones on a set~$A$ such that
$\funset{C}\subseteq\funset{D}$.
If~$\funset{C}$ is equationally additive, then so is~$\funset{D}$.
\end{corollary}
Hence if~$\ab{A}$ is an equational domain, then not only
is~$\Clo\ab{A}$ equationally additive, but also its extension
$\POL\ab{A}$.

An algebra~$\ab{A}$ is called \emph{essentially at most unary}
if the clone $\Clo\ab{A}$ is generated by its unary part.
We shall now prove that non-trivial equational domains must contain a
function with at least two essential arguments.
\begin{theorem}\label{thm:ess-unary-algs-not-eq-add}
Let~$\ab{A}$ be an essentially at most unary algebra with at least two
elements. Then $\POL\ab{A}$ is not equationally additive.
\end{theorem}
\begin{proof}
The algebra~$\ab{A}$ being essentially at most unary means that its
clone $\Clo\ab{A}$ is generated by its unary part; hence $\POL\ab{A}$
is generated by~$\funset{F}$ defined as the union of
$\Clo\ari{1}\ab{A}$ and all unary constants. That is, for every
$n\in\N$ and $g\in\POL\ari{n}\ab{A}$ there is some $i\in\finset{n}$
and
some $f\in\funset{F}$ such that $g(x_1,\dotsc,x_n)=f(x_i)$ for all
$x_1,\dotsc,x_n\in A$. In order to obtain a contradiction, let us
assume that $\POL\ab{A}$ is equationally additive, which means
that~$\deltarelation$ is the solution set of some system of equations
over $\POL\ari{4}\ab{A}$ (cf.\
Theorem~\ref{teor:eq_additive_iff_delta_four_algebraic}).
Each of the equations is of the form $f(x_i)=g(x_j)$
for some $f,g\in\funset{F}$ and $i,j\in\finset{4}$, and it must
be satisfied by all tuples in~$\deltarelation$.
Let us now consider any particular such equation.\par
As a first case we assume that in this equation $i\neq j$.
For any $a,b\in A$ we can find a tuple
$\vec{x}\in\deltarelation$ such that $x_i=a$ and $x_j=b$. For instance,
if $(i,j)=(1,4)$ we may choose $(a,a,a,b)$, if
$(i,j)=(1,2)$ we may choose $(a,b,b,b)$, etc.
Since $f(x_i)=g(x_j)$ is satisfied by the constructed
$\vec{x}\in\deltarelation$, we obtain $f(a)=g(b)$ for all $a,b\in A$.
This implies that~$f$ and~$g$ are constant with the same value; but
then the equation $f(x_i)=g(x_j)$ is satisfied by all $\vec{x}\in A^4$.
\par
Let us now investigate the case where $i=j$, that is, the considered
equation is of the form $f(x_i)=g(x_i)$ with $i\in\finset{4}$.
Again, for any $a\in A$ we can choose
$\vec{x}=(a,a,a,a)\in\deltarelation$ to show that
$f(a)=g(a)$ holds for all $a\in A$. Thus $f=g$ and the equation
$f(x_i)=g(x_i)$ is again satisfied by all tuples in~$A^4$.\par
As a consequence, all the
equations that were assumed to define~$\deltarelation$ are actually
satisfied by any quadruple in~$A^4$. This, however, means that their
solution set is~$A^4$, which properly contains~$\deltarelation$, due to
$\crd{A}\geq 2$. This contradiction shows that $\POL\ab{A}$ cannot be
equationally additive.
\end{proof}

\begin{lemma}\label{lemma:images_of_the_polynomials_defining_delta_are_alpha_related}
Let~$\funset{C}$ be a clone on~$A$, let $\{p_i\mid
i\in I\},\{q_i\mid i\in I\}\subseteq  \funset{C}\arii{4}$ such that
$\deltarelation=\{\vec{a}\in A^4\mid\forall i\in I\colon
p_i(\vec{a})=q_i(\vec{a}) \}$, and let $\ab{A}=\algop{A}{\funset{C}}$.
Then for all $\alpha\in \Con\ab{A}$, for all $(a_1,a_2)\in \alpha$,
for all $x,y\in A$ and for all $i\in I$, we have
$p_i(a_1,a_2,x,y)\mathrel{\alpha}q_i(a_1,a_2,x,y)$ and
$p_i(x, y, a_1, a_2)\mathrel{\alpha} q_i(x, y, a_1, a_2)$.
\end{lemma}
\begin{proof}
Let $\alpha\in\Con\ab{A}$, let $(a_1,a_2)\in\alpha$, let $x,y\in A$
and let $i\in I$. We have
\[p_i(a_1,a_2,x,y)\mathrel{\alpha}
p_i(a_1,a_1,x,y)=q_i(a_1,a_1,x,y)\mathrel{\alpha} q_i(a_1,a_2,x,y)\]
and
\[p_i(x,y,a_1,a_2)\mathrel{\alpha} p_i(x,y,a_1,a_1)
 =q_i(x,y,a_1,a_1)\mathrel{\alpha} q_i(x,y,a_1,a_2).\qedhere\]
\end{proof}

The next result tells that every equational domain is finitely
subdirectly irreducible.
\begin{proposition}\label{prop:nec_condi_fin_sub_irre}
For any set~$A$ and every equationally additive clone~$\funset{C}$
on~$A$ the algebra $\ab{A}=\algop{A}{\funset{C}}$ is finitely subdirectly
irreducible.
\end{proposition}
\begin{proof}
If~$\funset{C}$ is equationally additive, then there exist an
index set~$I$
and functions $p_i,q_i\in\funset{C}\arii{4}$ for $i\in I$, such that
$\deltarelation=\{\vec{x}\in A^4\mid \forall i\in I\colon
p_i(\vec{x})=q_i(\vec{x})\}$.
Let $\alpha,\beta\in\Con\ab{A}\setminus \{\bottom{A}\}$. We show
that $(\alpha\cap \beta)\setminus \bottom{A}\neq \emptyset$.
Since $\alpha\neq \bottom{A}$ and $\beta\neq \bottom{A}$, there exist
$(a_1,a_2)\in \alpha\setminus \bottom{A}$ and $(b_1,b_2)\in
\beta\setminus \bottom{A}$. Let $i\in I$ be such that
$p_i(a_1,a_2,b_1,b_2)\neq q_i(a_1,a_2,b_1,b_2)$. Since
$(a_1,a_2)\in\alpha$,
Lemma~\ref{lemma:images_of_the_polynomials_defining_delta_are_alpha_related}
yields that
$p_i(a_1,a_2,b_1,b_2)\mathrel{\alpha}q_i(a_1,a_2,b_1,b_2)$;
likewise $(b_1,b_2)\in\beta$ implies
$p_i(a_1,a_2,b_1,b_2)\mathrel{\beta}q_i(a_1,a_2,b_1,b_2)$.
Thus, $(p_i(a_1,a_2,b_1,b_2),q_i(a_1,a_2,b_1,b_2))$ belongs to
$(\alpha\cap \beta)\setminus \bottom{A}$.
\end{proof}

We say that an algebra~$\ab{A}$ has a \emph{weak difference
term}
if there exists $d\in\Clo\ari{3}\ab{A}$ such that for all
$\theta\in\Con\ab{A}$ and all $(a,b)\in \theta$ we have
$d(a,b,b)\mathrel{[\theta, \theta]} a\mathrel{[\theta,\theta]}
d(b,b,a)$.
A \emph{weak difference polynomial} is defined analogously using
$\POL\ab{A}$.
Note that a Ma\v{l}cev polynomial is also a weak difference
polynomial.

\begin{proposition}\label{prop:nec_condi_weak_diff_implies_non_Abelian}
Let~$\ab{A}$ be an algebra with a weak difference polynomial.
If\/~$\POL\ab{A}$ is equationally additive,
then for all $\alpha\in \Con\ab{A}\setminus\{\bottom{A}\}$ we have
$[\alpha, \alpha]>\bottom{A}$.
\end{proposition}
\begin{proof}
If~$\POL\ab{A}$ is equationally additive, then there exist an
index set~$I$
and functions $p_i,q_i\in\POL\ari{4}\ab{A}$ for $i\in I$, such that
\[\deltarelation=\{\vec{x}\in A^4\mid \forall i\in I\colon
p_i(\vec{x})=q_i(\vec{x})\}.\]
Let $\alpha\in\Con\ab{A}\setminus \{\bottom{A}\}$ and let
$(a,b)\in\alpha\setminus \bottom{A}$. As
$(a,b,a,b)\notin\deltarelation$, there exists $i\in I$ such that
$p_i(a,b,a,b)\neq q_i(a,b,a,b)$. Let us define the polynomial
operation~$f$ for all $x_1,x_2\in A$ by
\[f(x_1,x_2)\defeq d(p_i(a,b,a,b),p_i(a,x_1,x_2,b),q_i(a,x_1,x_2,b)).\]
By using the definition of weak difference polynomial and noting that,
due to
Lemma~\ref{lemma:images_of_the_polynomials_defining_delta_are_alpha_related},
we have $p_i(a,b,a,b)\mathrel{\alpha}q_i(a,b,a,b)$, we can verify that
\begin{alignat*}{2}
f(a,a)&=d(p_i(a,b,a,b),p_i(a,a,a,b),q_i(a,a,a,b))&&\mathrel{[\alpha,\alpha]}p_i(a,b,a,b),\\
f(a,b)&=d(p_i(a,b,a,b),p_i(a,a,b,b),q_i(a,a,b,b))&&\mathrel{[\alpha,\alpha]}p_i(a,b,a,b),\\
f(b,a)&=d(p_i(a,b,a,b),p_i(a,b,a,b),q_i(a,b,a,b))&&\mathrel{[\alpha,\alpha]}q_i(a,b,a,b),\\
f(b,b)&=d(p_i(a,b,a,b),p_i(a,b,b,b),q_i(a,b,b,b))&&\mathrel{[\alpha,\alpha]}p_i(a,b,a,b).
\end{alignat*}
Therefore, we have that $f(a,a)\mathrel{[\alpha, \alpha]} f(a,b)$.
Thus, applying the definition of commutator to~$f$ yields that
\[q_i(a,b,a,b)\mathrel{[\alpha,\alpha]}f(b,a)\mathrel{[\alpha,\alpha]}f(b,b)\mathrel{[\alpha,\alpha]}p_i(a,b,a,b).\]
Since $q_i(a,b,a,b)\neq p_i(a,b,a,b)$ we deduce that $[\alpha,
\alpha]>\bottom{A}$.
\end{proof}
We will use the notion of \emph{Taylor operation} on a set~$A$
as defined, e.g., in~\cite[Definition~6.6.1]{Bod21}.
We say that~$\ab{A}$ has a \emph{Taylor term}
(cf.~\cite[Definition~6.6.2]{Bod21}) if $\Clo\ab{A}$ contains a Taylor
operation, and that~$\ab{A}$ has a \emph{Taylor polynomial} if
$\POL\ab{A}$ contains a Taylor operation.
\begin{corollary}
Let~$\ab{A}$ be a finite, at least two-element algebra with an
idempotent Taylor polynomial. If\/~$\POL\ab{A}$ is equationally additive,
then~$\ab{A}$ is subdirectly irreducible and its monolith is
non-Abelian.
\end{corollary}
\begin{proof}
Let $\ab{A}'=\algop{A}{\POL\ab{A}}$. Since~$A$ is finite and has at least
two elements and since $\Con\ab{A}=\Con\ab{A}'$,
Proposition~\ref{prop:nec_condi_fin_sub_irre} yields that~$\ab{A}$
and~$\ab{A}'$ are subdirectly irreducible. Let~$\mu$ be the
monolithic congruence of~$\ab{A}$ and~$\ab{A}'$.
Since $\ab{A}'$ has an idempotent Taylor operation,
it generates a variety satisfying a non-trivial
idempotent Ma\v{l}cev condition.
Hence that variety satisfies condition~(2) of~\cite[Theorem~9.6]{HobbMcK}, and
thus, by the latter theorem, the variety omits type~$\tp{1}$. Therefore, by
\cite[Theorem~7.12]{HobbMcK}, $\ab{A}'$ has a weak difference term.
Consequently, $\ab{A}$ has a weak difference polynomial; and therefore
Proposition~\ref{prop:nec_condi_weak_diff_implies_non_Abelian}
yields that~$\mu$ is non-Abelian.
\end{proof}

Hence, using Corollary~\ref{cor:equational_add_and_inclusion}, it
follows that all finite non-trivial equational domains having a Taylor
polynomial are subdirectly irreducible with a non-Abelian monolith.
\par

Since a Ma\v{l}cev operation is a Taylor operation, we obtain the
following.
\begin{corollary}\label{cor:nec_condi_malcev_polynomial_and_finite_yields_sub_irreducible_and_type_three}
Let~$\ab{A}$ be a finite algebra with at least two elements and a
Ma\v{l}cev polynomial. If\/~$\POL\ab{A}$ is equationally additive,
then~$\ab{A}$ is subdirectly irreducible and its monolith is
non-Abelian.
\end{corollary}
We now focus on those clones on a finite set~$A$ with the property
that~$\deltarelation$ is the solution set of a single equation of
the form $f\approx a$ with $a\in A$.
An example is given by the clone of polynomial functions of a ring
with no zero divisors, where $f(x_1,x_2,x_3,x_4)=(x_1-x_2)(x_3-x_4)$
and $a=0$.
\begin{lemma}\label{lemma:sudoku_problem}
Let~$A$ be a finite set with $\crd{A}\geq 2$, let $0\in A$,
let $f\colon A^4\to A$ be
such that $\deltarelation=\{\vec{x}\in A^4\mid f(\vec{x})=0\}$,
and let $\ab{A}=\algop{A}{f}$.
Then there exists $p\in \POL\ari{1}\ab{A}$ and there exists $i\in
f[A^4]\setminus\{0\}$ such that $p(0)=0$ and $p(x)=i$ for all $x\in
A\setminus \{0\}$.
\end{lemma}
\begin{proof}
We proceed by induction on $\crd{A}\geq 2$.

\textbf{Base step}: $\crd{A}=2$, $A=\set{0,i}$ with $i\neq 0$:
The unary
polynomial~$p$, defined by $x\mapsto f(0,x,0,x)$ for all $x\in A$,
satisfies all the desired properties.
In fact $(0,0,0,0)\in\deltarelation$, hence $p(0)=f(0,0,0,0)=0$, and
$(0,i,0,i)\notin\deltarelation$, hence $p(i)=i$.

\textbf{Induction step}: For each element $a\in A$, let us define
a unary polynomial $p_a\in \POL\ari{1}\ab{A}$ by $p_a(x)=f(0,x,0,a)$
for
all $x\in A$.
Note that, if $a\in A\setminus\set{0}$, then $p_a$ preserves
$A\setminus\set{0}$ as a subuniverse since
$(0,x,0,a)\notin\deltarelation$ for all
$x\in A\setminus\set{0}$. Moreover, we have $p_a(0)=0$.
We split the induction step into two cases.

\textbf{Case 1}: \textit{For all $a\in A\setminus\{0\}$ the
                         function~$p_a$ induces a permutation on~$A$.}
Set $m\defeq\crd{A}!$ and consider any $a\in A\setminus\set{0}$.
The order of~$p_a$ in the full symmetric group on~$A$
divides~$m$, hence $p_a^m(x)=x$ for all $x\in A$.
Since $(0,x,0,0)\in\deltarelation$
the $m$-th iterated power
of~$p_0$ is still the constant zero function of arity one.
Therefore, for all $x\in A$, given $a\neq 0$, we have $p_a^m(x) = x$,
while $p_a^m(x)=0$ if $a=0$.
We now pick an arbitrary element $i\in f[A^4]\setminus\set{0}$ (this
is possible since there is some $a\in A\setminus\set{0}$, for
which~$p_a$ is a permutation) and define $p(x)\defeq p_x^m(i)$ for all
$x\in A$. Clearly, if $x\neq 0$, then $p(x)=p_x^m(i)=i$, and
$p(0)=p_0^m(i)=0$.
Moreover, $p\in\POL\ari{1}\ab{A}$ because it is constructed as an
iterated substitution of~$f$ within itself wherein some positions have
been filled by constant values.

\textbf{Case~2}: \textit{There is $a\in A\setminus\{0\}$ where the
function~$p_a$ is not a permutation of~$A$.}
Let $m\in\N$ be such that $e\defeq p_a^m\in\POL\ari{1}\ab{A}$ is
idempotent, i.e., $e\circ e=e$. Let $B\defeq e[A]$ be its image,
which contains~$0$ since $p_a(0)=0$.
Since~$p_a$ preserves~$\set{0}$
and $A\setminus\set{0}$, so does~$e$, and hence we have
\begin{equation}\label{eq:when_is_p_prime_equal_zero}
\forall x\in A\colon\quad e(x)=0\iff x=0.
\end{equation}
Moreover, we have $B=e[A]\subseteq p_a[A]\subsetneq A$ since~$p_a$
is not surjective; hence
the algebra $\ab{B}=\algop{B}{(e\circ f)\restrict{B}}$ is defined on a
set with smaller cardinality than~$A$.
Given $n\in\N$, a straightforward induction on the polynomial terms
describing $\POL\ari{n}\ab{A}$ shows that
\begin{equation}\label{eq:POLB_0subseteqPOLA_restrict}
\forall n\in\N \, \forall g\in \POL\ari{n}\ab{B}\, \exists \hat{g}\in
\POL\ari{n}\ab{A}\, \forall \vec{b}\in B^n\colon\quad
g(\vec{b})=\hat{g}(\vec{b}).
\end{equation}
Because of~\eqref{eq:when_is_p_prime_equal_zero},
for all $\vec{x}\in A^4$ we have
\[(e\circ f)(\vec{x})=0\iff f(\vec{x})=0\iff
\vec{x}\in\deltarelation,\]
which implies that
\[\deltarelation[B] = B^4\cap\deltarelation
=\Set{\vec{b}\in B^4| \vec{b}\in\deltarelation}
=\Set{\vec{b}\in B^4| (e\circ f)\restrict{B}(\vec{b})=0}.\]
This means that the induction hypothesis can be applied to~$\ab{B}$,
as $\crd{B}<\crd{A}$.
Thus, there exists
$q\in \POL\ari{1}\ab{B}$ and there is
$i\in B\setminus\set{0}\subseteq
p_a[A]\setminus\set{0}\subseteq f[A^4]\setminus\{0\}$ such that
$q(0)=0$ and $q(b)=i$ for all
$b\in B\setminus\set{0}$. Moreover,
\eqref{eq:POLB_0subseteqPOLA_restrict} yields that there
exists $\hat{q}\in\POL\ari{1}\ab{A}$ such that $q(b)=\hat{q}(b)$
for all $b\in B$.
Let us define $p\defeq \hat{q}\circ e\in \POL\ari{1}\ab{A}$.
Then~\eqref{eq:when_is_p_prime_equal_zero} yields
$p(0)=q(e(0))=q(0)=0$. Moreover, for all $a\in A\setminus\{0\}$ we
have by~\eqref{eq:when_is_p_prime_equal_zero} that $e(a)\in
B\setminus\set{0}$, and therefore $p(a)=\hat{q}(e(a))=q(e(a))=i$.
This concludes the proof.
\end{proof}

\begin{proposition}\label{prop:the_TCT_type_of_monolith_when_delta_is_f_equal_constant}
Let~$A$ be a finite set with at least two elements, let~$\funset{C}$
be
a clone on~$A$, let $f\in\funset{C}\arii{4}$, let $0\in A$ be
such that
$\deltarelation=\{\vec{a}\in A^4\mid f(\vec{a})=0\}$, and let
$\ab{A}=\algop{A}{\funset{C}}$.
Then~$\ab{A}$ is subdirectly irreducible, there exists $i\in
f[A^4]\setminus\{0\}$ such that $\mu=\Cg{\{(0,i)\}}$ is
the monolith of~$\ab{A}$,
and $\type{\bottom{A}}{\mu}=\tp{3}$.
\end{proposition}
\begin{proof}
Lemma~\ref{lemma:sudoku_problem} yields that there exists $i\in
f[A^4]\setminus\{0\}$ and there exists $p\in\POL\ari{1}\ab{A}$
such that $p(0)=0$ and $p(a)=i$ for all $a\in A\setminus\{0\}$.
Take any $\theta\in\Con\ab{A}\setminus\{\bottom{A}\}$
and $(a,b)\in\theta\setminus\{\bottom{A}\}$. We show that
$(0,i)\in\theta$.
Let $h\colon A^4\to A$ be defined by $h(\vec{x})=p(f(\vec{x}))$
for all
$\vec{x}\in A^4$. Clearly, $h\in \POL\ab{A}$. Moreover,
$(0,i,a,b)\equiv_\theta (0,i,a,a)$. Thus, we have
$0=h(0,i,a,a)\mathrel{\theta} h(0,i,a,b)=i$, and therefore $(0,i)\in
\theta$. Hence~$\ab{A}$ is subdirectly irreducible and the monolith is
$\Cg{\{(0,i)\}}$. Since~$p$ is idempotent and has image
$\{0,i\}$, the set $\{0,i\}$ is minimal in the sense of tame
congruence theory (cf.~\cite[Definition~2.5]{HobbMcK}).

Next, we define $c\colon A\to A$ by letting $c(x)=p(f(x,i,x,i))$ for
all $x\in A$, and we introduce $m\colon A^2\to A$ by
$m(x_1,x_2)=p(f(p(f(x_1, i,x_1,i)), i, x_2, 0))$ for all
$x_1,x_2\in A$. Clearly, $c\in\POL\ari{1}\ab{A}$ and
$m\in\POL\ari{2}\ab{A}$. Moreover, we have
\begin{align*}
(0,i,0,i),(0,i,i,0)&\notin \deltarelation,\\
(i,i,i,i),(i,i,i,0),(0,i,0,0),(i,i,0,0)&\in \deltarelation.
\end{align*}
Therefore, we have
\begin{align*}
f(i,i,i,i)&=f(i,i,i,0)=f(0,i,0,0)=f(i,i,0,0)=0,\\
f(0,i,0,i)&\neq 0,\\
f(0,i,i,0)&\neq 0.
\end{align*}
Thus, we have
\begin{align*}
p(f (i,i,i,i))&=p(f (i,i,i,0))=p(f(0,i,0,0))=p(f(i,i,0,0))=0,\\
p(f(0,i,0,i) )&=p(f(0,i,i,0) )=i.
\end{align*}
Hence $c(0)=p(f(0,i,0,i))=i$ and $c(i)=p(f(i,i,i,i))=0$. Consequently,
$c\in\POL(\ab{A}\restrict{\{0,i\}})$, and~$c$ acts as a complement on
$\{0,i\}$. Moreover, $m$ satisfies
\begin{alignat*}{3}
m(0,0)&=p(f(p(f(0,i,0,i)),i,0,0))&&=p(f(i,i,0,0))&&=0,\\
m(i,0)&=p(f(p(f(i,i,i,i)),i,0,0))&&=p(f(0,i,0,0))&&=0,\\
m(0,i)&=p(f(p(f(0,i,0,i)),i,i,0))&&=p(f(i,i,i,0))&&=0,\\
m(i,i)&=p(f(p(f(i,i,i,i)),i,i,0))&&=p(f(0,i,i,0))&&=i.
\end{alignat*}
Consequently, $m\in\POL(\ab{A}\restrict{\{0,i\}})$ and it acts as
a meet on $\{0,i\}$.
Hence $\ab{A}\restrict{\{0,i\}}$ is polynomially
equivalent to a two-element Boolean algebra, and therefore
$\type{\bottom{A}}{\mu}=\tp{3}$.
\end{proof}

For the subsequent three results, the following notation to extend an
algebra $\ab{A}=\algop{A}{\funset{F}}$ by a single operation
$f\colon A^k\to A$, $k\in\N$, comes handy. We define $\ab{A}+f$ as an
abbreviation of the algebra ${\algop{A}{\funset{F}\cup \{f\}}}$.
\begin{lemma}\label{lemma:Erhardex13a}
Let~$\ab{A}$ be a finite algebra, let $a,b\in A$ with $a\neq b$,
and let $\alpha=\Cg{\{(a,b)\}}$. Then there exists
$f\colon A^4\to A$ such that
$(\ab{A}+f)+\ca$ is a subdirectly irreducible equational domain
with monolith~$\alpha$,
$\type{\bottom{A}}{\alpha}=\tp{3}$ in $(\ab{A}+f)+\ca$,
and $((\ab{A}+f)+\ca)/\alpha=(\ab{A}/\alpha +\cna[4]{a/\alpha})+\ca[a/\alpha]$.
\end{lemma}
\begin{proof}
Let $f\colon A^4 \to A$ be defined for all
$\vec{x}\in A^4$ by $f(\vec{x})=a$ if $\vec{x}\in \deltarelation$ and
$f(\vec{x})=b$ otherwise.
Proposition~\ref{prop:the_TCT_type_of_monolith_when_delta_is_f_equal_constant}
yields that $\Clo((\ab{A}+f)+\ca)$ is equationally additive,
$(\ab{A}+f)+\ca$ is subdirectly irreducible
with monolith $\nu\defeq\Cg[(\ab{A}+f)+\ca]{\{(a,b)\}}$, and
$\type{\bottom{A}}{\nu}=\tp{3}$.
Since the image of~$f$ is a subset of~$a/\alpha$, the
equivalence relation~$\alpha$ is preserved by~$f$ and~$\ca$; thus we
have $\alpha\in \Con((\ab{A}+f)+\ca)$.
As $\nu\in\Con((\ab{A}+f)+\ca)\subs\Con\ab{A}$ and $(a,b)\in\nu$, we
have $\alpha\subseteq \nu$, and since~$\nu$ is the
monolithic congruence of $(\ab{A}+f)+\ca$, we infer that $\nu=\alpha$.
The final equality of the lemma follows from
$f_\alpha= \cna[4]{a/\alpha}$.
\end{proof}

We say that an algebra~$\ab{A}$ is \emph{weakly isomorphic}
to an algebra~$\ab{C}$ if there exists an algebra~$\ab{B}$ with the
same universe as~$\ab{A}$ such that
$\Clo\ab{A}=\Clo\ab{B}$ and $\ab{B}\cong \ab{C}$.

\begin{theorem}\label{teor:Erhardex13b}
Let~$\ab{A}$ be a finite algebra with at least two elements.
Then there exists a subdirectly irreducible
finite equational domain~$\ab{B}$ with monolith~$\mu_\ab{B}$
and an algebra~$\ab{C}$ such that
$\type{\bottom{B}}{\mu_\ab{B}}=\tp{3}$,
$\ab{A}$ is weakly isomorphic to~$\ab{C}$, and~$\ab{C}$
is polynomially equivalent to~$\ab{B}/\mu_\ab{B}$.
\end{theorem}
\begin{proof}
Let~$\ab{D}$ be an algebra on the same universe as~$\ab{A}$ with at
least one at least binary functional symbol in its type, such that
$\Clo\ab{A}=\Clo\ab{D}$. For example, we may take
$\ab{D}=\ab{A}+\eni[2]{1}$, adding the binary projection onto the
first argument to~$\ab{A}$.
Then~\cite[Theorem~3.1]{JezKep02} yields that there exists a finite
subdirectly irreducible algebra~$\ab{E}$ with monolith~$\mu_\ab{E}$
such that $\ab{D}\cong \ab{E}/\mu_\ab{E}\eqdef\ab{C}$, i.e., $\ab{A}$
is weakly isomorphic to~$\ab{C}$.

Let $a,b\in E$ such that $\mu_\ab{E}=\Cg[\ab{E}]{\{(a,b)\}}$.
Then Lemma~\ref{lemma:Erhardex13a} states that there exists
$f\colon E^4\to E$ such that $\ab{B}\defeq (\ab{E}+f)+\ca$ is a finite
subdirectly irreducible equational domain with monolith
$\mu_{\ab{B}}=\mu_\ab{E}$,
$\type{\bottom{E}}{\mu_\ab{B}}=\tp{3}$ in~$\ab{B}$, and
$\ab{B}/\mu_{\ab{B}}
 =((\ab{E}+f)+\ca)/\mu_\ab{E}
 =(\ab{E}/\mu_\ab{E}+\cna[4]{a/\mu_\ab{E}})+\ca[a/\mu_\ab{E}]$.

Then~$\ab{A}$ is weakly isomorphic to~$\ab{C}$,
and~$\ab{C}$ is polynomially equivalent to the quotient
$\ab{B}/\mu_\ab{B}$.
\end{proof}
Theorem~\ref{teor:Erhardex13b} can be improved if we assume
that~$\ab{A}$ generates a congruence modular variety.
For the basic properties of modular lattices and
congruence modular varieties
we refer the reader to~\cite[Section~2.3]{McKMcnTay88}.
\begin{theorem}\label{teor:Erhardex13bcongruence_modular}
Let~$\ab{A}$ be a finite at least two-element algebra
in a congruence modular variety
and let $a\in A$.
Then there exist an algebra~$\ab{B}$ with universe~$B$
in the variety generated by~$\ab{A}$,
$b\in B$, and $f\colon B^4\to B$ such that
$(\ab{B}+f)+\ca[b]$ is a subdirectly irreducible equational domain with
monolith~$\alpha$,
$\type{\bottom{B}}{\alpha}=\tp{3}$ in $(\ab{B}+f)+\ca[b]$, and
$((\ab{B}+f)+\ca[b])/\alpha\cong(\ab{A}+\cna[4]{a})+\ca$.
\end{theorem}
\begin{proof}
Let $\ab{B}\defeq\ab{A}\times \ab{S}$
where~$\ab{S}$ is a simple quotient of~$\ab{A}$ with at least two
elements.
Then~$\ab{B}$ belongs to the variety generated by~$\ab{A}$,
and thus $\Con\ab{B}$ is a modular lattice.
Since~$\ab{S}$ is a simple quotient of~$\ab{A}$ and $\crd{S}\geq 2$,
there are $s_1, s_2\in S$ such
that $s_1\neq s_2$;
we define $\vec{a}_1=(a, s_1)$ and $\vec{a}_2=(a, s_2)$, and we set
$\alpha=\Cg[\ab{B}]{\{(\vec{a}_1, \vec{a}_2)\}}$.
Let~$\Pi_1$ be the canonical homomorphism from~$\ab{B}$ onto~$\ab{A}$,
and let~$\Pi_2$ be the canonical homomorphism of~$\ab{B}$ onto~$\ab{S}$.
Since $(\vec{a}_1, \vec{a}_2)\in\ker \Pi_1$, we have
$\ker\Pi_1\supseteq \alpha$.
Moreover, since $\ab{B}/\ker \Pi_2\cong \ab{S}$ and~$\ab{S}$ is simple
with more than one element, $\ker\Pi_2$ is a co-atom in $\Con\ab{B}$.
Since $\interval{\bottom{B}}{\ker\Pi_1}$ and $\interval{\ker\Pi_2}{\uno{B}}$
are transposes, and thus projective, and $\Con\ab{B}$ is modular, we
infer that these intervals are
isomorphic~\cite[Corollary~2.28]{McKMcnTay88}.
Therefore~$\ker \Pi_1$ is an atom of $\Con\ab{B}$,
whence we conclude that $\alpha=\ker \Pi_1$;
accordingly, we have $\ab{A}\cong \ab{B}/\ker\Pi_1=\ab{B}/\alpha$.
\par

Next, Lemma~\ref{lemma:Erhardex13a} yields that
there exists $f\colon (A\times S)^4\to A\times S$ such that
$(\ab{B}+f)+\ca[\vec{a}_1]$ is a subdirectly irreducible
equational domain with monolith~$\alpha$,
$\type{\bottom{B}}{\alpha}=\tp{3}$ in $(\ab{B}+f)+\ca[\vec{a}_1]$, and
\[((\ab{B}+f)+\ca[\vec{a}_1])/\alpha=(\ab{B}/\alpha
+\cna[4]{\vec{a}_1/\alpha})+\ca[\vec{a}_1/\alpha]
\cong (\ab{A}+\cna[4]{a})+\ca.\qedhere\]
\end{proof}

\section{Characterization of equationally additive constantive Ma\v{l}cev clones}\label{section:the_main_result}
In this section we provide a characterization of equationally
additive constantive Ma\v{l}cev clones in terms of properties of
the term condition commutator
(cf.~Theorem~\ref{teor:equ_additive_cofinite_non_abelianity}).
We start by stating a few well-known properties of the commutator for
algebras with a Ma\v{l}cev polynomial.
\begin{lemma}[{cf.~\cite[Propositions~2.3 and~2.4]{Aic06}}]\label{lemma:commutator-properties-Malcev-algebras}
Let~$\ab{A}$ be an algebra with a Ma\v{l}cev polynomial and
let $\alpha, \beta, \alpha', \beta'\in\Con\ab{A}$ satisfy
$\alpha\leq \alpha'$ and $\beta\leq \beta'$. Then we have
\begin{enumerate}[\upshape(a)]
\item\label{item:char-commutator-in-Malcev-algs}
  $C(1,1,\alpha,\beta,\eta) \iff C(\alpha,\beta;\eta) \iff
   [\alpha,\beta]\leq\eta$,
   \par\noindent
  in particular, the commutator is completely determined by the binary
  polynomials of~$\ab{A}$;
\item\label{item:commutator-below-meet}
  $[\alpha,\beta]\leq \alpha\wedge \beta$;
\item\label{item:commutator-monotone}
  $[\alpha,\beta]\leq [\alpha',\beta]\leq [\alpha',\beta']$.
\end{enumerate}
\end{lemma}
\begin{proof}
The first equivalence of
statement~\eqref{item:char-commutator-in-Malcev-algs} is shown
in~\cite[Proposition~2.3]{Aic06}, the second one
in~\cite[Proposition~2.4]{Aic06}.
Statements~\eqref{item:commutator-below-meet}
and~\eqref{item:commutator-monotone} are obvious consequences
of the definition of $[\alpha,\beta]$ that hold for every
algebra~$\ab{A}$.
\end{proof}

\begin{lemma}[{cf.~\cite[Proposition~2.6]{Aic06}}]\label{lemma:d_affine}
Let $k\in\N$, let~$\ab{A}$ be an algebra with a Ma\v{l}cev
polynomial~$d$, let $\alpha,\beta\in\Con\ab{A}$, and let
$p\in\POL\ari{k}\ab{A}$. For all $\vec{u},\vec{v},\vec{w}\in A^k$ such
that $\vec{u}\equiv_\alpha\vec{v}\equiv_\beta\vec{w}$, we have
\[
d(p(\vec{u}),p(\vec{v}),p(\vec{w}))
\equiv_{[\alpha,\beta]}
p(d(u_1,v_1,w_1),d(u_2,v_2,w_2),\dotsc,d(u_k,v_k,w_k)).
\]
\end{lemma}
Following~\cite{AicMud10}, for $p\in\POL\ari{2}\ab{A}$ and $u_1,u_2\in
A$,
we say that~$p$ is \emph{absorbing at $(u_1,u_2)$} if for all
$x_1, x_2\in A$ we have $p(x_1, u_2)=p(u_1, x_2)=p(u_1,u_2)$.
\begin{lemma}[{cf.~\cite[Lemma~6.13]{AicMud10}}]\label{lemma:char_comm_in_malcev}
Let~$\ab{A}$ be an algebra with a Ma\v{l}cev
polynomial~$d$, let $\alpha=\Cg{\{(u_1,v_1)\}}$ and let
$\beta=\Cg{\{(u_2,v_2)\}}$. Then
\begin{equation}
[\alpha,\beta]=\{(z(v_1,v_2),z(u_1,u_2))\mid
z\in\POL\ari{2}\ab{A}\text{ is absorbing at } (u_1,u_2)\}.
\label{eq:malcev-commutators-bin-absorbing-poly}
\end{equation}
\end{lemma}
\begin{proof}
Let~$\eta$ denote the right-hand side
of~\eqref{eq:malcev-commutators-bin-absorbing-poly};
we first prove that this set is a congruence. Since constant functions
are absorbing at $(u_1, u_2)$, the relation~$\eta$ is reflexive. Let $n\in\N$ and
let~$f$ be an $n$-ary basic operation of~$\ab{A}$. If $z_1,\dots,z_n$
are binary polynomials absorbing at $(u_1,u_2)$, then
$f(z_1,\dotsc, z_n)$ is a binary polynomial absorbing at $(u_1, u_2)$.
Thus~$\eta$ is a subalgebra of $\ab{A}\times\ab{A}$.
Hence Lemma~\ref{lem:Malcev-weak-con-is-con} yields that~$\eta$ is a
congruence of~$\ab{A}$.

Next, we prove that $C(\alpha, \beta; \eta)$. For this, according
to
Lemma~\ref{lemma:commutator-properties-Malcev-algebras}\eqref{item:char-commutator-in-Malcev-algs},
let us take an
arbitrary $q\in\POL\ari{2}\ab{A}$ and any $a,b,u,v\in A$ with
$a\mathrel{\alpha} b$ and $u\mathrel{\beta}v$. We assume
that $q(a,u)\mathrel{\eta} q(a, v)$ and want to show
$q(b, u)\mathrel{\eta} q(b,v)$. Since~$\alpha$ and~$\beta$ are
generated by a single pair,
Lemma~\ref{lemma:generating_congruences_in_malcev_with_polynomials}
yields
that there are unary polynomials $p_1,p_2\in\POL\ari{1}\ab{A}$ such
that $a=p_1(u_1)$, $b=p_1(v_1)$, $u=p_2(u_2)$, $v=p_2(v_2)$.
Setting $p(x,y)\defeq q(p_1(x),p_2(y))$ for $x,y\in A$, we have
$p(u_1,u_2)\mathrel{\eta}p(u_1,v_2)$.
Let us define $f\colon A^2\to A$ by
\begin{multline*}
f(x_1, x_2)\defeq\\
 d(d(p(x_1,x_2), p(x_1,u_2), p(v_1,u_2)),
   d(p(u_1,x_2), p(u_1,u_2), p(v_1,u_2)),
   p(v_1, u_2))
\end{multline*}
for all $x_1,x_2\in A$.
Clearly, $f\in\POL\ari{2}\ab{A}$, for $p\in\POL\ari{2}\ab{A}$.
For arbitrary $a_1,a_2\in A$ we have
\begin{multline*}
f(u_1, a_2)=\\
 d(d(p(u_1,a_2), p(u_1,u_2), p(v_1,u_2)),
   d(p(u_1,a_2), p(u_1,u_2), p(v_1,u_2)),
   p(v_1, u_2))\\
=p(v_1,u_2)
\end{multline*}
and
\begin{multline*}
f(a_1, u_2)=\\
 d(d(p(a_1,u_2), p(a_1,u_2), p(v_1,u_2)),
   d(p(u_1,u_2), p(u_1,u_2), p(v_1,u_2)),
   p(v_1, u_2))\\
=d(p(v_1,u_2),p(v_1,u_2),p(v_1,u_2))=p(v_1,u_2).
\end{multline*}
Hence~$f$ is absorbing at $(u_1,u_2)$ with value $p(v_1, u_2)$.
Therefore, we have
\begin{equation}\label{eq:equazione_seconda_dimostrazione_lemma_char_comm_in_malcev_fvvSfuu}
f(v_1, v_2)\mathrel{\eta} f(u_1, u_2)=p(v_1, u_2).
\end{equation}
Moreover, since $p(u_1, u_2)\mathrel{\eta} p(u_1, v_2)$, we have
\begin{align}\label{eq:equazione_seconda_dimostrazione_lemma_char_comm_in_malcev_fvvSpvv}
f(v_1,v_2)&=d(p(v_1,v_2), d(p(u_1,v_2), p(u_1, u_2), p(v_1, u_2)),
p(v_1, u_2))\notag\\
&\mathrel{\eta} d(p(v_1,v_2), d(p(u_1,v_2), p(u_1, v_2), p(v_1,
u_2)), p(v_1, u_2))\\
&=d(p(v_1, v_2), p(v_1, u_2), p(v_1, u_2))=p(v_1, v_2).\notag
\end{align}
Putting together~\eqref{eq:equazione_seconda_dimostrazione_lemma_char_comm_in_malcev_fvvSfuu}
and~\eqref{eq:equazione_seconda_dimostrazione_lemma_char_comm_in_malcev_fvvSpvv},
we obtain
\[
q(b,u)=p(v_1, u_2)=f(u_1, u_2)\mathrel{\eta} f(v_1, v_2)\mathrel{\eta}
p(v_1, v_2)=q(b,v).
\]
Thus, we have that $C(\alpha, \beta;\eta)$ and hence
$[\alpha,\beta]\subseteq\eta$.

For the converse inclusion let $\gamma\defeq [\alpha,\beta]$ and
$(a,b)\in\eta$. Thus, there is $c\in\POL\ari{2}\ab{A}$ that is
absorbing at $(u_1,u_2)$ such that $a=c(v_1,v_2)$, $b=c(u_1,u_2)$.
We have $u_1\mathrel{\alpha} v_1$ and $u_2\mathrel{\beta} v_2$;
moreover, $c(u_1,u_2)=c(u_1, v_2)$ by the absorption property at
$(u_1,u_2)$, hence $c(u_1, u_2)\mathrel{\gamma}c(u_1,v_2)$.
Since, by the definition of the commutator, $\alpha$
centralizes~$\beta$
modulo~$\gamma$ and~$c$ absorbs at $(u_1,u_2)$, it follows that
$b= c(u_1,u_2)=c(v_1,u_2)\mathrel{\gamma}c(v_1,v_2)=a$, i.e.,
$(b,a)\in\gamma$ and hence $(a,b)\in\gamma$. This concludes the proof
that $\eta\subseteq\gamma=[\alpha,\beta]$.
\end{proof}

\begin{proposition}\label{prop:interpolating_polynomials_ins_sub_irr_type_three_malcev}
Let~$\ab{A}$ be a subdirectly irreducible algebra with a non-Abelian
monolith $\mu\in\Con\ab{A}$, let~$d$ be a Ma\v{l}cev polynomial,
let $o\in A$, let $U=o/\mu$, let $k\in \N$, let $D\subseteq A^k$
and let $l\colon D\to U$. Then, for all $T\subseteq D$ finite,
there exists a polynomial $p_T\in \POL\ari{k}\ab{A}$ such that
$p_T(\vec{t})=l(\vec{t})$ for all $\vec{t}\in T$, and $p_T(\vec{x})\in
U$ for all $\vec{x}\in A^k$.
\end{proposition}
\begin{proof}
Let $T\subseteq D$ be finite. We prove that there exists
$p_T\in\POL\ari{k}\ab{A}$ such that $p_T(\vec{t})=l(\vec{t})$ for all
$\vec{t}\in T$, and $p_T(\vec{x})\in U$ for all $\vec{x}\in A^k$. We
proceed by induction on the cardinality of $T=\{\vec{t}_1,\dots,
\vec{t}_n\}$.

\textbf{Case} $\crd{T}\leq1$:
If $\crd{T}=1$, the constant polynomial~$p_T$ with
value~$l(\vec{t}_1)$ interpolates~$l$ at~$\vec{t}_1$. If $\crd{T}=0$,
any constant polynomial~$p_T$ with value in~$U$, e.g., $o\in U$, will
satisfy the required conditions.

\textbf{Case} $\crd{T}=2$:
If $l(\vec{t}_1)=l(\vec{t}_2)$, a constant polynomial with
value~$l(\vec{t}_1)$ interpolates~$l$ on~$T$. Let
us now assume that $l(\vec{t}_1)\neq l(\vec{t}_2)$; this implies that
$\crd{U}\geq 2$.
Let $l(\vec{t}_1)=f$ and $l(\vec{t}_2)=g$. Since~$\mu$ is not Abelian,
Lemma~\ref{lemma:commutator-properties-Malcev-algebras}\eqref{item:char-commutator-in-Malcev-algs}
implies that there exist
$a,b,u,v\in A$ and $t\in \POL\ari{2}\ab{A}$ such that $a\,\mu\,b$,
$u\,\mu\, v$, $t(a,u)=t(a,v)$ and $t(b,u)\neq t(b,v)$. Moreover, since
$\vec{t}_1\neq \vec{t}_2$, there is $j\in\finset{k}$ such that
$\vec{t}_1(j)\neq \vec{t}_2(j)$, whence
\[
(u,v)\in\mu\subseteq
\Cg{\{(\vec{t}_1(1),\vec{t}_2(1)),\dots,
(\vec{t}_1(k),\vec{t}_2(k))\}},
\]
for~$\mu$ is the monolith of~$\ab{A}$.
Thus, by
Lemma~\ref{lemma:generating_congruences_in_malcev_with_polynomials},
there is $h\in\POL\ari{k}\ab{A}$ such that
$h(\vec{t}_1)=u$ and $h(\vec{t}_2)=v$.
Since
$(f,g)\in U^2\subseteq\mu\subseteq
\Cg{\{(t(b,u),t(b,v))\}}$,
Lemma~\ref{lemma:generating_congruences_in_malcev_with_polynomials}
yields a $p\in \POL\ari{1}\ab{A}$ such that $p(t(b,u))=f$
and $p(t(b,v))=g$.
Let us define the $k$-ary polynomial $p_T\colon A^k\to A$ by
\[p_T(\vec{z})\defeq
p(d(t(b,h(\vec{z})),t(a,h(\vec{z})),t(a,u)))\]
for all $\vec{z}\in A^k$.
Then for any $\vec{x}\in A^k$ such that $h(\vec{x})=u$, we have
\begin{equation}\label{eq:pT-is-lt1-on-hx=u}
p_T(\vec{x})=p(d(t(b,u),t(a,u),t(a,u)))=p(t(b,u))=f=l(\vec{t}_1).
\end{equation}
In particular, this holds for $\vec{x}=\vec{t}_1$.
Moreover, since $t(a,u)=t(a,v)$, we have
\[
p_T(\vec{t}_2)=p(d(t(b,v),t(a,v),t(a,u)))=p(t(b,v))=g=l(\vec{t}_2).
\]
For all $\vec{x}\in A^k$,
let $\alpha_\vec{x}\defeq \Cg{\{(u, h(\vec{x}))\}}$.
We have that for all $\vec{x}\in A^k$
\[(b, h(\vec{x}))\equiv_\mu (a, h(\vec{x}))\equiv_{\alpha_{\vec{x}}}
(a,u).\]
Thus, Lemma~\ref{lemma:d_affine} implies that for all $\vec{x}\in A^k$
\begin{equation}\label{eq:equations_affine_modulo_commutators_in_interpolating_proposition}
d(t(b, h(\vec{x})),t(a, h(\vec{x})),t(a,u))
\mathrel{[\mu, \alpha_{\vec{x}}]}
t(d(b,a,a), d(h(\vec{x}),h(\vec{x}),u))=t(b,u).
\end{equation}
Since~$\mu$ is not Abelian and
$\mu \subseteq \Cg{\{(u,h(\vec{x}))\}}=\alpha_{\vec{x}}$ for
all $\vec{x}\in A^k$ such that $h(\vec{x})\neq u$,
statements~\eqref{item:commutator-below-meet}
and~\eqref{item:commutator-monotone} of
Lemma~\ref{lemma:commutator-properties-Malcev-algebras} yield that
for such $\vec{x}\in A^k$
\[
\mu=[\mu, \mu]\leq [\mu,\alpha_{\vec{x}}]\leq
\mu\wedge\alpha_{\vec{x}}=\mu.
\]
Therefore, for all $\vec{x}\in A^k$ satisfying $h(\vec{x})\neq u$,
we have $[\mu,\alpha_\vec{x}]=\mu$. This together with
condition~\eqref{eq:equations_affine_modulo_commutators_in_interpolating_proposition}
yields that for all such $\vec{x}\in A^k$
\[
d(t(b, h(\vec{x})),t(a, h(\vec{x})),t(a,u))\mathrel{\mu} t(d(b,a,a),
d(h(\vec{x}),h(\vec{x}),u)).
\]
Thus, for all $\vec{x}\in A^k$ with $h(\vec{x})\neq u$ we have
\begin{align*}
p_T(\vec{x})&=p(d(t(b,h(\vec{x})), t(a, h(\vec{x})), t(a,u)))\\
&\equiv_\mu  p(t(d(b,a,a),d(h(\vec{x}),h(\vec{x}),u)))\\
&=p(t(b,u))=f=l(\vec{t}_1)\equiv_\mu o,
\end{align*}
while for $\vec{x}\in A^k$ with $h(\vec{x})=u$ we directly have
$p_T(\vec{x})=l(\vec{t}_1)\equiv_\mu o$ by
equation~\eqref{eq:pT-is-lt1-on-hx=u}.
Hence $p_T(\vec{x})\in U$ for all $\vec{x}\in A^k$.

\textbf{Induction step}: Let $\crd{T}=n\geq 3$ and let us
assume that~$l$ can be interpolated at any $n-1$ points of~$T$
by a polynomial whose image is a subset of~$U$. We prove that~$l$ can be
interpolated on~$T$ by a polynomial with image inside~$U$.
To this end, let us consider the following three sets
\begin{align*}
\beta&=\Cg{\{(\vec{t}_1(1),\vec{t}_2(1)),\dots,
(\vec{t}_1(k),\vec{t}_2(k))\}};\\
\eta&=\Biggl\{(p(\vec{t}_1),q(\vec{t}_1))\mathrel{\Bigg\vert} p,
q\in \POL\ari{k}\ab{A},
\begin{aligned}[c]
&\forall \vec{x}\in A^k\colon& p(\vec{x})&\mathrel{\mu}q(\vec{x}),\\
&\forall i\in\{2,\dots , n\}\colon& p(\vec{t}_i)&=q(\vec{t}_i)
\end{aligned}
\Biggr\};\\
\alpha&=\Biggl\{(p(\vec{t}_1),q(\vec{t}_1))\mathrel{\Bigg\vert} p,
q\in \POL\ari{k}\ab{A},
\begin{aligned}
&\forall \vec{x}\in A^k\colon& p(\vec{x})&\mathrel{\mu} q(\vec{x}),\\
&\forall i\in\{3,\dots, n\}\colon& p(\vec{t}_i)&=q(\vec{t}_i)
\end{aligned}
\Biggr\}.
\end{align*}
It is easy to see that~$\eta$ and~$\alpha$ are reflexive and symmetric
subuniverses of $\ab{A}\times\ab{A}$ that are contained in~$\mu$.
Now, by Lemma~\ref{lem:Malcev-weak-con-is-con}, we have
$\alpha,\beta,\eta\in\Con\ab{A}$, and
$\alpha,\eta\in\set{\bottom{A},\mu}$ since $\alpha,\eta\leq \mu$.
\par

Our next goal is to prove that $\alpha=\eta$.
The definition of~$\alpha$ and~$\eta$ yields $\eta\leq \alpha$.
If $\alpha=\bottom{A}\leq \eta$, we have the desired equality; hence
we assume that $\bottom{A}<\alpha\leq \mu$, i.e., $\alpha=\mu$.
\par
We shall first prove that $C(1, 1,\alpha,\beta,\eta)$.
To this end let $(u,v)\in \beta$, $(a,b)\in \alpha$ and
$p\in\POL\ari{2}\ab{A}$ such that $p(a,u)\mathrel{\eta} p(a,v)$.
We have to show that $p(b, u) \mathrel{\eta} p(b, v)$.
Since $(a,b)\in \alpha$, there exist $p_a, p_b,\in \POL\ari{k}\ab{A}$
such that
\begin{enumerate}[(1)]
\item $\forall \vec{x}\in A^k\colon p_a(\vec{x})\mathrel{\mu}
p_b(\vec{x})$;
\item $\forall j\in \{3,\dots n\}\colon
p_a(\vec{t}_j)=p_b(\vec{t}_j)$;
\item $p_a(\vec{t}_1)=a$ and $p_b(\vec{t}_1)=b$.
\end{enumerate}
Since $(u,v)\in \beta$,
Lemma~\ref{lemma:generating_congruences_in_malcev_with_polynomials}
yields
that there exist $q,q'\in \POL\ari{k}\ab{A}$ such that
$q(\vec{t}_1)=u$, $q(\vec{t}_2)=v$, $q'(\vec{t}_1)=v$ and
$q'(\vec{t}_2)=u$. We define $p_u,p_v\colon A^k \to A$ by letting
\begin{align*}
p_u(\vec{x})&=d(q(\vec{x}),q'(\vec{x}), v)&
&\text{and}&&
&p_v(\vec{x})&= d(q(\vec{x}),u,v)
\end{align*}
for all $\vec{x}\in A^k$.
We observe that $p_u,p_v\in \POL\ari{k}\ab{A}$, and moreover we can see
that $p_u(\vec{t}_1)=u$,
$w\defeq p_u(\vec{t}_2)=d(v,u,v)=p_v(\vec{t}_2)$ and
$p_v(\vec{t}_1)=v$.
We further define $h,\hbar\in\POL\ari{k}\ab{A}$ at each
$\vec{x}\in A^k$ by
\begin{align*}
\hbar(\vec{x})&=  p(p_b(\vec{x}),p_u(\vec{x})),\\
h(\vec{x})    &=d(\hbar(\vec{x}),
                  d(p(p_a(\vec{x}), p_v(\vec{x})),
                    p(p_a(\vec{x}), p_u(\vec{x})),
                    \hbar(\vec{x})),
                  p(p_b(\vec{x}), p_v(\vec{x}))).
\end{align*}
For each $j\in\set{3, \dotsc, n}$ we have
$p(p_a(\vec{t}_j),p_u(\vec{t}_j))=\hbar(\vec{t}_j)$, and hence
\begin{align*}
h(\vec{t}_j)&=d\bigl(\hbar(\vec{t}_j),
                d\bigl(p(p_a(\vec{t}_j), p_v(\vec{t}_j)),
                  p(p_a(\vec{t}_j), p_u(\vec{t}_j)),
                  \hbar(\vec{t}_j)\bigr),
                p(p_b(\vec{t}_j), p_v(\vec{t}_j))\bigr)\\
            &=d\bigl(\hbar(\vec{t}_j),
                d\bigl(p(p_a(\vec{t}_j), p_v(\vec{t}_j)),
                  \hbar(\vec{t}_j),
                  \hbar(\vec{t}_j)\bigr),
                p(p_b(\vec{t}_j), p_v(\vec{t}_j))\bigr)\\
            &=d\bigl(\hbar(\vec{t}_j),
                p(p_a(\vec{t}_j), p_v(\vec{t}_j)),
                p(p_b(\vec{t}_j), p_v(\vec{t}_j))\bigr)\\
            &=d\bigl(\hbar(\vec{t}_j),
                p(p_b(\vec{t}_j), p_v(\vec{t}_j)),
                p(p_b(\vec{t}_j), p_v(\vec{t}_j))\bigr)
             =\hbar(\vec{t}_j).
\end{align*}
Moreover, using $w=p_u(\vec{t}_2)=p_v(\vec{t}_2)$, we have
\begin{align*}
h(\vec{t}_2)&=d\bigl(\hbar(\vec{t}_2),
                d\bigl(p(p_a(\vec{t}_2), p_v(\vec{t}_2)),
                  p(p_a(\vec{t}_2), p_u(\vec{t}_2)),
                  \hbar(\vec{t}_2)\bigr),
                p(p_b(\vec{t}_2), p_v(\vec{t}_2))\bigr)\\
            &=d\bigl(\hbar(\vec{t}_2),
                d\bigl(p(p_a(\vec{t}_2), w),
                  p(p_a(\vec{t}_2), w),
                  \hbar(\vec{t}_2)\bigr),
                p(p_b(\vec{t}_2), w)\bigr)\\
            &=d\bigl(\hbar(\vec{t}_2),
                \hbar(\vec{t}_2),
                p(p_b(\vec{t}_2), w)\bigr)\\
            &=p(p_b(\vec{t}_2), w)
             =p(p_b(\vec{t}_2), p_u(\vec{t}_2))
             =\hbar(\vec{t}_2).
\end{align*}
For every $\vec{x}\in A^k$ we have
$p_a(\vec{x})\mathrel{\mu}p_b(\vec{x})$,
and hence we get
\begin{align*}
p(p_a(\vec{x}),p_u(\vec{x}))&\mathrel{\mu}p(p_b(\vec{x}),p_u(\vec{x}))
=\hbar(\vec{x})\\
p(p_a(\vec{x}),p_v(\vec{x}))&\mathrel{\mu}p(p_b(\vec{x}),p_v(\vec{x})).
\end{align*}
Consequently,
\begin{multline*}
d\bigl(p\bigl(p_a(\vec{x}),p_v(\vec{x})\bigr),p\bigl(p_a(\vec{x}),p_u(\vec{x})\bigr),\hbar(\vec{x})\bigr)
\mathrel{\mu}
d\bigl(p\bigl(p_b(\vec{x}),p_v(\vec{x})\bigr),\hbar(\vec{x}),\hbar(\vec{x})\bigr)\\
=p(p_b(\vec{x}),p_v(\vec{x})),
\end{multline*}
and therefore
\begin{align*}
h(\vec{x})&=  d\bigl(\hbar(\vec{x}),
                d\bigl(p(p_a(\vec{x}), p_v(\vec{x})),
                  p(p_a(\vec{x}), p_u(\vec{x})),
                  \hbar(\vec{x})\bigr),
                p(p_b(\vec{x}), p_v(\vec{x}))\bigr)\\
&\mathrel{\mu}d\bigl(\hbar(\vec{x}),
                p(p_b(\vec{x}), p_v(\vec{x})),
                p(p_b(\vec{x}), p_v(\vec{x}))\bigr)
=\hbar(\vec{x}).
\end{align*}
From this we deduce that $h(\vec{t}_1)\mathrel{\eta}\hbar(\vec{t}_1)$,
and thus, by applying the unary polynomial
$z\mapsto d(p(b,u),d(p(a,v),z,p(b,u)),p(b,v))$ to the pair
$(p(a,v),p(a,u))\in\eta$, we have
\begin{align*}
p(b,v)
&=d\bigl(p(b, u), p(b, u), p(b, v)\bigr)\\
&=d\bigl(p(b, u), d(p(a, v), p(a, v),p(b, u)), p(b, v)\bigr)\\
&\mathrel{\eta}d\bigl(p(b, u), d(p(a, v), p(a, u),p(b, u)), p(b, v)\bigr)\\
&=h_1(\vec{t}_1)\mathrel{\eta} h_2(\vec{t}_1)=p(b, u).
\end{align*}
Hence $p(b, u)\mathrel{\eta} p(b, v)$.
This proves that $C(1,1,\alpha,\beta,\eta)$. Now,
Lemma~\ref{lemma:commutator-properties-Malcev-algebras}\eqref{item:char-commutator-in-Malcev-algs}
implies $[\alpha,\beta]\leq \eta$.
Since $\vec{t}_1\neq\vec{t}_2$ we have $\bottom{A}<\beta$, thus
$\mu\leq \beta$. With $\alpha=\mu$ being non-Abelian,
Lemma~\ref{lemma:commutator-properties-Malcev-algebras}\eqref{item:commutator-monotone}
yields
$\alpha=\mu=[\mu,\mu]\leq[\mu,\beta]=[\alpha,\beta]\leq\eta\leq\alpha$.
This concludes the proof of $\alpha=\eta$.

Now we construct the interpolating function. By the induction
hypothesis there are $p,q\in\POL\ari{k}\ab{A}$ with image inside~$U$,
such that~$p$ interpolates~$l$ at $\{\vec{t}_2,\dots , \vec{t}_n\}$
and~$q$ interpolates~$l$ at $\{\vec{t}_1,\vec{t}_3,\dots , \vec{t}_n\}$.
Since $U^2\subseteq \mu$, we have that
$p(\vec{x})\mathrel{\mu} q(\vec{x})$ for all $\vec{x}\in A^k$,
and moreover that $p(\vec{t}_i)= l(\vec{t}_i)=q(\vec{t}_i)$ for every
$i\in\{3,\dots, n\}$. Hence $(p(\vec{t}_1), q(\vec{t}_1))\in\alpha$.
Since $\alpha=\eta$ and $q(\vec{t}_1)= l(\vec{t}_1)$, we have that
$(p(\vec{t}_1), l(\vec{t}_1))\in\eta$.
Therefore, there exist $p_2, p_3\in \POL\ari{k}\ab{A}$ such that
\begin{enumerate}[(1)]
\item $\forall i\in\{2,\dots, n\}\colon
p_2(\vec{t}_i)=p_3(\vec{t}_i)$;
\item $\forall \vec{x}\in A^k\colon p_2(\vec{x})\mathrel{\mu}
p_3(\vec{x})$;
\item $p_2(\vec{t}_1)=p(\vec{t}_1)$ and $p_3(\vec{t}_1)=l(\vec{t}_1)$.
\end{enumerate}
Let $p_T\colon A^k\to A$ be defined by
$p_T(\vec{x})=d(p(\vec{x}),p_2(\vec{x}),p_3(\vec{x}))$ for all
$\vec{x}\in A^k$.
Clearly, $p_T\in\POL\ari{k}\ab{A}$. Moreover, we have that for all
$i\in\{2,\dots n\}$
\[p_T(\vec{t}_i)=
d(p(\vec{t}_i), p_2(\vec{t}_i),
p_3(\vec{t}_i))=p(\vec{t}_i)=l(\vec{t}_i).\]
Furthermore,
\[p_T(\vec{t}_1)=
d (p(\vec{t}_1), p_2(\vec{t}_1),
p_3(\vec{t}_1))=d(p(\vec{t}_1),p(\vec{t}_1),
l(\vec{t}_1))=l(\vec{t}_1).
\]
Moreover, we have that for all $\vec{x}\in A^k$
\[p_T(\vec{x})=
d(p(\vec{x}),p_2(\vec{x}), p_3(\vec{x}))\,\mu \, d(o,p_2(\vec{x}),
p_2(\vec{x}))=o.
\]
Thus, $p_T(\vec{x})\in o/\mu=U$ and we can conclude that~$p_T$ has
codomain~$U$ and interpolates~$l$ on~$T$.
\end{proof}

The following proposition is a partial converse of
Proposition~\ref{prop:the_TCT_type_of_monolith_when_delta_is_f_equal_constant}.
In particular, it states that every finite subdirectly irreducible
algebra with a monolith of type~$\tp{3}$ (which is non-Abelian
by~\cite[Theorem~5.7]{HobbMcK}) and a Ma\v{l}cev polynomial is
an equational domain with respect to its clone of polynomial
operations.
\begin{proposition}\label{prop:inverse_of_prop_the_TCT_type_of_monolith_when_deltarelation_is_f_equal_constant}
Let~$\ab{A}$ be a finite subdirectly irreducible algebra with a
Ma\v{l}cev polynomial, let~$\mu$ be the monolith and let us assume
that~$\mu$ is non-Abelian.
Then there exist $f\in\POL\ari{4}\ab{A}$ and $a\in A$ such that
$\deltarelation=\{\vec{x}\in A^4\mid f(\vec{x})=a\}$, and
$f_{\mu}$~is constant.
\end{proposition}
\begin{proof}
Let~$U$ be an equivalence class of~$\mu$ with at least two distinct
elements $a,b$ and let $f\colon A^4\to U$ be defined by
\[
f(x_1,x_2,x_3,x_4)=\begin{cases}
a &\text{ if } x_1=x_2 \text{ or } x_3=x_4;\\
b &\text{ otherwise}.
\end{cases}
\]
Proposition~\ref{prop:interpolating_polynomials_ins_sub_irr_type_three_malcev}
implies that $f\in \POL\ari{4}\ab{A}$.
From the definition of~$f$ we see that
$\deltarelation=\{\vec{x}\in A^4\mid f(\vec{x})=a\}$, and~$f_{\mu}$ has
$a/\mu=U=b/\mu$ as its single value.
\end{proof}
Next, we determine what can be said about a Ma\v{l}cev algebra whose
universal algebraic geometry contains all finite relations.
\begin{proposition}\label{prop:general_necessary_conditions_for_cofinitness}
Let~$\ab{A}$ be an algebra with a Ma\v{l}cev polynomial
and assume that every three-element quaternary relation on~$A$
is an algebraic set with respect to $\POL\ab{A}$. Then for all
$\alpha,\beta\in\Con\ab{A}\setminus \{\bottom{A}\}$ we have $[\alpha,
\beta]>\bottom{A}$.
\end{proposition}
\begin{proof}
Let~$d$ be the Ma\v{l}cev polynomial and let
$\alpha,\beta\in \Con\ab{A}\setminus \{\bottom{A}\}$.
We prove that $\neg C(\alpha, \beta; \bottom{A})$.
To this end, take $(a,b)\in \alpha\setminus \bottom{A}$ and
$(u,v)\in\beta\setminus \bottom{A}$ to form
$B=\set{(a,a,u, u),(a,a,u, v),(a,b,u,
u)}$, which does not contain $(a,b,u,v)$. By our assumption, the quaternary relation~$B$
is algebraic
with respect to $\POL\ab{A}$. Hence
Lemma~\ref{lemma:equivalent_definition_algebraic_set} yields that
there are quaternary polynomials $p,q\in\POL\ari{4}\ab{A}$ such that
$p(a,b,u, v)\neq q(a,b,u, v)$
and $p\restrict{B}=q\restrict{B}$.
We use these to define the binary polynomial operation~$f$
for all $x_1,x_2\in A$ by
\[
f(x_1,x_2)\defeq d(p(a,b,u, v),p(a,x_1,u, x_2), q(a,x_1,u, x_2)).
\]
For~$d$ is a Ma\v{l}cev operation, we readily verify
\begin{alignat*}{2}
f(a,u)&=d(p(a,b,u,v),p(a,a,u,u), q(a,a,u,u))&&=p(a,b,u,v),\\
f(a,v)&=d(p(a,b,u,v),p(a,a,u,v), q(a,a,u,v))&&=p(a,b,u,v),\\
f(b,u)&=d(p(a,b,u,v),p(a,b,u,u), q(a,b,u,u))&&=p(a,b,u,v),\\
f(b,v)&=d(p(a,b,u,v),p(a,b,u,v), q(a,b,u,v))&&=q(a,b,u,v).
\end{alignat*}
Since $p(a,b,u,v)\neq q(a,b,u,v)$, we have that $\neg C(\alpha,
\beta; \bottom{A})$;
thus the definition of the commutator
yields $[\alpha,\beta]\neq \bottom{A}$, cf.\ also
Lemma~\ref{lemma:commutator-properties-Malcev-algebras}\eqref{item:char-commutator-in-Malcev-algs}.
\end{proof}

The following proposition provides a condition on the
commutator which is sufficient for equational additivity in Ma\v{l}cev
algebras.
\begin{proposition}\label{prop:char_infinite_equ_additive_malcev}
Let~$\ab{A}$ be an algebra on a set~$A$ with
a Ma\v{l}cev polynomial $d\in\POL\ari{3}\ab{A}$.
If for all $\alpha,\beta\in\Con\ab{A}\setminus\{\bottom{A}\}$
we have $[\alpha, \beta]>\bottom{A}$, then $\POL\ab{A}$ is
equationally additive.
\end{proposition}
\begin{proof}
Let $n\in\N$, let $C,B\in\Alg\ari{n}\POL\ab{A}$ and let $\vec{w}\in
A^n\setminus (C\cup B)$. We prove that there exist a constant
$\tau_{\vec{w}}\in A$ and a polynomial $p_\vec{w}\in
\POL\ari{n}\ab{A}$ such that
$p_{\vec{w}}(\vec{w})\neq\tau_{\vec{w}}$ and
$p_{\vec{w}}(\vec{x})=\tau_{\vec{w}}$ for all
$\vec{x}\in C\cup B$.
Since~$C$ and~$B$ are algebraic and $\vec{w}\notin C\cup B$,
Lemma~\ref{lemma:equivalent_definition_algebraic_set} applied to~$B$
and~$C$, respectively, gives us
$f_C,f_B, g_C, g_B\in \POL\ari{n}\ab{A}$ such that
$f_C\restrict{C}=g_C\restrict{C}$,
$f_B\restrict{B}=g_B\restrict{B}$,
$f_C(\vec{w})\neq g_C(\vec{w})$ and
$f_B(\vec{w})\neq g_B(\vec{w})$.
Hence the congruences generated by these respective pairs are
non-trivial:
\begin{align*}
\alpha\defeq\Cg{\{(f_{C}(\vec{w}),g_{C}(\vec{w}))\}}&\neq\bottom{A},&
\beta \defeq\Cg{\{(f_{B}(\vec{w}),g_{B}(\vec{w}))\}}&\neq\bottom{A}.
\end{align*}
Therefore, the assumption yields that
\[
[\Cg{\{(f_C(\vec{w}),g_C(\vec{w}))\}},
 \Cg{\{(f_B(\vec{w}),g_B(\vec{w}))\}}]
=[\alpha,\beta]>\bottom{A}.\]
Thus, Lemma~\ref{lemma:char_comm_in_malcev} implies that there
exists a
polynomial $q\in \POL\ari{2}\ab{A}$ such that
\[q(g_C(\vec{w}),g_B(\vec{w}))\neq q(f_C(\vec{w}), f_B(\vec{w}))\]
and
$q(a_1, f_B(\vec{w}))=q(f_C(\vec{w}),
a_2)=q(f_C(\vec{w}),f_B(\vec{w}))$
holds for all $a_1,a_2\in A$.\par
Let us now define the polynomial $p_\vec{w}\in\POL\ari{n}\ab{A}$
for all $\vec{x}\in A^n$ by
\[
p_{\vec{w}}(\vec{x})=q(d(g_C(\vec{x}),f_C(\vec{x}),f_C(\vec{w})),d(g_B(\vec{x}),
f_B(\vec{x}), f_B(\vec{w}))).
\]
For every $\vec{c}\in C$ we have
\begin{align*}
p_{\vec{w}}(\vec{c})
&=q(d(g_C(\vec{c}),f_C(\vec{c}),f_C(\vec{w})), d(g_B(\vec{c}),
f_B(\vec{c}), f_B(\vec{w})))\\
&=q(f_C(\vec{w}),d(g_B(\vec{c}), f_B(\vec{c}), f_B(\vec{w})))\\
&=q(f_C(\vec{w}),f_B(\vec{w})),
\end{align*}
while for every $\vec{b}\in B$ we have
\begin{align*}
p_{\vec{w}}(\vec{b})
&=q(d(g_C(\vec{b}),f_C(\vec{b}),f_C(\vec{w})),d(g_B(\vec{b}),
f_B(\vec{b}), f_B(\vec{w})))\\
&=q(d(g_C(\vec{b}),f_C(\vec{b}),f_C(\vec{w})), f_B(\vec{w}))\\
&=q(f_C(\vec{w}),f_B(\vec{w})).
\end{align*}
On the other hand, we have
\begin{align*}
p_{\vec{w}}(\vec{w})&=q(d(g_C(\vec{w}),f_C(\vec{w}),f_C(\vec{w})),d(g_B(\vec{w}),
f_B(\vec{w}), f_B(\vec{w})))\\
&=    q(g_C(\vec{w}),g_B(\vec{w}))
 \neq q(f_C(\vec{w}),f_B(\vec{w})).
\end{align*}
Therefore, setting $\tau_{\vec{w}}=q(f_C(\vec{w}),f_B(\vec{w}))$ we
have that $p_{\vec{w}}(\vec{w})\neq \tau_{\vec{w}}$, whereas for all
$\vec{x}\in C\cup B$ the equality
$p_{\vec{w}}(\vec{x})=\tau_{\vec{w}}$ holds.
Hence Lemma~\ref{lemma:equivalent_definition_algebraic_set} yields
that $C\cup B\in\Alg\ari{n}(\POL\ab{A})$.
\end{proof}
\begin{theorem}\label{teor:equ_additive_cofinite_non_abelianity}
\newcounter{tmpenumcnt}%
Let~$\ab{A}$ be an algebra with at least two elements and a
Ma\v{l}cev polynomial.
Then the following statements are equivalent:
\begin{enumerate}[\upshape(a)]
\item\label{item:teor:equ_additive_cofinite_non_abelianity_item_eq_additive}
      $\POL\ab{A}$ is equationally additive.
\item\label{item:teor:equ_additive_cofinite_non_abelianity_item_cofinite}
      For all $n\in\N$, any finite subset of~$A^n$ belongs to
      $\Alg(\POL(\ab{A}))$.
\item\label{item:teor:equ_additive_cofinite_non_abelianity_item_three_element}
      Every three-element subset of~$A^4$ belongs to
      $\Alg(\POL(\ab{A}))$.
\item\label{item:teor:equ_additive_cofinite_non_abelianity_item_commutators}
      For all $\alpha,\beta\in\Con\ab{A}\setminus\{\bottom{A}\}$
      we have $[\alpha, \beta]>\bottom{A}$.
\setcounter{tmpenumcnt}{\value{enumi}}%
\end{enumerate}
If~$A$ is finite,
\eqref{item:teor:equ_additive_cofinite_non_abelianity_item_eq_additive}--\eqref{item:teor:equ_additive_cofinite_non_abelianity_item_commutators}
are furthermore equivalent to the following:
\begin{enumerate}[\upshape(a)]
\setcounter{enumi}{\value{tmpenumcnt}}%
\item\label{item:teor:equ_additive_cofinite_non_abelianity_item_finite_type_three}
      $\ab{A}$ is subdirectly irreducible and the monolith~$\mu$ is
      non-Abelian.
\item\label{item:teor:equ_additive_cofinite_non_abelianity_item_deltarelation}
      There exist $f\in \POL\ari{4}\ab{A}$ and $a\in A$ such that
      $\deltarelation=\{\vec{x}\in A^4\mid f(\vec{x})=a\}$
      and~$f_{\gamma}$ is constant
      for all $\gamma\in\Con\ab{A}\setminus\{\bottom{A}\}$.
\end{enumerate}
\end{theorem}
\begin{proof}
Since $\POL\ab{A}$ contains all constant operations,
for every $n\in\N$, every singleton $\set{(a_1,\dotsc,a_n)}$ can be
written as
\[
\set{\vec{x}\in A^n
\Bigm\vert\cna{a_1}(\vec{x})=\eni{1}(\vec{x})\land\dotsm
                  \land \cna{a_n}(\vec{x})=\eni{n}(\vec{x})}
\]
and $\emptyset=\set{x\in A\mid \ca(x)=\ca[b](x)}$
where $a,b\in A$ are distinct elements. Therefore,
\eqref{item:teor:equ_additive_cofinite_non_abelianity_item_eq_additive}
implies~\eqref{item:teor:equ_additive_cofinite_non_abelianity_item_cofinite},
and~\eqref{item:teor:equ_additive_cofinite_non_abelianity_item_three_element}
is just a special case
of~\eqref{item:teor:equ_additive_cofinite_non_abelianity_item_cofinite}.
Proposition~\ref{prop:general_necessary_conditions_for_cofinitness}
proves
that~\eqref{item:teor:equ_additive_cofinite_non_abelianity_item_three_element}
implies~\eqref{item:teor:equ_additive_cofinite_non_abelianity_item_commutators},
while Proposition~\ref{prop:char_infinite_equ_additive_malcev} shows
that~\eqref{item:teor:equ_additive_cofinite_non_abelianity_item_commutators}
implies~\eqref{item:teor:equ_additive_cofinite_non_abelianity_item_eq_additive}.
If we assume that~$A$ is finite, then by
Corollary~\ref{cor:nec_condi_malcev_polynomial_and_finite_yields_sub_irreducible_and_type_three}
we have
that~\eqref{item:teor:equ_additive_cofinite_non_abelianity_item_eq_additive}
implies~\eqref{item:teor:equ_additive_cofinite_non_abelianity_item_finite_type_three};
moreover,
Proposition~\ref{prop:inverse_of_prop_the_TCT_type_of_monolith_when_deltarelation_is_f_equal_constant}
shows
that~\eqref{item:teor:equ_additive_cofinite_non_abelianity_item_finite_type_three}
implies~\eqref{item:teor:equ_additive_cofinite_non_abelianity_item_deltarelation}.
From Theorem~\ref{teor:eq_additive_iff_delta_four_algebraic}
we see
that~\eqref{item:teor:equ_additive_cofinite_non_abelianity_item_deltarelation}
implies~\eqref{item:teor:equ_additive_cofinite_non_abelianity_item_eq_additive}.
\end{proof}

We now specify our results to Artinian rings.
We remark that the commutator of two ideals
as defined in Section~\ref{sec:preliminaries}
in the case of rings coincides with
the classical ideal product
(cf.~\cite[Exercise~4.156(12)]{McKMcnTay88}).
\begin{corollary}\label{cor:rings}
Let\/~$\ab{R}$ be an Artinian ring with unity. Then $\POL\ab{R}$ is
equationally additive if and only if\/~$\ab{R}$ is isomorphic to the
ring of linear endomorphisms of a finite dimensional vector space
over a division ring.
\end{corollary}
\begin{proof}
If~$\ab{R}$ consists of a single element, then $\POL\ab{R}$ is
equationally additive and~$\ab{R}$ is isomorphic to the endomorphism
ring of a zero-dimensional vector space. From now on let us assume
that zero and unity in~$\ab{R}$ are distinct.
If~$\ab{R}$ is the ring of linear transformations of a
finite dimensional vector space over a division  ring,
then it is simple and non-Abelian. Hence it satisfies~\eqref{item:teor:equ_additive_cofinite_non_abelianity_item_commutators}
of Theorem~\ref{teor:equ_additive_cofinite_non_abelianity}.
Let us now assume that $\POL\ab{R}$
is equationally additive. Then~$\ab{R}$ satisfies~\eqref{item:teor:equ_additive_cofinite_non_abelianity_item_commutators}
of Theorem~\ref{teor:equ_additive_cofinite_non_abelianity}, and
therefore it is subdirectly irreducible and the monolithic ideal~$I$
satisfies $I\cdot I=I$.
Since the Jacobson radical $\rad{\ab{R}}$ is
nilpotent (cf.~\cite[Theorem~4.3]{BSAII}), we infer
that $\rad{\ab{R}}=\set{0}$. Thus, $\ab{R}$ is primitive
(cf.~\cite[Propositions~4.1 and~4.4]{BSAII}). Hence the
Wedderburn--Artin Theorem yields that~$\ab{R}$ is isomorphic to the
ring of linear transformations of a finite dimensional vector space
over a division ring.
\end{proof}

Corollary~\ref{cor:rings} entails that the equational domains among all
Artinian rings with unity are simple.
This is a consequence of the fact
that for Artinian rings with unity
condition~\eqref{item:teor:equ_additive_cofinite_non_abelianity_item_commutators}
of Theorem~\ref{teor:equ_additive_cofinite_non_abelianity} implies simplicity.
In the case of near-rings (cf.~\cite[Definition~1.1]{PilzNR}) this
is not any more true.
We provide an example of a finite near-ring that is an equational domain
but not simple.
\begin{corollary}
For a prime number~$p>2$
the near-ring $\ab{N}={\algop{C_{0}(\Z_{p^2})}{+, \circ}}$ of
zero-preserving congruence-preserving functions on~$\Z_{p^2}$ is
an equational domain but not simple.
\end{corollary}
\begin{proof}
One readily verifies that~$\ab{N}$ satisfies the assumptions
of~\cite[Corollary~5.2]{AicCanEckKabNeu08}, and therefore~$\ab{N}$ is
subdirectly irreducible, and its monolithic ideal
is equal to $M=(0:p\Z_{p^2})\cap (p\Z_{p^2}:\Z_{p^2})$, that is, the ideal
consisting of all the maps that send $p\Z_{p^2}$ to~$0$ and~$\Z_{p^2}$
to $p\Z_{p^2}$.
Next, we show that $[M,M]\neq {0}$. To this end, we define $a,b,x\colon
\Z_{p^2}\to \Z_{p^2}$
as follows. For every $n\in\Z_{p^2}$ we set
\begin{align*}
a(n)&\defeq
\begin{cases}
0 & \text{ if } n\in \{1, \dots, p-1\}\cup p\Z_{p^2},\\
kp & \text{ if } n\in \{kp+1, \dots, (k+1)p-1\}, \text{ with }
k\in\{1, \dots, p-1\};
\end{cases}\\
b(n)&\defeq pn;\\
x(n)&\defeq
\begin{cases}
n &\text{ if } n\in p\Z_{p^2},\\
n\bmod p &\text{ if } n\in \Z_{p^2}\setminus p\Z_{p^2}.
\end{cases}
\end{align*}
We observe that $a,b\in M$, and $x\in C_0(\Z_{p^2})$. Hence
$a\circ(b+x)-a\circ x\in [M,M]$ (cf.~\cite[Definition~2.1]{AicMud09}
and~\cite[Theorem~3.1]{JanMarVel16}).
Thus, we need only show that $a\circ(b+x)-a\circ x$ is
not constantly zero. One readily verifies that $(a\circ
x)\,(1)=0$ and $(a\circ (b+x))\,(1)=p$. Hence $[M,M]\neq \set{0}$, and
Theorem~\ref{teor:equ_additive_cofinite_non_abelianity} yields
that~$\ab{N}$ is an equational domain.
Since $\set{0}\subsetneq M\subsetneq C_{0}(\Z_{p^2})$, we obtain
that~$\ab{N}$ is not simple.
\end{proof}

\section{The number of equationally additive constantive expansions of finite Abelian groups}\label{section:on_the_number_abelian_groups}
In the present section we study the number of equationally
additive constantive expansions of Abelian groups on finite sets.
We begin with a lemma.
\begin{lemma}\label{lem:corollario_quanti_cloni_equationally_additive_Z_m}
Let $l,m,p,q\in\N$, with~$p$ and~$q$ prime, and~$m$ square-free.
For $n\in\N$ we write $\ab{Z}_n\defeq \algop{\Z_n}{+,-,0}$ for the
cyclic group of integers
$\set{0,\dotsc,n-1}$ modulo~$n$. Let~$\ab{H}$ be a finite group the
centre of which contains a subgroup of order~$q^2$.
Then the following statements hold:
\begin{enumerate}[\upshape(a)]
\item There are only finitely many
      equationally additive clones that contain the clone
      $\Clo(\ab{Z}_m)$.
      \label{item:lem:corollario_quanti_cloni_equationally_additive_Z_m_item_squarefree}
\item The number of equationally additive clones containing
      $\POL(\ab{Z}_p\times\ab{Z}_p)$ or $\POL(\ab{Z}_{p^2})$,
      respectively, is finite.
      \label{item:lem:corollario_quanti_cloni_equationally_additive_Z_m_item_p_quadro}
\item If $l\geq3$, then there are exactly~$\aleph_0$ equationally
      additive clones that contain $\POL(\ab{Z}_{p^l})$.
      \label{item:lem:corollario_quanti_cloni_equationally_additive_Z_p_to_k}
\item There are exactly~$\aleph_0$ equationally additive clones
      above $\POL(\ab{H}\times\ab{Z}_{p^l})$.
      \label{item:lem:corollario_quanti_cloni_equationally_additive_Z_m_item_npquadro_q}
\end{enumerate}
\end{lemma}
\begin{proof}
In~\cite[Corollary~1.3]{Fio21} it is shown that for
square-free~$m$, there are only finitely many clones containing
$\Clo(\ab{Z}_m)$;
hence~\eqref{item:lem:corollario_quanti_cloni_equationally_additive_Z_m_item_squarefree}
follows (the assumption of equational additivity is not used for this).
\par

We now
prove~\eqref{item:lem:corollario_quanti_cloni_equationally_additive_Z_m_item_p_quadro}.
Let~$\ab{V}$ be an expansion of~$\ab{Z}_{p^2}$ or
of~$\ab{Z}_{p}\times\ab{Z}_p$ such that $\POL\ab{V}$ is equationally
additive.
Then
Corollary~\ref{cor:nec_condi_malcev_polynomial_and_finite_yields_sub_irreducible_and_type_three}
implies that~$\ab{V}$ is subdirectly irreducible and that the
monolith~$U$ is non-Abelian.
Then, either~$\ab{V}$ is simple and non-Abelian, or $\Id\ab{V}$
is a three-element chain with a non-Abelian monolith.
Thus, $\ab{V}$ satisfies the property~(SC1) as defined
in~\cite[p.~310]{AicMud09}, which for finite expanded groups is
equivalent to the property~(SC1) given in~\cite[p.~48]{IdzSlo01}, as
was argued in~\cite[p.~310]{AicMud09}.
Hence, by~\cite[Lemma~21]{IdzSlo01}, $\Id\ab{V}$ satisfies (APMI) as
defined in~\cite[p.~310 and Definition~8.1, p.~324]{AicMud09}.
Thus, \cite[Corollary~11.3]{AicMud09} yields that~$\ab{V}$ is weakly
polynomially rich, that is, according
to~\cite[Definition~3.7]{AicMud09}, the clone of polynomial
functions of~$\ab{V}$ coincides with the clone of extended type
preserving
functions as defined in~\cite[Definition~3.4]{AicMud09}.
Moreover, under (APMI), \cite[Corollary~11.7]{AicMud09} yields
that the
clone of functions preserving the extended types of~$\ab{V}$ is
generated by
the binary functions it contains. Therefore, we can infer that
$\POL\ab{V}$ is
generated by~$\POL\ari{2}\ab{V}$, a subset the $p^{2p^4}$-element
set of binary
operations on the carrier of~$\ab{V}$ that contains addition,
the two projections and all $p^2$~constants. Thus,
$\POL\ab{V}$ is one of at most $2^{p^{2p^4}-p^2-3}$ clones.

Next, we
prove~\eqref{item:lem:corollario_quanti_cloni_equationally_additive_Z_p_to_k}.
To this end
let $N\defeq \Braket{p^{l-1}}$. Clearly, $N$ is normal,
$N\cong \ab{Z}_p$, and
$\ab{Z}_{p^l}/N\cong \ab{Z}_{p^{l-1}}$.
We define $f\colon \Z_{p^{l}}^4\to\Z_{p^{l}}$
for $x_1,x_2,x_3,x_4\in \Z_{p^l}$ by
$f(x_1,x_2,x_3,x_4)\defeq 0$ if $x_1=x_2$ or $x_3=x_4$, and
$f(x_1,x_2,x_3,x_4)\defeq p^{l-1}$ otherwise.
Furthermore, for each $i\in\N\setminus\set{1}$, we set
$\ab{V}_i=\algop{\Z_{p^l}}{+,-,0,f,h_i}$, where
$h_i\colon \Z_{p^l}^i\to\Z_{p^l}$ is given for
$x_1,\dotsc,x_i\in \Z_{p^l}$ by
$h_i(x_1,\dotsc,x_i)\defeq  p^{l-2}\prod_{j=1}^i x_j$.
One readily checks that~$N$ is an ideal of~$\ab{V}_i$
and that any ideal~$I$ of~$\ab{V}_i$ with at least one element $a\neq
0$ must contain
${f(a,0,0,0)-f(0,0,0,0)=p^{l-1}}$.
Hence~$N$ is the monolith of~$\Id\ab{V}_i$.
Since $\deltarelation[\Z_{p^l}]=\{\vec{x}\in (\Z_{p^l})^4\mid
f(\vec{x})=0\}$,
Theorem~\ref{teor:eq_additive_iff_delta_four_algebraic}
yields that $\Clo\ab{V}_i$ is equationally additive, and, by
Corollary~\ref{cor:equational_add_and_inclusion},
$\POL\ab{V}_i$ is equationally additive, as well.
The map $\phi\colon \Z_{p^{l-1}}\to \Z_{p^l}/N$ sending each
$x\in\Z_{p^{l-1}}$ to $\phi(x)=x+N$ provides an isomorphism between
the
algebra $\ab{B}_i\defeq\algop{\Z_{p^{l-1}}}{+,-,0,\cna[4]{0},p_i}$ and~$\ab{V}_i/N$,
where~$\cna[4]{0}$ is the quaternary constant zero function and
$p_i(x_1,\dotsc,x_i) = p^{l-2}x_1\dotsm x_i$ for
$x_1,\dotsc,x_{i}\in\Z_{p^{l-1}}$.
For each $i\in\N\setminus\{1\}$, let
$\ab{A}_i=\algop{\Z_{p^{l-1}}}{+,-,0,p_i}$,
cf.~\cite[proof of Theorem~1.3]{AicRos23}.
Then, for all $i\in\N\setminus\{1\}$ we have that
$\POL\ab{B}_i=\POL\ab{A}_i$.
Moreover, in~\cite[proof of Theorem~1.3]{AicRos23} it is proved that
\begin{equation}\label{eq:equazione_AicRos23_for_expansions_of_Z_p_square}
\forall i, j\in \N\setminus\{1\}\colon \POL\ab{A}_i=\POL\ab{A}_j
\iff i=j.
\end{equation}
It is our goal
to show this for $\POL\ab{V}_i$ and $\POL\ab{V}_j$, as
well. To this end, let $i,j\in \N\setminus\{1\}$ be such
that $\POL\ab{V}_i=\POL\ab{V}_j$. We show that $i=j$.
Let~$\psi(N)$ be the congruence associated to~$N$ as defined in
Section~\ref{sec:preliminaries}, that is, the kernel of~$\phi$.
If $\POL\ab{V}_i=\POL\ab{V}_j$, then
$(\POL\ab{V}_i)/\psi(N)=(\POL\ab{V}_j)/\psi(N)$, since~$\psi(N)$ does
not depend on the choice of~$i$ and~$j$. Therefore,
equation~\eqref{eq:the_clone_of_actions_of_function_on_a_quotient_algebra}
yields that $\POL(\ab{V}_i/N)=\POL(\ab{V}_j/N)$. Then, since~$\phi$
does not depend on the choice of~$i$ and~$j$,
$\POL\ab{B}_i=\POL\ab{B}_j$, and hence $\POL\ab{A}_i=\POL\ab{A}_j$.
Finally,
condition~\eqref{eq:equazione_AicRos23_for_expansions_of_Z_p_square}
yields that $i=j$.
Thus, the map $i\mapsto \POL\ab{V}_i$ from $\N\setminus\{1\}$ to
the set of clones on~$\Z_{p^l}$ is injective.
This proves that there are at least~$\aleph_0$ distinct equationally
additive
clones that contain $\POL\ab{Z}_{p^l}$ for $l\geq 3$.
Since there are at most~$\aleph_0$ constantive Ma\v{l}cev clones on a
finite set~\cite[Theorem~5.3]{Aic10}, $\aleph_0$ is the exact number.

It remains to
prove~\eqref{item:lem:corollario_quanti_cloni_equationally_additive_Z_m_item_npquadro_q}.
For any finite Abelian group~$\ab{G}'$, for any $n\in\N$ and for an
$n$-ary operation
$h\colon H^n\to H$, we define
$\iota_{\ab{G}'}(h)\colon (\ab{H}\times\ab{G'})^n\to \ab{H}\times
\ab{G}'$ by
$\iota_{\ab{G}'}(h)((x_1,y_1),\dotsc,(x_n,y_n))\defeq (h(x_1,\dotsc,x_n),0)$
for all $\vec{x}\in H^n$ and all $\vec{y}\in (G')^n$.
Let us set
$\ab{V}_{\ab{G}',f',\funset{H}}=
\algop{H\times G'}{+,-,(0,0),f',(\iota_{\ab{G}'}(h))_{h\in \funset{H}}}$,
where $\funset{H}$ is any set of operations on~$H$ and~$f'$ is a
quaternary operation on~$H\times G'$.
For $n\in\N$,
and for any $p\in\POL\ari{n}\ab{V}_{\ab{G}',f',\funset{H}}$,
let $\pi(p)\colon H^n\to H$ be the function that maps
every $\vec{x}\in H^n$ to the projection of
$p\pair{\vec{x}\\\vec{0}}$ to its first component.
We set
$\pi\bigl(\POL\ab{V}_{\ab{G}',f',\funset{H}}\bigr)\defeq
\bigcup_{n\in\N}\{\pi(p)\mid
p\in\POL\ari{n}\ab{V}_{\ab{G}',f',\funset{H}}\}$;
moreover let~$\cna[4]{\vec{0}}$ denote the constant zero function on
$H\times G'$ of arity four.
\par
Next, we demonstrate that for each constantive clone~$\funset{H}$
on~$H$ that contains the group operation of~$\ab{H}$, and for each
finite Abelian group~$\ab{G}'$, we have
\begin{equation}\label{eq:fancyH_e_la_prima_proiezione_di_V_sottoscritto_co_fancyH}
\funset{H}=\pi\bigl(\POL\ab{V}_{\ab{G}',\cna[4]{\vec{0}},\funset{H}}\bigr).
\end{equation}
For each $h\in H$ we see that $\pi(\iota_{\ab{G}'}(h)) = h$, and, since
$\iota_{\ab{G}'}(h)$ is a fundamental operation of
$\ab{V}_{\ab{G}',\cna[4]{\vec{0}},\funset{H}}$, we thus have
$\funset{H}\subs\pi\bigl(\POL\ab{V}_{\ab{G}',\cna[4]{\vec{0}},\funset{H}}\bigr)$.
For the opposite inclusion, we note that,
since the second parameter of
$\ab{V}_{\ab{G}',\cna[4]{\vec{0}},\funset{H}}$ is constant, we can
write this algebra as the product
$\ab{\hat{H}}\times\ab{\hat{G}}$ where
$\ab{\hat{H}}=\algop{H}{+,-,0,\cna[4]{0},(h)_{h\in\funset{H}}}$,
$\ab{\hat{G}}=\algop{G'}{+,-,0,\cna[4]{0},\bigl(\cna[\arity(h)]{0}\bigr)_{h\in\funset{H}}}$
and~$\arity(h)$ denotes the arity of $h\in\funset{H}$.
We now extend the signature of these algebras by all constant values of
$H\times G'$ as follows. We define
\begin{align*}
\ab{\hat{H}}_+&\defeq\algop{H}{+,-,0,\cna[4]{0},(h)_{h\in\funset{H}},
                               (a)_{(a,b)\in H\times G'}},\\
\ab{\hat{G}}_+&\defeq\algop{G'}{+,-,0,\cna[4]{0},
                       \bigl(\cna[\arity(h)]{0}\bigr)_{h\in\funset{H}},
                       (b)_{(a,b)\in H\times G'}},
\end{align*}
such that the term operations of $\ab{\hat{H}}_+\times\ab{\hat{G}}_+$
become the polynomial operations of
$\ab{\hat{H}}\times\ab{\hat{G}}=\ab{V}_{\ab{G}',\cna[4]{\vec{0}},\funset{H}}$.
We now use the homomorphism property of the projection~$\pi_H$
onto~$\ab{\hat{H}}_+$. For every $n$-ary term~$t$ in the language
of~$\ab{\hat{H}}_+$ and $\vec{x}\in H^n$ we have
\[
\pi\bigl(t^{\ab{\hat{H}}_+\times\ab{\hat{G}}_+}\bigr)\,(\vec{x})
=\pi_H\mleft(t^{\ab{\hat{H}}_+\times\ab{\hat{G}}_+}
                                    \pair{\vec{x}\\\vec{0}}\mright)
=\pi_H\mleft(\begin{smallmatrix}
t^{\ab{\hat{H}}_+}(\vec{x})\\t^{\ab{\hat{G}}_+}(\vec{0})
\end{smallmatrix}\mright)
=t^{\ab{\hat{H}}_+}(\vec{x}),
\]
hence the operation
$\pi\bigl(t^{\ab{\hat{H}}_+\times\ab{\hat{G}}_+}\bigr)$ coincides with
the operation $t^{\ab{\hat{H}}_+}\in\Clo\ab{\hat{H}}_+$.
Moreover, as~$\funset{H}$ is a constantive clone on~$H$ including the
addition of~$\ab{H}$, we observe that
$\funset{H}\subs\POL\ab{\hat{H}}\subs\Clo\ab{\hat{H}}_+\subs
\funset{H}$. Thus, $\pi$ maps every polynomial operation of
$\ab{V}_{\ab{G}',\cna[4]{\vec{0}},\funset{H}}
=\ab{\hat{H}}\times\ab{\hat{G}}$ into~$\funset{H}$,
proving
equation~\eqref{eq:fancyH_e_la_prima_proiezione_di_V_sottoscritto_co_fancyH}.
\par
Let $\ab{G}\defeq \ab{H}\times\ab{Z}_{p^l}$ and consider the subgroup
$N\defeq \Braket{(0,p^{l-1})}$ of~$\ab{G}$. Again, $N$ is normal,
$N\cong \ab{Z}_{p}$, and $\ab{G}/N \cong \ab{H}\times
\ab{Z}_{p^{l-1}}$
by mapping $\vec{x}\in\ab{H}\times\ab{Z}_{p^{l-1}}$ to
$\phi(\vec{x}) = \vec{x}+N$.
For
$\vec{x}_1,\vec{x}_2,\vec{x}_3,\vec{x}_4\in G$ set
$f(\vec{x}_1,\vec{x}_2,\vec{x}_3,\vec{x}_4)\defeq (0,0)$ if
$\vec{x}_1=\vec{x}_2$ or $\vec{x}_3=\vec{x}_4$, and
$f(\vec{x}_1,\vec{x}_2,\vec{x}_3,\vec{x}_4)\defeq (0,p^{l-1})$ otherwise;
this gives $f\colon G^4\to G$.
In the proof of Theorem~6 from~\cite{Idz99}, Idziak constructs a
strictly increasing infinite sequence
$\funset{C}'_3\subsetneq \funset{C}'_4\subsetneq
\funset{C}'_5\subsetneq \dots$
of clones on~$H$, containing the group operation and all constants
from~$H$. For any constantive clone~$\funset{H}$ on~$H$ and for any
function $f'\colon G^4\to G$, we abbreviate
$\ab{V}_{j,f',\funset{H}}\defeq \ab{V}_{\ab{Z}_{p^j},f',\funset{H}}$ for
any $j\in\N$.
A routine check establishes that~$N$ is an ideal of the expanded
groups~$\ab{V}_{l,f,\funset{H}}$, for any choice of~$\funset{H}$.
As argued in the proof
of~\eqref{item:lem:corollario_quanti_cloni_equationally_additive_Z_p_to_k},
any ideal~$I$ of~$\ab{V}_{l,f,\funset{H}}$ with
$\vec{0}\neq \vec{a}\in I$ must contain the element
$f(\vec{a},\vec{0},\vec{0},\vec{0})-f(\vec{0},\vec{0},\vec{0},\vec{0})=(0,p^{l-1})$.
Therefore, $N$ is the monolith of $\ab{V}_{l,f,\funset{H}}$, and
the map~$\phi$ from above provides an isomorphism
from $\ab{V}_{l-1,\cna[4]{\vec{0}},\funset{H}}$ to
$\ab{V}_{l,f,\funset{H}}/N$, where $\cna[4]{\vec{0}}$ is the quaternary
constant zero function.
By Theorem~\ref{teor:eq_additive_iff_delta_four_algebraic} and
Corollary~\ref{cor:equational_add_and_inclusion},
$\POL\ab{V}_{j,f,\funset{H}}$ is equationally additive for
every~$j\in\N$, since~$f$ and the constant with value~$\vec{0}=(0,0)$
allow us to
define the algebraic set~$\deltarelation[H\times\Z_{p^j}]$.
\par
Let~$\funset{H}_1, \funset{H}_2$ be two constantive clones on~$H$
that contain the group operation of~$\ab{H}$ and let us assume that
$\POL\ab{V}_{l,f,\funset{H}_1}= \POL\ab{V}_{l,f,\funset{H}_2}$. We
prove that $\funset{H}_1=\funset{H}_2$.
Since~$N$ does not depend on the choice of~$\funset{H}_1$
and~$\funset{H}_2$, setting~$\psi(N)$ to be the congruence associated
to~$N$ (cf. Section~\ref{sec:preliminaries} and the proof of~\eqref{item:lem:corollario_quanti_cloni_equationally_additive_Z_p_to_k}),
we have that
$(\POL\ab{V}_{l,f,\funset{H}_1})/\psi(N)=(\POL\ab{V}_{l,f,\funset{H}_2})/\psi(N)$.
Thus,
\eqref{eq:the_clone_of_actions_of_function_on_a_quotient_algebra}
yields that
\[\POL(\ab{V}_{l,f,\funset{H}_1}/N)=\POL(\ab{V}_{l,f,\funset{H}_2}/N).\]
Since for every constantive clone~$\funset{H}$ on~$H$ that contains
the
group operation of~$\ab{H}$, $\phi$ provides an isomorphism between
$\ab{V}_{l-1,\cna[4]{\vec{0}},\funset{H}}$ and $\ab{V}_{l,f,\funset{H}}/N$
that does not depend on the choice of~$\funset{H}$, we infer that
$\POL\ab{V}_{l-1,\cna[4]{\vec{0}},\funset{H}_1}=\POL\ab{V}_{l-1,\cna[4]{\vec{0}},\funset{H}_2}$,
and
therefore~\eqref{eq:fancyH_e_la_prima_proiezione_di_V_sottoscritto_co_fancyH}
yields
$\funset{H}_1=\pi\bigl(\POL\ab{V}_{l-1,\cna[4]{\vec{0}},\funset{H}_1}\bigr)
             =\pi\bigl(\POL\ab{V}_{l-1,\cna[4]{\vec{0}},\funset{H}_2}\bigr)
             =\funset{H}_2.$
This means that the map
$\funset{H}\mapsto\POL\ab{V}_{l,f,\funset{H}}$,
defined for constantive expansions~$\funset{H}$ of~$\Clo\ab{H}$, is
injective. Therefore, for each of the~$\aleph_0$ examples given by
Idziak, we have a distinct equationally additive clone
$\POL\ab{V}_{l,f,\funset{C}'_j}
 \supseteq \POL(\ab{H}\times\ab{Z}_{p^l})$.
As argued in the proof
of~\eqref{item:lem:corollario_quanti_cloni_equationally_additive_Z_p_to_k},
\cite[Theorem~5.3]{Aic10} shows that the
number of constantive equationally additive expansions of
$\Clo(\ab{H}\times\ab{Z}_{p^l})$ cannot be larger than~$\aleph_0$.
\end{proof}

\begin{theorem}\label{thm:char-eqn-add-poly-expansions-ab-grp}
Let~$\ab{G}$ be a finite Abelian group with~$m$ elements.
If~$m$ is square-free or the square of a prime, the set of equationally
additive clones containing $\POL\ab{G}$ is finite. Otherwise, it is
countably infinite.
\end{theorem}
\begin{proof}
Consider the representation of~$\ab{G}$ as a direct product of cyclic
groups of prime power order.
\par
First, we suppose that every factor of this product is of prime order.
That is, $\ab{G}\cong\prod_{i=1}^n \bigl(\ab{Z}_{p_i}\bigr)^{k_i}$
with
$n\geq 0$, distinct primes $p_1,\dotsc,p_n$ and integers
$k_1,\dotsc,k_n\in\N$.
If~$\ab{G}$ is trivial, i.e., $n=0$, or $n\geq 1$ and $k_i=1$ for all
$i\in\finset{n}$, then $m=\prod_{i=1}^n p_i$ is square-free and the
result follows from
Lemma~\ref{lem:corollario_quanti_cloni_equationally_additive_Z_m}\eqref{item:lem:corollario_quanti_cloni_equationally_additive_Z_m_item_squarefree}.
Otherwise, there is $i\in\finset{n}$ such that $k_i\geq 2$, and no
generality is lost in assuming $i=1$.
\par
As a subcase we consider the possibility that $n=1$, that is,
$\ab{G}\cong\ab{Z}_{p_1}^{k_1}$ with $k_1\geq 2$.
If $k_1=2$, then $m=k_1^2$, and hence, by
Lemma~\ref{lem:corollario_quanti_cloni_equationally_additive_Z_m}\eqref{item:lem:corollario_quanti_cloni_equationally_additive_Z_m_item_p_quadro},
there are only finitely many equationally additive clones containing
$\POL\ab{G}$.
If, otherwise, $k_1\geq 3$, then
$\ab{G}\cong\ab{Z}_{p_1}^{k_1-1}\times\ab{Z}_{p_1}$, $m=p_1^{k_1}$ and
$\ab{Z}_{p_1}\times\ab{Z}_{p_1}\times\set{0}^{k_1-3}$ is an Abelian
subgroup of~$\ab{Z}_{p_1}^{k_1-1}$ of order~$p_1^2$. Then the result
follows from
Lemma~\ref{lem:corollario_quanti_cloni_equationally_additive_Z_m}\eqref{item:lem:corollario_quanti_cloni_equationally_additive_Z_m_item_npquadro_q}.
This finishes the subcase where $n=1$.
\par
The opposite possibility is that $n\geq 2$; in this subcase we
represent~$\ab{G}$ as
$\ab{G}\cong\ab{Z}_{p_1}^{k_1}\times\ab{Z}_{p_2}^{k_2-1}\times\bigl(\prod_{i=3}^n\ab{Z}_{p_i}^{k_i}\bigr)\times\ab{Z}_{p_2}$,
and
$\ab{Z}_{p_1}\times\ab{Z}_{p_1}\times\set{0}^{k_1+\dotsm+k_n-3}$ is an
Abelian subgroup of
$\ab{Z}_{p_1}^{k_1}\times
           \ab{Z}_{p_2}^{k_2-1}\times\prod_{i=3}^n\ab{Z}_{p_i}^{k_i}$
of order~$p_1^2$. Clearly, the order~$m$ of~$\ab{G}$ is neither
square-free ($k_1\geq 2$) nor the square of a prime ($n\geq 2$) in this
case. Again
Lemma~\ref{lem:corollario_quanti_cloni_equationally_additive_Z_m}\eqref{item:lem:corollario_quanti_cloni_equationally_additive_Z_m_item_npquadro_q}
shows that the result claimed by the theorem is true.
\par
Second, we suppose that there is a prime~$p$ in the representation
of~$\ab{G}$ with a cyclic factor~$\ab{Z}_{p^{l}}$ where $l\geq 2$;
hence~$m$ is not square-free.
If that factor is the only one in the representation, then
$\ab{G}\cong\ab{Z}_{p^l}$. The case where $l=2$ and $m=p^2$ is
solved by
Lemma~\ref{lem:corollario_quanti_cloni_equationally_additive_Z_m}\eqref{item:lem:corollario_quanti_cloni_equationally_additive_Z_m_item_p_quadro},
and the case where $l\geq 3$, $m=p^l$, is handled by
Lemma~\ref{lem:corollario_quanti_cloni_equationally_additive_Z_m}\eqref{item:lem:corollario_quanti_cloni_equationally_additive_Z_p_to_k}.
Now let us assume that more factors appear in the decomposition,
being either cyclic groups the order of which is a power of the same
prime~$p$ or of another prime. This means there are a prime~$q$,
not necessarily distinct from~$p$, an exponent $k\geq 1$ and an
Abelian
group~$\ab{G}'$ such that
$\ab{G}\cong\ab{Z}_{p^l}\times\ab{G}'\times\ab{Z}_{q^k}$.
Then the order~$m$ is neither square-free, nor the square of a
single prime. Moreover,
$\Braket{\set{p^{l-2}}}\times\set{0_{\ab{G}'}}$ is an Abelian subgroup
of order~$p^2$
of $\ab{Z}_{p^l}\times\ab{G}'$, and
Lemma~\ref{lem:corollario_quanti_cloni_equationally_additive_Z_m}\eqref{item:lem:corollario_quanti_cloni_equationally_additive_Z_m_item_npquadro_q}
shows that the number of equationally additive clones containing
$\POL\ab{G}$ is~$\aleph_0$.
\end{proof}

\section{Characterization of equationally additive Boolean clones}\label{section:equatio_add_boolean_clones}
In this section we shall describe which clones from Post's lattice are
equationally additive (see also Figure~\ref{fig:eqn-add-Boolean-clones}).
This hence answers which algebras on the set~$\set{0,1}$ are
equational domains.
From Theorem~\ref{teor:eq_additive_iff_delta_four_algebraic} we know
that equational additivity is equivalent to~$\deltarelation$
being algebraic. For the two-element set we shall see
that we can get along with a ternary relation instead
of~$\deltarelation[\set{0,1}]$.
\begin{lemma}\label{lem:deltathree_pp_equiv_deltafour}
For any set~$A$ the relations~$\deltarelation$ and
\[\deltathree\defeq\Set{(x_1,x_2,x_3)\in A^3| x_1=x_2\lor x_2=x_3}\]
are primitive positively definable from each other, namely for all
$x_1,\dotsc,x_4\in A$ we have
\begin{align*}
(x_1,x_2,x_3)\in \deltathree&\Leftrightarrow
(x_1,x_2,x_2,x_3)\in\deltarelation,\\
(x_1,x_2,x_3,x_4)\in\deltarelation&\Leftrightarrow
\exists y_1,y_2\in A\colon
\begin{aligned}[t]
&(x_1,x_2,y_1)\in\deltathree\land
(y_1,x_3,x_4)\in\deltathree\land{}\\
&(x_2,x_1,y_2)\in\deltathree\land
(y_2,x_3,x_4)\in\deltathree.
\end{aligned}
\end{align*}
\end{lemma}
\begin{proof}
It is obvious that~$\deltathree$ can be obtained by identifying
arguments in~$\deltarelation$. For the second equivalence, take any
$(x_1,x_2,x_3,x_4)\in\deltarelation$. If $x_1=x_2$, then we let
$y_1=y_2=x_3$, and the right-hand side is satisfied. If $x_1\neq x_2$,
then $x_3=x_4$ because $(x_1,x_2,x_3,x_4)\in\deltarelation$, and in
this case we let $y_1=x_2$ and $y_2=x_1$ to satisfy the right-hand
side. Now conversely, suppose that there are elements $y_1,y_2\in A$
such that
$(x_1,x_2,y_1),(y_1,x_3,x_4),(x_2,x_1,y_2),(y_2,x_3,x_4)\in\deltathree$.
In order to get a contradiction, we assume that
$(x_1,x_2,x_3,x_4)\notin \deltarelation$, that is,
$x_1\neq x_2$ and $x_3\neq x_4$.
From the definition of~$\deltathree$ it follows that
$x_2=y_1$ and $x_1=y_2$, and $y_1=x_3$ and $y_2=x_3$, wherefore
$x_1=y_2 = x_3 = y_1 = x_2$, contradicting the choice of~$x_1$
and~$x_2$.
\end{proof}
The following is a folklore fact from clone theory.
\begin{corollary}[{cf.~\cite[Lemma~6.1.17]{Bod21}
and~\cite[Lemma~1.3.1]{PosKal79}}]\label{cor:Pol_delta_ess_at_most_unary}
For any set~$A$ we have
\[\Pol\set{\deltathree} = \Pol\set{\deltarelation} =
\Clo\mleft(A;A^A\mright),\]
i.e., the polymorphism clone of~$\deltathree$ coincides with that
of~$\deltarelation$, which is the clone of all essentially at most unary
operations.
\end{corollary}
\begin{proof}
The first equality follows directly from
Lemma~\ref{lem:deltathree_pp_equiv_deltafour}, the second one
is proved in~\cite[Lemma~1.3.1a)]{PosKal79}.
\end{proof}

Note that in the context of $A=\set{0,1}$ the relation~$\deltathree$
has become known in computer science under the pseudonym
$\dup=\set{0,1}^3\setminus\set{(0,1,0),(1,0,1)}$ \cite[\tablename~1,
p.~61]{BoehlerReithSchnoorVollmerBasesBooleanCoclones},
the polymorphism clone of which is the clone~$\Bcl{N}$ generated
by all
unary operations.
\par
The Ma\v{l}cev condition considered in the following lemma will appear
again in the characterization of the equationally additive Boolean
clones in Theorem~\ref{thm:char-Boolean-eqn-additive}.
\begin{lemma}\label{lem:malcev-cond-implying-CD}
Any variety~$\variety{V}$ admitting the Ma\v{l}cev condition
\begin{align*}
      f(x,x,y)&\approx x\approx f(x,y,x),\\
      f(y,x,x)&\approx f(x,y,f(y,x,x))
\end{align*}
is congruence distributive.
\end{lemma}
\begin{proof}
By assumption there is a ternary term~$f$ in the language
of~$\variety{V}$ such that the above identities are universally
satisfied in~$\variety{V}$. Based on~$f$ we can define the following
five ternary terms over the language of~$\variety{V}$ by substitution:
\begin{align*}
f_0(x,y,z)&\defeq x,\\
f_1(x,y,z)&\defeq f(x,y,f(z,x,x)),\\
f_2(x,y,z)&\defeq f(z,x,x),\\
f_3(x,y,z)&\defeq f(z,x,y), \text{ and}\\
f_4(x,y,z)&\defeq z.
\end{align*}
These form a sequence of J\'{o}nsson terms for~$\variety{V}$:
The equations $f_i(x,y,x)\approx x$ for $0\leq i\leq 4$ follow
from the
identity $f(x,x,y)\approx x \approx f(x,y,x)$, as does the condition
${f_0(x,x,y)\approx x\approx f(x,x,f(y,x,x))\approx f_1(x,x,y)}$.
The subsequent condition $f_1(x,y,y)\approx f(x,y,f(y,x,x))\approx
f(y,x,x)\approx f_2(x,y,y)$ follows from the second part of the
Ma\v{l}cev condition, while
$f_2(x,x,y)\approx f(y,x,x) \approx f_3(x,x,y)$ is
trivially fulfilled. Moreover, $f_3(x,y,y)\approx f(y,x,y)\approx y
\approx f_4(x,y,y)$ follows from the first part of the assumed
Ma\v{l}cev condition. Since the J\'{o}nsson identities for
$f_0,\dotsc,f_4$ hold in~$\variety{V}$, the variety is congruence
distributive, see~\cite[Theorem~12.6]{BurSan81}.
\end{proof}

In~\cite{TotWal17}, T\'oth and Waldhauser
explore necessary conditions for a relation
to be the solution set of finitely many equations from a given
clone~$\funset{C}$. Since the complement of a finitary relation on a
finite set is finite, every algebraic set that can be given as the
solution set of an infinite system of $\funset{C}$-equations, can also
be described by a finite subset of these equations: we use one
equation to exclude each point of the complement (cf.\
Lemma~\ref{lemma:equivalent_definition_algebraic_set}). Hence, on a
finite set~$A$, the solution sets from~\cite{TotWal17} are exactly the
algebraic sets in our sense.
T\'oth and Waldhauser investigate whether a relation is algebraic
for~$\funset{C}$ in terms of the \emph{centralizer} clone
$\funset{C}^{*}=\bigcup_{n\in\N}\Hom(\algop{A}{\funset{C}}^n,\algop{A}{\funset{C}})$,
consisting of all functions that commute with all the operations
in~$\funset{C}$.
With respect to the Boolean domain, T\'oth and
Waldhauser prove a characterization that can be rephrased in our
terminology as follows:
\begin{theorem}[{\cite[Theorem~4.1]{TotWal17}}]\label{thm:Alg-F-is-Inv-cent-F}
For every clone~$\funset{F}$ on the set~$\set{0,1}$
we have
$\Alg\funset{F} = \Inv(\funset{F}^{*})$, where~$\funset{F}^{*}$ is the
centralizer clone of~$\funset{F}$.
\end{theorem}

It follows from Theorem~\ref{thm:Alg-F-is-Inv-cent-F} that
$\Alg\funset{F}^{**} = \Inv(\funset{F}^{***}) =
\Inv(\funset{F}^{*})=\Alg\funset{F}$
for every Boolean clone~$\funset{F}$, since the tricentralizer and the
centralizer of a set of operations coincide. Thus, to determine
whether a Boolean
clone is equationally additive, it suffices to consider its bicentral
closure~$\funset{F}^{**}$; in other words, considering all Boolean
centralizer clones provides the complete picture. There are
precisely 25
centralizer clones on~$\set{0,1}$. They were originally presented by
Kuznecov~\cite[p.~27]{KuznecovCentralisers1979}, but the arguments
given there remain rather sketchy. A complete description can be found
in~\cite{Her08}.

The following theorem characterizes which Boolean
clones are
equationally additive. The result is illustrated in
Figure~\ref{fig:eqn-add-Boolean-clones}, where also the identifiers
for
Boolean clones used in the theorem are clarified.
With the exception of the top clone~$\funset{O}_2$, we will denote
Boolean clones by the standard symbols given in~\cite[\figurename~2,
p.~8]{CreignouVollmerBooleanCSPsWhenDoesPostsLatticeHelp}, where
explicit generating systems are listed, too.

\begin{figure}
{\centering
\input{postlat}\par
\noindent
\begin{tabular}{lll}
\tikz{\node[clone,type1] at (0,0){};} clones of TCT-type~$\tp{1}$&
\tikz{\node[clone,type2] at (0,0){};} clones of TCT-type~$\tp{2}$&
\tikz{\node[clone,type5] at (0,0){};} clones of TCT-type~$\tp{5}$\\
\tikz{\node[clone,type3] at (0,0){};} clones of TCT-type~$\tp{3}$&
\tikz{\node[clone,type4] at (0,0){};} clones of TCT-type~$\tp{4}$\\
\multicolumn{2}{l}{%
\tikz{\node[clone,type3] at (0,0){};}/\tikz{\node[clone,type4] at
(0,0){};}
equationally additive clones}\\
\multicolumn{2}{l}{%
\tikz{\node[clone,type4,extremal] at (0,0){};}
minimal equationally additive clones}\\
\multicolumn{2}{l}{%
\tikz{\node[clone,type5,extremal] at (0,0){};}/%
\tikz{\node[clone,type2,extremal] at (0,0){};}
maximal non-equationally additive clones}
\end{tabular}
}
\caption{Lattice of Boolean clones labelled according
to~\cite{CreignouVollmerBooleanCSPsWhenDoesPostsLatticeHelp};
         equationally additive clones forming the order filter
         shown by the completely filled nodes}
\label{fig:eqn-add-Boolean-clones}
\end{figure}

\begin{theorem}\label{thm:char-Boolean-eqn-additive}
Let~$\funset{F}$ be a clone on~$\set{0,1}$ and
${g,h,p,t,t^{\partial}\colon \set{0,1}^3\to \set{0,1}}$ be the ternary
Boolean operations given for $x,y,z\in\set{0,1}$ by the following
rules:
\begin{align*}
h(x,y,z)&\defeq (x\land y)\lor (x\land z)\lor (y\land z),\\
g(x,y,z)&\defeq (x + y + z)\bmod 2,\\
p(x,y,z)&\defeq (x\land z) \lor
                (x\land \negate{y}\land \negate{z})\lor
                (\negate{x}\land \negate{y}\land z)
                \text{ where }
                \negate{x}=(1+x)\bmod 2,\\
t(x,y,z)&\defeq x\lor(y\land z),\\
t^{\partial}(x,y,z)&\defeq x\land(y\lor z),
\end{align*}
that is, $h$ is the Boolean majority operation, $g$ the
Boolean minority (Ma\v{l}cev) operation and~$p$ the Pixley operation.
Then the following facts are equivalent.
\begin{enumerate}[\upshape(a)]
\item\label{item:Boolean-clone-eqn-add}
      $\funset{F}$ is equationally additive.
\item\label{item:delta-four-algebraic}
      $\deltarelation[{\set{0,1}}]\in \Alg\ari{4}\funset{F}$.
\item\label{item:delta-three-algebraic}
      $\deltathree[{\set{0,1}}]\in \Alg\ari{3}\funset{F}$.
\item\label{item:cent-ess-unary}
      $\funset{F}^{*}\subseteq \Bcl{N}$,
      that is, the centralizer of~$\funset{F}$ is essentially at most
      unary.
\item\label{item:bicent-selfdual-conservative}
      $\Bcl{D}_1\subseteq\funset{F}^{**}$, that is, the bicentralizer
      of~$\funset{F}$ contains all self-dual conservative operations.
\item\label{item:bicent-minority-majority}
      $g,h\in\funset{F}^{**}$.
\item\label{item:bicent-pixley}
      $p\in\funset{F}^{**}$.
\item\label{item:minimal-elements}
      $\Bcl{D}_2\subs\funset{F}$ or
      $\Bcl{S}_{00}\subs\funset{F}$ or
      $\Bcl{S}_{10}\subs\funset{F}$.
\item\label{item:minimal-elements-generators}
      $h\in\funset{F}$ or
      $t\in\funset{F}$ or
      $t^{\partial}\in\funset{F}$.
\item\label{item:maximal-elements-complement}
      Neither $\funset{F}\subs\Bcl{E}$ nor
      $\funset{F}\subs\Bcl{V}$, nor $\funset{F}\subs\Bcl{L}$.
\item\label{item:minimal-elements-Malcev-cond}
      There is $f\in\funset{F}\arii{3}$ satisfying
      the Ma\v{l}cev condition
      $f(x,x,y)\approx x\approx f(x,y,x)$ and
      $f(y,x,x)\approx f(x,y,f(y,x,x))$.
\item\label{item:TCT-types}
      The algebra $\algop{\set{0,1}}{\funset{F}}$ is of
      TCT-type~$\tp{3}$ (Boolean algebra)
      or~$\tp{4}$ (Boolean lattice).
\item\label{item:forbidden-TCT-types}
      The algebra $\algop{\set{0,1}}{\funset{F}}$ is not of
      TCT-types~$\tp{1}$
      (group action), $\tp{2}$ (vector space) or~$\tp{5}$
      (semilattice).
\item\label{item:CD-variety}
      The algebra $\algop{\set{0,1}}{\funset{F}}$ generates a
      congruence distributive variety.
\end{enumerate}
\end{theorem}
\begin{proof}
Points~\eqref{item:Boolean-clone-eqn-add}
and~\eqref{item:delta-four-algebraic} are equivalent by
Theorem~\ref{teor:eq_additive_iff_delta_four_algebraic}.
Let~$\rho$ be one among $\deltarelation[{\set{0,1}}]$
or~$\deltathree[{\set{0,1}}]$.
By Theorem~\ref{thm:Alg-F-is-Inv-cent-F} we have
$\rho\in\Alg{\funset{F}}=\Inv\funset{F}^{*}$ if and only if
$\funset{F}^{*}\subs \Pol\set{\rho} = \Bcl{N}$, where the last
equality
follows from Corollary~\ref{cor:Pol_delta_ess_at_most_unary}. Hence,
each of~\eqref{item:delta-four-algebraic}
and~\eqref{item:delta-three-algebraic} is equivalent
to~\eqref{item:cent-ess-unary}. The latter is certainly equivalent to
$\funset{F}^{**}\supseteq\Bcl{N}^{*} = \Bcl{D}_1$ because
the centralizer of~$\Bcl{N}$ is the centralizer of the negation and
the two
Boolean constants, that is, the intersection of the clone of self-dual
operations with the clones of zero- and one-preserving functions, in
other words, the clone~$\Bcl{D}_1$ of self-dual conservative
operations.
Thus, \eqref{item:cent-ess-unary}
and~\eqref{item:bicent-selfdual-conservative} are equivalent.
Since~$\Bcl{D}_1$ is generated by $\set{g,h}$ (it is the join of the
minimal clones~$\Bcl{L}_2$ and~$\Bcl{D}_2$ generated by~$g$ and~$h$,
respectively) or~$\set{p}$
(cf.~\cite[\figurename~2,
p.~8]{CreignouVollmerBooleanCSPsWhenDoesPostsLatticeHelp}),
statement~\eqref{item:bicent-selfdual-conservative} is equivalent to
each of~\eqref{item:bicent-minority-majority}
and~\eqref{item:bicent-pixley}.
Now for $\funset{G}=\Bcl{D}_2$ the
least centralizer clone above~$\funset{G}$ is
$\funset{G}^{**}=\Bcl{D}_1$, for
$\funset{G}\in\set{\Bcl{S}_{00},\Bcl{S}_{10}}$, it is the
clone~$\funset{G}^{**}=\Pol\set{\set{0},\set{1}}\supseteq\Bcl{D}_1$
of conservative
operations, cf.~\cite[\figurename~5, p.~3158]{Her08}. Therefore, from
$\funset{G}\subs\funset{F}$, i.e.~\eqref{item:minimal-elements},
we obtain
$\Bcl{D}_1\subs\funset{G}^{**}\subs\funset{F}^{**}$,
i.e.~\eqref{item:bicent-selfdual-conservative}.
If~$\funset{F}$ does not satisfy~\eqref{item:minimal-elements}, then,
according to Post's lattice, there is
$\funset{G}\in\set{\Bcl{E},\Bcl{V},\Bcl{L}}$ such that
$\funset{F}\subs\funset{G}$. From~\cite[\figurename~5]{Her08} we see
that~$\funset{G}$ is a centralizer clone, wherefore
$\funset{F}^{**}\subs\funset{G}^{**} = \funset{G}$. This means
that~\eqref{item:bicent-selfdual-conservative} fails, as
$\Bcl{D}_1\not\subs\funset{G}$, and hence,
\eqref{item:bicent-selfdual-conservative}
and~\eqref{item:minimal-elements} are equivalent. Moreover,
\eqref{item:minimal-elements}
and~\eqref{item:maximal-elements-complement} are
equivalent as~$\Bcl{E}$, $\Bcl{V}$ and~$\Bcl{L}$ are the maximal
elements in the complement of the order-filter of equationally
additive clones in Post's lattice described by its minimal elements
in~\eqref{item:minimal-elements}.
Furthermore, we infer from~\cite[\figurename~2,
p.~8]{CreignouVollmerBooleanCSPsWhenDoesPostsLatticeHelp} that the
clones~$\Bcl{D}_2$, $\Bcl{S}_{00}$ and~$\Bcl{S}_{10}$ are generated by
the Boolean majority operation~$h$, $t$ and~$t^{\partial}$,
respectively. Therefore, condition~\eqref{item:minimal-elements} is
equivalent to~\eqref{item:minimal-elements-generators}.
\par
We have now established that
statements~\eqref{item:Boolean-clone-eqn-add}--\eqref{item:maximal-elements-complement} are all equivalent.
As our next step we shall show
that~\eqref{item:minimal-elements-generators}
and~\eqref{item:minimal-elements-Malcev-cond} are equivalent.
For this let us first
assume the truth of~\eqref{item:minimal-elements-generators} and
let~$f$ be~$h$, $t$ or~$t^{\partial}$, respectively. If $f=h$, then
$f(x,x,y)\approx x\approx f(x,y,x)$ and
$f(y,x,x)\approx x\approx f(x,y,x)\approx f(x,y,f(y,x,x))$ are
trivial.
Otherwise, the conditions stated
in~\eqref{item:minimal-elements-Malcev-cond} follow from the
idempotence, commutativity and absorption laws for lattices; for
example, for $f=t$ we have
$f(y,x,x)\approx y\lor x\approx x\lor y\approx x\lor (y\land
(y\lor x))
\approx f(x,y,f(y,x,x))$, and dually for $f=t^{\partial}$.
Conversely, if we have a ternary operation~$f$ on~$\{0,1\}$ subject
to~\eqref{item:minimal-elements-Malcev-cond}, then the equation
$f(x,x,y)\approx x\approx f(x,y,x)$ uniquely determines the values
of~$f$ on~$6$ out of the~$8$ argument triples. Hence, there are in total
$2^{8-6}=4$ possible ternary Boolean operations~$f$ satisfying
this condition.
These are $h,t,t^{\partial}$ and~$\eni[3]{1}$, however,
by the second part of~\eqref{item:minimal-elements-Malcev-cond},
$f$ cannot be the first projection,
wherefore~\eqref{item:minimal-elements-generators} follows.
\par
Subsequently we will prove that~\eqref{item:minimal-elements} implies~\eqref{item:TCT-types} and (by its contrapositive)
that~\eqref{item:forbidden-TCT-types}
implies~\eqref{item:maximal-elements-complement}. Since, clearly,
\eqref{item:TCT-types} and~\eqref{item:forbidden-TCT-types} are
equivalent, this will show that all the statements~\eqref{item:Boolean-clone-eqn-add}--\eqref{item:forbidden-TCT-types} are equivalent.
Finally, we will show that~\eqref{item:minimal-elements-Malcev-cond}
implies~\eqref{item:CD-variety} and, by contradiction,
that~\eqref{item:CD-variety}
implies~\eqref{item:maximal-elements-complement}, and the proof will be
finished.
Let us also note that the equivalence of
statements~\eqref{item:minimal-elements} and~\eqref{item:CD-variety},
which appears as a part of our theorem is already known from the
literature; to our knowledge it was first proved
in~\cite[Proposition~2.1]{AglBak99}.
\par

To prove that~\eqref{item:minimal-elements}
implies~\eqref{item:TCT-types}, let
$\funset{G}\in\set{\Bcl{D}_2,\Bcl{S}_{00},\Bcl{S}_{10}}$ and suppose
$\funset{G}\subs\funset{F}$. Then the polynomial expansion
of~$\funset{G}$ obtained by joining the Boolean clone~$\Bcl{I}$ of all
constant operations is the maximal Boolean clone of monotone
operations
$\Bcl{M}=\funset{G}\vee \Bcl{I}\subs\funset{F}\vee
\Bcl{I}$. Therefore,
$\funset{F}\vee\Bcl{I}\in\set{\Bcl{M},\funset{O}_2}$, and the TCT-type
of~$\funset{F}$ is~$\tp{4}$ if $\funset{F}\vee\Bcl{I}=\Bcl{M}$,
or~$\tp{3}$ if
$\funset{F}\vee\Bcl{I}=\funset{O}_2$. This shows
that~\eqref{item:minimal-elements} implies~\eqref{item:TCT-types}.
Conversely, let us assume the negation
of~\eqref{item:maximal-elements-complement}, that is, that
$\funset{F}\subs\funset{G}$ for some
$\funset{G}\in\set{\Bcl{V},\Bcl{E},\Bcl{L}}$. Then
$\funset{G}\supseteq\Bcl{I}$, wherefore
$\funset{F}\vee\Bcl{I}\subs\funset{G}\vee\Bcl{I} = \funset{G}$. Hence,
$\funset{F}$ is polynomially equivalent to a semilattice, a vector
space
over $\GF(2)$, or---if it is essentially at most unary---a group
action, which is the negation of~\eqref{item:forbidden-TCT-types}.
\par

Lastly, to show that~\eqref{item:minimal-elements-Malcev-cond}
implies~\eqref{item:CD-variety} let $f\in\funset{F}\arii{3}$ be
an operation as claimed in~\eqref{item:minimal-elements-Malcev-cond}
and let~$\variety{V}$ be the variety generated by
$\algop{\set{0,1}}{\funset{F}}$. Due to~\eqref{item:minimal-elements-Malcev-cond}
the $\variety{V}$-term $f(x,y,z)$ shows that~$\variety{V}$ admits the
Ma\v{l}cev condition from Lemma~\ref{lem:malcev-cond-implying-CD},
hence~$\variety{V}$ is congruence distributive, i.e.,
\eqref{item:CD-variety} holds.
For the converse, we assume
now~\eqref{item:CD-variety} together with the negation
of~\eqref{item:maximal-elements-complement}, which would imply
that~$\funset{F}$ and hence one of~$\Bcl{E}$, $\Bcl{V}$ or~$\Bcl{L}$
would have a sequence of J\'{o}nsson operations. Therefore,
$\algop{\{0,1\}}{\land, \ca[0], \ca[1]}$, $\algop{\{0,1\}}{\lor, \ca[0], \ca[1]}$ or
$\algop{\{0,1\}}{+, \ca[0], \ca[1]}$ would generate a congruence distributive
variety, which is false, as in each case the congruence lattice of the
square of the respective algebra already fails to be distributive.
This contradiction shows that~\eqref{item:CD-variety}
entails~\eqref{item:maximal-elements-complement}.
\end{proof}

The equivalence of statements~\eqref{item:Boolean-clone-eqn-add}
and~\eqref{item:TCT-types} of
Theorem~\ref{thm:char-Boolean-eqn-additive} will be widely used in the
subsequent sections.
\begin{corollary}\label{cor:boolean_clones_eq_add_TCT_type_3_4}
Let~$\ab{A}$ be an algebra on a two-element set. Then
$\Clo\ab{A}$ is equationally additive if and only if\/
$\Type{\ab{A}}\in\{\tp{3},\tp{4}\}$.
\end{corollary}

Knowing that the Boolean clones~$\Bcl{D}_2$, $\Bcl{S}_{00}$,
$\Bcl{S}_{10}$ and all clones above them are equationally additive,
Theorem~\ref{teor:eq_additive_iff_delta_four_algebraic} tells
that~$\deltarelation[{\set{0,1}}]$ is an algebraic set, hence
definable as a solution set of some system of equations. In the
following remark, we exhibit an explicit system of equations
defining~$\deltarelation[\set{0,1}]$.

\begin{remark}\label{rem:def-Delta-D2-S00-S10}
The clone~$\Bcl{D}_2$ is generated by the Boolean majority
operation~$h$, and every clone in the principal filter generated by
this clone is equationally additive, since
for all $x_1,x_2,x_3,x_4\in\set{0,1}$ we have
(cf.~\cite{AicBehRos23DatasetSystems4EqnAdd})
\[
(x_1,x_2,x_3,x_4)\in\deltarelation[\set{0,1}]\iff
h(x_3,x_4,x_1)=h(x_3,x_4,x_2).
\]
\par
With respect to the clones~$\Bcl{S}_{00}$ and~$\Bcl{S}_{10}$, we infer
from~\cite[\figurename~2,
p.~8]{CreignouVollmerBooleanCSPsWhenDoesPostsLatticeHelp}
that they are generated by the ternary functions~$t$ and~$t^{\partial}$
(cf.\ Theorem~\ref{thm:char-Boolean-eqn-additive}), respectively, which
are given for $x,y,z\in\set{0,1}$ by the rules $t(x,y,z)=x\lor(y\land z)$ and
$t^{\partial}(x,y,z)=x\land (y\lor z)$.
The clones~$\Bcl{S}_{00}$ and~$\Bcl{S}_{10}$ are dual to each other,
and for both of them, i.e., for $\tau\in\set{t,t^{\partial}}$, we have
for all $x_1,x_2,x_3,x_4\in\set{0,1}$ that
(cf.~\cite{AicBehRos23DatasetSystems4EqnAdd})
\begin{align*}
(x_1,x_2,x_3,x_4)\in\deltarelation[\set{0,1}]\iff
\tau(x_3,x_4,x_1)&=\tau(x_3,x_4,x_2),\text{ and}\\
\tau(x_4,x_3,x_1)&=\tau(x_4,x_3,x_2).
\end{align*}
\par
Computing the four-generated free algebra in the variety generated
by the algebra $\ab{A}=\algop{\set{0,1}}{\tau}$, we find that there are
exactly~$53$ quaternary term operations of~$\ab{A}$,
cf.~\cite{AicBehRos23DatasetSystems4EqnAdd}.
One can check that for every
pair of quaternary term operations $f,g\in(\Bcl{S}_{00})\arii{4}$
that agree on the~$12$
quadruples in~$\deltarelation[\set{0,1}]$, they also agree on at
least one of
the four elements of $\set{0,1}^4\setminus\deltarelation[\set{0,1}]$,
see also~\cite{AicBehRos23DatasetSystems4EqnAdd}.
It is hence impossible to define~$\deltarelation[\set{0,1}]$ by a
single equation of the form $f(x_1,x_2,x_3,x_4)=g(x_1,x_2,x_3,x_4)$
over~$\Bcl{S}_{00}$.
\end{remark}

In the following, we prove that the characterization of equational
domains given in Corollary~\ref{cor:boolean_clones_eq_add_TCT_type_3_4}
carries over to finite E-minimal algebras as defined
in~\cite[Definition~2.14]{HobbMcK}.
We recall that a finite algebra is \emph{E-minimal} if it has at least
two elements and every unary idempotent polynomial is constant or the
identity operation. Finite non-trivial $p$-groups provide prominent
examples of such algebras.
In~\cite[Theorem~4.32]{HobbMcK} it is proved that the prime quotients
of every E-minimal algebra all have the same type.
Hence one can associate to each E-minimal algebra one of the
five types of minimal algebras introduced in
Section~\ref{sec:preliminaries}.

We fix some notation that will only be used in the proof of the
following lemma.
For a set~$A$, $n\in\N$, $i\in\{1,\dots,n\}$, $f\colon A^n\to A$ and
$\vec{a}\in A^{n-1}$ we define the unary polynomial
$f_i^\vec{a}\colon A\to A$ by
$f_i^\vec{a}(x)=f(a_1,\dots, a_{i-1},x,a_{i},\dots,a_{n-1})$ for all
$x\in A$.
\begin{lemma}\label{lemma:Eminima_of_type_1}
Let~$\ab{A}$ be a subdirectly irreducible (finite) E-minimal algebra of
type~$\tp{1}$. Then $\Clo\ab{A}$ is not equationally additive.
\end{lemma}
\begin{proof}
Since~$\ab{A}$ is E-minimal, we have $k\defeq \crd{A}\geq 2$. Without loss of
generality, let us assume that $A=\set{1,\dotsc,k}$ and the
monolith~$\mu$ of~$\ab{A}$ has the form $\mu=\Cg{\{(1,2)\}}$.
Since~$\ab{A}$ has type~$\tp{1}$, \cite[Theorem~4.4]{Kis97} implies
that for all $n\in\N$ and for all $f\in\Clo\ari{n}\ab{A}$ (exactly) one
of the following two statements holds:
\begin{enumerate}[(1)]
\item\label{item:item1oflemmaEminimal_type1}
for each $i\in \{1,\dots,n\}$ and every $\vec{a}\in A^{n-1}$ we have
$f_i^\vec{a}(1)=f_i^\vec{a}(2)$, or
\item\label{item:item2oflemmaEminimal_type1}
there is $j\in \{1,\dots, n\}$ such that for each $\vec{a}\in A^{n-1}$
the function~$f_j^\vec{a}$ induces a permutation
on~$A$ and $f_i^\vec{a}(1)=f_i^\vec{a}(2)$ for all
$i\in\{1,\dots,n\}\setminus\{j\}$.
\end{enumerate}
Let $f,g\in\Clo\ari{4}\ab{A}$ be such that
$f\restrict{\deltarelation}=g\restrict{\deltarelation}$. We prove
that \[f(2,1,2,1)=g(2,1,2,1).\]
First we observe that for all $i\in\{1,\dots, 4\}$ if
$f_i^{(1,1,1)}(1)\neq f_i^{(1,1,1)}(2)$, then
$g_i^{(1,1,1)}(1)\neq g_i^{(1,1,1)}(2)$ and vice versa:
in fact, we have
\begin{align*}
g_i^{(1,1,1)}(1)&=g(1,1,1,1)=f(1,1,1,1)=f_i^{(1,1,1)}(1)\neq\\
f_i^{(1,1,1)}(2)&=f(1,\dots, \underset{i}{2},\dots, 1)=g(1,\dots,
\underset{i}{2},\dots, 1)=g_i^{(1,1,1)}(2).
\end{align*}
Thus, either both~$f$ and~$g$
satisfy~\eqref{item:item1oflemmaEminimal_type1} or they both
satisfy~\eqref{item:item2oflemmaEminimal_type1} with the same
$j\in\{1,\dots,n\}$.
If both~$f$ and~$g$ satisfy~\eqref{item:item1oflemmaEminimal_type1}
or both satisfy~\eqref{item:item2oflemmaEminimal_type1} with
$j\neq 1$, then we have
\begin{align*}
f(2,1,2,1)&=f_1^{(1,2,1)}(2)=f_1^{(1,2,1)}(1)=f(1,1,2,1)=g(1,1,2,1)\\
          &=g_1^{(1,2,1)}(1)=g_1^{(1,2,1)}(2)=g(2,1,2,1).
\end{align*}
If both~$f$ and~$g$ satisfy~\eqref{item:item2oflemmaEminimal_type1}
with $j=1$, then we have
\begin{align*}
f(2,1,2,1)&=f_3^{(2,1,1)}(2)=f_3^{(2,1,1)}(1)=f(2,1,1,1)=g(2,1,1,1)\\
          &=g_3^{(2,1,1)}(1)=g_3^{(2,1,1)}(2)=g(2,1,2,1).
\end{align*}
This concludes the proof of the fact that~$\deltarelation$ is not an
algebraic set. Therefore,
Theorem~\ref{teor:eq_additive_iff_delta_four_algebraic} yields that
$\Clo\ab{A}$ is not equationally additive.
\end{proof}
\begin{lemma}\label{lemma:E_minimal_cd}
Let~$\ab{A}$ be a (finite) E-minimal algebra.
Then~$\ab{A}$ generates a congruence distributive variety
if and only if~$\ab{A}$ has TCT-type~$\tp{3}$ or~$\tp{4}$.
\end{lemma}
\begin{proof}
If~$\ab{A}$ is of type~$\tp{3}$ or~$\tp{4}$,
then~\cite[Lemma~4.29]{HobbMcK} yields that $\crd{A}=2$
and the result follows from the equivalence
of~\eqref{item:TCT-types} and~\eqref{item:CD-variety}
in Theorem~\ref{thm:char-Boolean-eqn-additive}.
If~$\ab{A}$ is of type~$\tp{1}$, $\tp{2}$, or~$\tp{5}$,
then~\cite[Theorem~8.6]{HobbMcK} yields that~$\ab{A}$ does
not generate a congruence distributive variety.
\end{proof}
\begin{lemma}\label{lemma:the_TCT_type_of_E_minimal_algebras}
Let~$\ab{A}$ be a (finite) E-minimal algebra. Then $\Clo\ab{A}$ is
equationally additive if and only if~$\ab{A}$ has type~$\tp{3}$
or~$\tp{4}$.
\end{lemma}
\begin{proof}
$\ab{A}$ being E-minimal implies $\crd{A}\geq 2$.
If~$\ab{A}$ has type~$\tp{3}$, $\tp{4}$ or~$\tp{5}$,
then~\cite[Lemma~4.29]{HobbMcK} yields that $\crd{A}=2$
and the equivalence follows from
Corollary~\ref{cor:boolean_clones_eq_add_TCT_type_3_4}.
\par
The opposite case is that~$\ab{A}$ is a finite E-minimal algebra of
type~$\tp{1}$ or~$\tp{2}$. This contradicts~$\ab{A}$ having type~$\tp{3}$
or~$\tp{4}$, hence, to fulfil the stated equivalence, we have to prove
that $\Clo\ab{A}$ fails to be equationally additive.
Since $2\leq\crd{A}<\aleph_{0}$,
Proposition~\ref{prop:nec_condi_fin_sub_irre}
implies that~$\ab{A}$ is subdirectly irreducible. If~$\ab{A}$ has
type~$\tp{1}$, then Lemma~\ref{lemma:Eminima_of_type_1} directly states
that $\Clo\ab{A}$ is not equationally additive. Therefore, the case
that is still to be discussed is that of a (finite non-trivial)
subdirectly irreducible E-minimal algebra~$\ab{A}$ of type~$\tp{2}$.
Let~$\mu$ be its monolith.
Now~\cite[Theorem~13.9]{HobbMcK} implies that~$\ab{A}$ is Ma\v{l}cev,
and by~\cite[Theorem~4.32(2)]{HobbMcK} all its prime quotients have
type~$\tp{2}$. In particular, we have $\type{\bottom{A}}{\mu}=\tp{2}$,
and hence~\cite[Theorem~5.7(3)]{HobbMcK}
yields $[\mu,\mu]=\bottom{A}$, i.e., that~$\mu$ is Abelian.
Thus, by Corollary~\ref{cor:nec_condi_malcev_polynomial_and_finite_yields_sub_irreducible_and_type_three},
$\POL\ab{A}$ cannot be equationally additive; therefore, by
Corollary~\ref{cor:equational_add_and_inclusion}, $\Clo\ab{A}$ cannot
be either.
\end{proof}
\begin{theorem}\label{teor:the_TCT_type_of_E_minimal_algebras}
For a (finite) E-minimal algebra~$\ab{A}$ the following statements are
equivalent:
\begin{enumerate}[\upshape(a)]
\item\label{item_eq_additivity:teorthe_TCT_type_of_E_minimal_algebras}
      $\Clo\ab{A}$ is equationally additive;
\item\label{item_type:teorthe_TCT_type_of_E_minimal_algebras}
      $\ab{A}$ is of type~$\tp{3}$ or~$\tp{4}$;
\item\label{item_variety:teorthe_TCT_type_of_E_minimal_algebras}
      $\ab{A}$ generates a congruence distributive variety.
\end{enumerate}
\end{theorem}
\begin{proof}
The equivalence of~\eqref{item_eq_additivity:teorthe_TCT_type_of_E_minimal_algebras} and~\eqref{item_type:teorthe_TCT_type_of_E_minimal_algebras} follows from Lemma~\ref{lemma:E_minimal_cd}.
The equivalence of~\eqref{item_type:teorthe_TCT_type_of_E_minimal_algebras}
and~\eqref{item_variety:teorthe_TCT_type_of_E_minimal_algebras}
follows from Lemma~\ref{lemma:the_TCT_type_of_E_minimal_algebras}.
\end{proof}

\section{Characterization of the equationally additive clones of self-dual operations}
Let $A=\{0,1,2\}$ and let the permutation $\zeta_3=(0\,1\,2)$ be the cyclic shift of
the three elements of~$A$.
An operation $f\colon A^n\to A$ with $n\in\N$ is called \emph{self-dual}
if $f\in\set{\zeta_3}^{*}$, that is, if it commutes with~$\zeta_3$,
in other words, if~$\zeta_3$ is an automorphism of the algebra $\algop{A}{f}$.
The ideal of the lattice of clones on the three-element set~$A$ generated
by the centralizer clone~$\set{\zeta_3}^{*}$ is fully described in
\cite[\figurename~2, p.~260]{Zhu15}.
In the present section we will stay with the notation introduced
in~\cite[Section~1]{Zhu15}, and we will describe all equationally
additive clones of self-dual operations on~$A$.

Let $\fipii\colon A^3\to A$ be defined as follows:
For each $\vec{x}=(x_1,x_2,x_3)\in A^3$ let
\[
\fipii(\vec{x})=
\begin{cases}
x_2  & \text{ if } \vec{x}\in\set{(0,1,1),(1,2,2),(2,0,0)},\\
x_1 &\text{ otherwise.}
\end{cases}
\]
With a quick glance at its operation table
(cf.\ also~\cite{AicBehRos23DatasetSystems4EqnAdd}), one verifies that this
operation coincides with the function introduced under the same name
in~\cite[p.~265]{Zhu15}. According to~\cite[Theorem~8, p.~266]{Zhu15},
the operation~$\fipii$ generates the clone~$\aipii$, defined on
page~261 of~\cite{Zhu15}. The dual~$\Aipii$ of this clone with respect
to the transposition $\sigma\colon A\to A$ switching~$0$ and~$1$
(cf.~\cite[pp.~255, 259, 261]{Zhu15}) is given by applying this switch
to every tuple of every relation defining $\aipii$ as a polymorphism
clone. It follows from this that~$\Aipii$ arises as an isomorphic copy
of~$\aipii$ by conjugating every operation in~$\aipii$ using the
transposition~$\sigma$. As a consequence~$\Aipii$ is generated
by $(\fipii)^{*}\colon A^3\to A$, given for all
$\vec{x}=(x_1,x_2,x_3)\in A^3$ by
\begin{align*}
(\fipii)^{*}(\vec{x})
&=\sigma(\fipii(\sigma^{-1}(x_1),\sigma^{-1}(x_2),\sigma^{-1}(x_3)))\\
&=
\begin{cases}
x_2  & \text{ if } \vec{x}\in\set{(1,0,0),(0,2,2),(2,1,1)},\\
x_1 &\text{ otherwise.}
\end{cases}
\end{align*}
\par

For use in the proof of Theorem~\ref{teor:self-dual_operations_char},
we observe that both~$\fipii$ and~$(\fipii)^{*}$ are idempotent, that
is, they equal the identity operation~$\id_A$ when all three arguments
are identified.

\begin{lemma}\label{lemma:aipii-Aipii-eqn-add}
All clones on $A=\set{0,1,2}$ containing~$\aipii$ or~$\Aipii$
from~\cite[p.~261]{Zhu15} are equationally additive.
\end{lemma}
\begin{proof}
Let~$f\in\set{\fipii,(\fipii)^{*}}$.
Moreover, let $S\subseteq A^4$ be the solution set of the following
system of equations
\[
\begin{cases}
&f(x_1,x_2,x_3)\approx f(x_1, x_2,x_4)\\
&f(x_2,x_1,x_3)\approx f(x_2, x_1,x_4)\\
&f(x_3,x_4,x_1)\approx f(x_3, x_4,x_2)\\
&f(x_4,x_3,x_1)\approx f(x_4, x_3,x_2).
\end{cases}
\]
Using a computer (cf.~\cite{AicBehRos23DatasetSystems4EqnAdd}),
one readily verifies that $S=\deltarelation$, whence~$\deltarelation$
is algebraic over any clone containing~$f$.
Thus, Theorem~\ref{teor:eq_additive_iff_delta_four_algebraic} yields that
every clone containing~$f$ is equationally additive.
As, by~\cite[Theorem~8]{Zhu15} (proved as~\cite[Theorem~30,
p.~304]{Zhu15}), $\fipii$ generates~$\aipii$, and
hence~$(\fipii)^{*}$ generates $\Aipii$, the statement of the lemma follows.
\end{proof}

\begin{corollary}
\label{cor:continuously_many_self_dual_eq_additive_clones}
On a set~$A$ with $\crd{A}=3$ there are exactly~$2^{\aleph_0}$
distinct equationally additive clones of self-dual operations.
\end{corollary}
\begin{proof}
Combining the definition of~$\aipii$ on p.~261 of~\cite{Zhu15}
with~\cite[Theorem~16, p.~269]{Zhu15} (proved as Theorem~38, p.~313),
we infer that there are exactly~$2^{\aleph_0}$ distinct clones of
self-dual operations on~$\{0,1,2\}$ that contain~$\aipii$. Therefore,
the result follows from Lemma~\ref{lemma:aipii-Aipii-eqn-add} and
the fact that there are only countably many finitary operations on a
finite set, thus no more than~$2^{\aleph_0}$ subsets (clones) on~$A$.
\end{proof}

Following~\cite{Pin17a}, we say that clones~$\funset{C}$
and~$\funset{D}$ on the same set~$A$ are \emph{algebraically
equivalent}, denoted  by $\funset{C}\algeq\funset{D}$, if
$\Alg\funset{C}=\Alg\funset{D}$.
It was shown in~\cite{Yab58} that on the three-element set there are~18
maximal clones (cf.~\cite[\tablename~4, p.~111]{PosKal79}).
Following~\cite[Definition~4.3.12]{PosKal79}, we define~$\cloneL$
as the clone of polymorphisms of
$\{(a,b,c,d)\in \{0,1,2\}^4\mid a+b=c+d\mod 3\}$.
\begin{corollary}
Let $A=\{0,1,2\}$, and let~$\funset{C}$ be a maximal clone
on~$A$ that is not the clone $\Pol\{{\preceq}\}$ of monotone operations
with respect to some bounded (linear) order~$\preceq$ on~$A$.
Then the number of algebraically inequivalent subclones of\/~$\funset{C}$ is
\begin{enumerate}[\upshape(a)]
\item finite, if\/ $\funset{C}=\cloneL$, the clone of (affine) linear operations;\label{item:cor_number_alg_in_clo_linear}
\item at most countable, if\/ $\funset{C}=\{\zeta_3\}^{*}$;\label{item:cor_number_alg_in_clo_self_dual}
\item continuum, otherwise. \label{item:cor_number_alg_in_clo_JM}
\end{enumerate}
\end{corollary}
\begin{proof}
In~\cite[Theorem~15]{BagDem82} it is proved that below the clone
of linear operations on any set of prime cardinality there are only
finitely many clones at all (see also~\cite[\figurename~3,
p.~121]{BagDem82} for the case $\crd{A}=3$),
hence~\eqref{item:cor_number_alg_in_clo_linear} follows.
\par
In~\cite{Pin17a} (see also~\cite{AicRos20}) it is shown that on a
finite set there are only finitely many equationally additive
clones up
to algebraic equivalence.
Lemma~\ref{lemma:aipii-Aipii-eqn-add} proves that all clones of
self-dual operations on $\set{0,1,2}$ above~$\aipii$ or~$\Aipii$ are
equationally additive, hence split into finitely many algebraic
equivalence classes.
In~\cite{Zhu15} it is proved that there are exactly~$\aleph_0$ clones of
self-dual operations that are neither above~$\aipii$ nor~$\Aipii$,
see~\cite[\figurename~2, p.~260]{Zhu15} and the
description on page~261 of~\cite{Zhu15}.
Therefore, there are at most countably many algebraically inequivalent
clones of self-dual operations on $\{0,1,2\}$, as claimed
in~\eqref{item:cor_number_alg_in_clo_self_dual}.
\par
In~\cite[Proposition~5.4]{AicRos23} it is proved that the~$2^{\aleph_0}$
clones from~\cite[3.1.4~Haupt\-satz(ii), p.~79]{PosKal79}
are algebraically inequivalent.
In~\cite[\S~1, proof of Theorem~1]{DemHan83}
the authors show how to find a conjugate of the clones from~\cite{JanMuc59}
below each of the remaining maximal clones. Except for the case of
monotone operations with respect to some bounded order, their argument
also works for the family of clones defined
in~\cite[3.1.4~Haupt\-satz(ii), p.~79]{PosKal79}.
Hence~\eqref{item:cor_number_alg_in_clo_JM} follows.
\end{proof}

We now work towards the description of the equationally additive clones
of self-dual operations on $A=\{0,1,2\}$.
We first prove that equational additivity is hereditary with respect to
restriction of the base set.
\begin{lemma}\label{lemma:char-eqn-add-clones-inv-subsets}
A clone~$\funset{C}$ on a set~$X$ is equationally additive if and only
if for every $B\subs X$ that is invariant under~$\funset{C}$ the
restriction
$\funset{C}\restrict{B}\defeq\set{f\restrict{B}\mid f\in\funset{C}}$ is
equationally additive.
\end{lemma}
\begin{proof}
Clearly, if restrictions to invariant subsets are equationally
additive, then $\funset{C}=\funset{C}\restrict{X}$ itself is
equationally additive. For the converse let $B\subs X$ belong to
$\Inv{\funset{C}}$ and let~$\funset{C}$ be equationally additive.
By Theorem~\ref{teor:eq_additive_iff_delta_four_algebraic} there is
an index set~$I$ and there are two families
$(p_i)_{i\in I}$ and $(q_i)_{i\in I}$ of operations
from~$\funset{C}\arii{4}$ such that $\deltarelation[X]=\{\vec{x}\in
X^4\mid \forall i\in I \colon p_i(\vec{x})=q_i(\vec{x})\}$.
Then $\deltarelation[B] = \deltarelation[X]\cap B^4$ implies that
$\deltarelation[B] =\{\vec{b}\in B^4\mid \forall i\in I\colon
p_i\restrict{B}(\vec{b})=q_i\restrict{B}(\vec{b})\}$, and thus
Theorem~\ref{teor:eq_additive_iff_delta_four_algebraic} shows that
$\funset{C}\restrict{B}$ is equationally additive.
\end{proof}

The following lemma will help to show that certain clones of self-dual
operations on $A=\set{0,1,2}$ fail to be equationally additive. It can
easily be verified based on the generators of the clones provided
in~\cite[Theorems~6 and~7, p.~265 et seq.]{Zhu15}.
For the aid of the reader, we define these and a few auxiliary
operations. By $\minority_B$ and $\majority_B$ we denote the unique
ternary minority and majority operation on an at most two-element
set~$B$, respectively (on $B=\set{0,1}$ we have $\minority_B = g$ and
$\majority_B=h$ as defined in
Theorem~\ref{thm:char-Boolean-eqn-additive}).
For all $x,y,z\in A$ we set
\begin{align*}
a(x,y)&\defeq 2x+2y+1 \bmod 3\\
\pluszero(x,y,z)&\defeq \begin{cases}
\minority_{\set{x,y,z}}(x,y,z)  &\text{if }\crd{\set{x,y,z}}\leq 2\\
                   x+1 \bmod 3  &\text{if }\crd{\set{x,y,z}} = 3
\end{cases}\\
m(x,y,z)&\defeq\begin{cases}
\majority_{\set{x,y,z}}(x,y,z)  &\text{if }\crd{\set{x,y,z}}\leq 2\\
                             x  &\text{if }\crd{\set{x,y,z}} = 3
\end{cases}\\
\ps(x,y,z)&\defeq\begin{cases}
x  &\text{if }\crd{\set{x,y,z}}\leq 2\\
y  &\text{if }\crd{\set{x,y,z}} = 3
\end{cases}\\
\rght(x,y)&\defeq 2(x^2+x+xy +y +y^2)\bmod3
&\begin{array}{r|ccc}
\rght(x,y)&0&1&2\\\hline
0&0&1&0\\
1&1&1&2\\
2&0&2&2
\end{array}\\
\lft(x,y)&\defeq x^2+2x+xy +2y +y^2\bmod3
&\begin{array}{r|ccc}
\lft(x,y)&0&1&2\\\hline
0&0&0&2\\
1&0&1&1\\
2&2&1&2
\end{array}
\end{align*}
According to~\cite[Theorem~6, p.~265]{Zhu15}, $a$ generates~$\SL$, $m$
generates~$\TN$, and $\pluszero$ generates~$\Ltwo$;
moreover, by~\cite[Theorem~7, p.~266]{Zhu15}, $\set{\rght,\ps}$
generates~$\aP$. The clone~$\AP$ is the dual of~$\aP$ under the
transposition $\sigma\colon A\to A$ swapping~$0$ and~$1$
(cf.~\cite[pp.~255, 259]{Zhu15}). It is thus generated by the operations
$\ps^{*}=\ps$ and
$(\rght)^{*}
=\sigma\circ\rght\circ(\sigma^{-1}\times\sigma^{-1}\times\sigma^{-1})
=\lft$.
\par
It is evident from the given definition of the generators that all of
these clones except for~$\SL$ preserve every subset of~$\set{0,1,2}$,
i.e., that they are conservative.

\begin{lemma}\label{lemma:existence_of_type_2_and_5_prime_quotients}
Let\/~$\aP,\AP$ be defined as in~\cite[p.~259]{Zhu15},
and
let~$\Ltwo, \SL$ be defined as in~\cite[p.~256]{Zhu15}.
Then the following facts about these clones on~$A=\{0,1,2\}$ hold:
\begin{enumerate}[\upshape(a)]
\item\label{item:aPAPL2-restrict-VEL2}
  The clones~$\aP$, $\AP$ and~$\Ltwo$ have $B=\set{0,1}$
  as an invariant subset and
  $\aP\restrict{B}=\Bcl{V}_2$,
  $\AP\restrict{B}=\Bcl{E}_2$,
  $\Ltwo\restrict{B}=\Bcl{L}_2$
  (cf.\ \figurename~\ref{fig:eqn-add-Boolean-clones} for the notation).
\item\label{item:SL-polyeq-Z3}
  The clone generated by~$\SL$ and all constant operations on~$A$ is the
  clone of polynomial operations of the $\GF(3)$-vector space~$\Z_3$.
\end{enumerate}
\end{lemma}
\begin{proof}
\begin{enumerate}[(a)]
\item This follows by a brief inspection of the generating functions
      provided above:
      $\rght\restrict{B} = \lor$,
      $\lft\restrict{B}=\land$, $\ps\restrict{B} = \eni[3]{1}$ and
      $\pluszero\restrict{B}=\minority_B = g$,
      where~$g$ is the Boolean minority operation as given in
      Theorem~\ref{thm:char-Boolean-eqn-additive}.
\item For $x,y\in A$ we have
      ${a(a(x,1),a(0,y))=a(2x+3,2y+1)=x+y\bmod 3}$;
      hence~$\SL$ and the clone generated by addition modulo~$3$
      have the same constantive expansion (the same polynomial operations).
      Therefore, $\algop{A}{\SL}$ is polynomially equivalent to the
      $\GF(3)$-vector space~$\Z_3$.
      \qedhere
\end{enumerate}
\end{proof}

We are now ready to prove that the characterization of equational
additivity found to be true in
Theorem~\ref{thm:char-Boolean-eqn-additive}\eqref{item:CD-variety}
for Boolean clones persists in the interval of clones of self-dual
operations on~$\{0,1,2\}$.

\begin{theorem}\label{teor:self-dual_operations_char}
For a clone $\funset{C}\subseteq \set{\zeta_3}^{*}$ on $A=\set{0,1,2}$
the following statements are equivalent.
\begin{enumerate}[\upshape(a)]
\item $\funset{C}$ is equationally
additive;\label{eq:teor_self_dua_eq_add}
\item\label{eq:teor_self_dual_filter_gen}
$\funset{C}$ contains one of the clones~$\aipii$, $\Aipii$ or\/~$\TN$
(cf.~\cite[\figurename~2, p.~260]{Zhu15});
\item\label{eq:teor_self_dual_variety_cd}
  $\algop{A}{\funset{C}}$ generates a congruence distributive variety.
\end{enumerate}
\end{theorem}
\begin{proof}
Let $\funset{C}\subseteq \set{\zeta_3}^{*}$.
If $\aipii\subseteq \funset{C}$,
or $\Aipii\subseteq \funset{C}$,
then Lemma~\ref{lemma:aipii-Aipii-eqn-add} yields
that~$\funset{C}$ is equationally additive.
Next, we prove that~$\TN$ is equationally additive.
To this end, let $m\colon A^3\to A$ be defined as
in~\cite[p.~264]{Zhu15}, cf.\ above;
it is evident from its definition that~$m$ is a majority operation.
According to~\cite[Theorem~6]{Zhu15}, we have that~$m$ generates~$\TN$.
Moreover, it is easy to show via a computer
(cf.~\cite{AicBehRos23DatasetSystems4EqnAdd}) that the
solution set of the following system of equations is~$\deltarelation$:
\[
\begin{cases}
&m(x_1,x_2,x_3)\approx m(x_1,x_2,x_4)\\
&m(x_2,x_1,x_3)\approx m(x_2, x_1, x_4).
\end{cases}
\]
Therefore, if $\TN\subseteq \funset{C}\subseteq \set{\zeta_3}^{*}$,
then~$\funset{C}$ is equationally additive by
Theorem~\ref{teor:eq_additive_iff_delta_four_algebraic}
and Corollary~\ref{cor:equational_add_and_inclusion}.
Hence~\eqref{eq:teor_self_dual_filter_gen}
implies~\eqref{eq:teor_self_dua_eq_add}.

Next, we prove that~\eqref{eq:teor_self_dua_eq_add}
implies~\eqref{eq:teor_self_dual_filter_gen}.
According to~\cite[\figurename~2]{Zhu15}, $\aP$, $\AP$, $\Ltwo$ and~$\SL$
are those clones of self-dual operations that are maximal with respect
to not containing either of the clones~$\aipii$, $\Aipii$ or\/~$\TN$.
Hence, as a consequence of Corollary~\ref{cor:equational_add_and_inclusion},
it suffices to prove that~$\aP$, $\AP$, $\Ltwo$ and~$\SL$ are not
equationally additive. If~$\SL$ were equationally additive,
then so would be its constantive expansion, which,
by Lemma~\ref{lemma:existence_of_type_2_and_5_prime_quotients}\eqref{item:SL-polyeq-Z3},
coincides with the clone of polynomial functions of the $\GF(3)$-vector
space~$\Z_3$. Since the vector space~$\Z_3$ is simple and has a
Ma\v{l}cev (term) operation,
Corollary~\ref{cor:nec_condi_malcev_polynomial_and_finite_yields_sub_irreducible_and_type_three}
says that equational additivity of its polynomial clone requires the
vector space to be a non-Abelian algebra, which it is not
(cf.~\cite[Exercise~3.2(2)]{HobbMcK}). Therefore,
$\SL$ cannot be equationally additive.
By Lemma~\ref{lemma:existence_of_type_2_and_5_prime_quotients}\eqref{item:aPAPL2-restrict-VEL2}, the
clones~$\aP$, $\AP$ and~$\Ltwo$ have $B\defeq\set{0,1}$ as an invariant
subset and $\aP\restrict{B}=\Bcl{V}_2$, $\AP\restrict{B}=\Bcl{E}_2$ and
$\Ltwo\restrict{B}=\Bcl{L}_2$; each of these Boolean clones fails to be
equationally additive by
Theorem~\ref{thm:char-Boolean-eqn-additive}\eqref{item:maximal-elements-complement}.
Hence, by Lemma~\ref{lemma:char-eqn-add-clones-inv-subsets},
none of~$\aP$, $\AP$ or~$\Ltwo$ can be equationally additive.
This establishes the equivalence of~\eqref{eq:teor_self_dua_eq_add}
and~\eqref{eq:teor_self_dual_filter_gen}.

The fact that~\eqref{eq:teor_self_dual_filter_gen}
implies~\eqref{eq:teor_self_dual_variety_cd} follows
from the fact that the clones~$\aipii$, $\Aipii$ and~$\TN$
have J\'{o}nsson operations, as argued in~\cite{BodVucZhu23}:
Namely, in the proof of~\cite[Proposition~5.3]{BodVucZhu23} it is shown
how one can derive a sequence of five quasi-J\'{o}nsson operations
from $\fipii\in\aipii$; since~$\fipii$ is idempotent, these are actually
J\'{o}nsson operations.
The exact same can be done with $(\fipii)^{*}\in\Aipii$.
As observed above, the generator~$m$ of~$\TN$ is a majority operation
(and thus gives rise to a sequence of three J\'{o}nsson operations).

Finally, we show that~\eqref{eq:teor_self_dual_variety_cd}
implies~\eqref{eq:teor_self_dual_filter_gen}.
To this end it suffices to prove that
for all clones~$\funset{D}$ below one of the clones~$\aP, \AP, \Ltwo$,
or~$\SL$, the algebra $\algop{A}{\funset{D}}$ does not generate a
congruence distributive variety.
Let $\funset{C}\in\set{\aP,\AP,\Ltwo,\SL}$ and
assume that $\algop{A}{\funset{D}}$ would generate a congruence
distributive variety for some clone $\funset{D}\subs\funset{C}$. Then
there would be a sequence of J\'{o}nsson operations in~$\funset{D}$ and
hence in~$\funset{C}$.
If $\funset{C}\in\set{\aP,\AP,\Ltwo}$, then $B\defeq\set{0,1}$ is
invariant for~$\funset{C}$ by
Lemma~\ref{lemma:existence_of_type_2_and_5_prime_quotients}\eqref{item:aPAPL2-restrict-VEL2},
and by restricting the J\'{o}nsson operations to~$B$ we would obtain
J\'{o}nsson operations in~$\funset{C}\restrict{B}$.
Hence $\algop{B}{\funset{C}\restrict{B}}$ would generate a congruence
distributive variety, thus Theorem~\ref{thm:char-Boolean-eqn-additive}
excludes that
$\funset{C}\restrict{B}\subs\Bcl{V},\Bcl{E}$ or~$\Bcl{L}$.
However,
Lemma~\ref{lemma:existence_of_type_2_and_5_prime_quotients}\eqref{item:aPAPL2-restrict-VEL2}
shows that exactly the latter is the case since
$\funset{C}\restrict{B}\in\set{\Bcl{V}_2,\Bcl{E}_2,\Bcl{L}_2}$.
Therefore, the only possible remaining case is that
$\funset{D}\subs\funset{C}=\SL$ and hence $\algop{A}{\funset{\SL}}$
would generate a congruence distributive variety.
But $\algop{A}{\funset{\SL}}$ is
polynomially equivalent to the vector space~$\Z_3$ by
Lemma~\ref{lemma:existence_of_type_2_and_5_prime_quotients}\eqref{item:SL-polyeq-Z3},
and thus $\algop{A}{\funset{\SL}}^2$ is polynomially equivalent
to~$\Z_3^2$.
Hence $\Con\bigl(\algop{A}{\funset{\SL}}^2\bigr)=\Con(\Z_3^2)$, which fails
to be distributive. This contradiction shows that our assumption is
impossible and thus~\eqref{eq:teor_self_dual_filter_gen} follows.
\end{proof}

\section{The number of equationally additive clones on finite sets}\label{sec:on_the_number}
In this section we investigate the cardinality of the order filter of
equationally additive clones on a finite set.
Our first basic observation is that the number of equationally
additive
clones on a set always is a lower bound for the number of equationally
additive clones on any superset.
\begin{lemma}\label{lemma:increasing_A_increases_the_number_of_equationally_additive_clones}
For sets $A\subseteq B$ there are at least as many equationally
additive clones on~$B$ as on~$A$.
\end{lemma}
\begin{proof}
On any set the clone of all finitary operations is equationally
additive, therefore the case $A=\emptyset$ is settled.
If $A=B$ the statement is also evident.
Therefore, from now on let us assume that $\emptyset\neq A\subsetneq B$.
Let us choose elements $a\in A$ and $b\in B\setminus A$
and let us denote the set of all clones
on~$A$ and~$B$ by~$\clonelat_A$ and~$\clonelat_B$, respectively.
If, for $n\in\N$, $f\colon A^n\to A$ and $u\colon B^n\setminus
A^n\to B$
are functions, then we denote by $f\oplus u\colon B^n\to B$ the
operation defined from these two functions by the obvious case
distinction. Clearly, any function constructed in this way
preserves~$A$, and conversely, any function $g\colon B^n\to B$
preserving~$A$ can be split up in this form. We employ the
well-known injection $\Phi\colon \clonelat_A\to\clonelat_B$ between
the clone lattices (cf.~\cite[3.3.3~Ein\-bettungs\-satz]{PosKal79}),
which is defined for every $\funset{F}\in\clonelat_A$ as
\[
\Phi(\funset{F})=\bigcup_{n\in\N}\{f\colon B^{n}\to B\mid f[A^n]\subs
A, f\restrict{A^n}\in\funset{F}\}.\]
Letting $\cna{b}\colon B^n\setminus A^n\to B$ be the constant
$n$-ary function with value~$b$, we observe for any $f\colon A^n\to A$
that $f\in\funset{F}$ if and only if
$f\oplus \cna{b}\in\Phi(\funset{F})$.
Hence we see that~$\Phi$ is injective, since for any $n$-ary
function~$f$ separating clones $\funset{F},\funset{G}\in\clonelat_A$
we
also have $f\oplus \cna{b}$ separating~$\Phi(\funset{F})$
and~$\Phi(\funset{G})$.
\par
The proof will be done once we have shown that~$\Phi$ preserves
equational additivity. To this end assume that
$\funset{F}\in\clonelat_A$ is equationally additive, and that,
according to Theorem~\ref{teor:eq_additive_iff_delta_four_algebraic},
there is some index set~$I$ and there are functions
$f_i,g_i\in\funset{F}\arii{4}$ for $i\in I$ such that
$\deltarelation
= \set{\vec{x}\in A^4\mid \forall i\in I\colon
f_i(\vec{x})=g_i(\vec{x})}$.
Define ${u\colon B^4\setminus A^4\to B}$ by $u(\vec{b})=b$ if
$\vec{b}\in\deltarelation[B]$ and by $u(\vec{b})=a$ otherwise.
Let~$\eni[4]{1}$ denote the quaternary projection on~$A$ to the first
coordinate. Then it is not hard to verify that
\begin{multline*}
\deltarelation[B]=\\
\Set{\vec{x}\in B^4|
\eni[4]{1}\oplus u(\vec{x})=\eni[4]{1}\oplus
\cna[4]{b}(\vec{x})\land
\forall i\in I\colon f_i\oplus \cna[4]{b}(\vec{x})=g_i\oplus
\cna[4]{b}(\vec{x})}.
\end{multline*}
Thus, $\deltarelation[B]$ is algebraic over the
clone~$\Phi(\funset{F})$, and hence~$\Phi(\funset{F})$ is equationally
additive by Theorem~\ref{teor:eq_additive_iff_delta_four_algebraic}.
\end{proof}

The theory presented in
Section~\ref{section:equatio_add_boolean_clones}, in particular
Theorem~\ref{thm:char-Boolean-eqn-additive} and
Figure~\ref{fig:eqn-add-Boolean-clones}, shows that the number of
equationally additive clones on any two-element set is countably
infinite. The next step to take is investigating equationally additive
clones on (at least) three-element carrier sets.
In combination with
Lemma~\ref{lemma:increasing_A_increases_the_number_of_equationally_additive_clones},
Corollary~\ref{cor:continuously_many_self_dual_eq_additive_clones}
shows that there are precisely continuum many equationally additive
clones on any finite set with at least three elements.
In the following we focus on the number of clones on
finite sets that are equationally additive \emph{and} contain all
constant (unary) operations. In the subsequent proposition we start
again by first considering three-element carrier sets. After that we
shall exploit a construction by \'{A}goston, Demetrovics and
Hann\'{a}k
from~\cite{AgoDemHan83} to cover the general case.
\begin{proposition}\label{prop:countinously_many_eq_additive_clones_on_the_3_element_set}
On the set $A=\{0,1,2\}$ there are exactly~$2^{\aleph_0}$ distinct
equationally additive constantive clones.
\end{proposition}
\begin{proof}
Let $f\colon A^3 \to A$ be defined as follows. For all
$\vec{x}=(x_1,x_2,x_3)\in A^3$ we set
\[
f(\vec{x})=
\begin{cases}
2 &\text{ if } \vec{x}\in \{(0,2,0),(0,1,1),(1,2,2)\},\\
x_1 & \text{ otherwise}.
\end{cases}
\]
Furthermore, let $S\subseteq A^4$ be the solution set of the following
system of equations
\[
\begin{cases}
&f(x_1,x_2,x_3)\approx f(x_1, x_2,x_4)\\
&f(x_2,x_1,x_3)\approx f(x_2, x_1,x_4)\\
&f(x_3,x_4,x_1)\approx f(x_3, x_4,x_2)\\
&f(x_4,x_3,x_1)\approx f(x_4, x_3,x_2).
\end{cases}
\]
Using a computer one quickly verifies that $\deltarelation=S$,
cf.~\cite{AicBehRos23DatasetSystems4EqnAdd}. Thus,
Theorem~\ref{teor:eq_additive_iff_delta_four_algebraic} yields that
every clone that contains~$f$ is equationally additive.

We now prove that there are~$2^{\aleph_0}$ constantive clones that
contain~$f$.
For each $a\in A$, let $\ca\colon A\to A$ be the unary constant
function with constant value~$a$.
For $k\in\N$ and $i\in \finset{k}$, we set~$\vec{e}_i^k$ to be the
element of $\{0,1\}^k$ with~$1$ in the $i$-th component and~$0$
elsewhere.
Moreover, we define the `forbidden' set
$B_k\defeq\{\vec{r}\in\{0,1\}^k\mid 3\leq \w{\vec{r}}\leq k-1\}$,
where~$\w{\vec{r}}$ denotes the number of occurrences of~$1$
in~$\vec{r}$,
and we set $\rho_k\defeq A^k\setminus(\{\vec{e}_1^k\}\cup B_k)$.
For $n\in\N$, $n\geq 2$ we define $f_n\colon A^n\to A$ as follows.
For all $\vec{x}\in A^n$ we let
\[
f_n(\vec{x})=
\begin{cases}
1 &\text{ if } \w{\vec{x}}=n,
   \text{i.e., }\vec{x}=\vec{1}\defeq(1,\dotsc,1),\\
0 &\text{ if } \vec{x}\in\{\vec{e}_1^n,\dotsc,\vec{e}_n^n\},\\
2 &\text{ otherwise}.
\end{cases}
\]

As a first step, we prove that
\begin{equation}\label{eq:f_and_const_are_poly}
\forall k\geq 2 \colon\quad \{f\}\cup\{\ca\mid a\in A\} \subseteq
\Pol(\{\rho_k\}).
\end{equation}
Since~$\rho_k$ is reflexive, clearly
$\{\ca\mid a\in A\} \subseteq \Pol(\{\rho_k\})$.
Let $\vec{z}_1, \vec{z}_2, \vec{z}_3\in \rho_k$.
Then, the component-wise action of~$f$ on
$(\vec{z}_1, \vec{z}_2, \vec{z}_3)$ yields a tuple~$\vec{z}$ satisfying
for all $i\in\finset{k}$ the condition $\vec{z}(i)=\vec{z}_{1}(i)$ or
$\vec{z}(i)=2$.
If there is $i\in\finset{k}$ with $\vec{z}(i)=2$, then
$\vec{z}\notin \set{\vec{e}_1^k}\cup B_k$, and therefore,
$\vec{z}\in \rho_k$. Otherwise, $\vec{z}=\vec{z}_1\in \rho_k$.

Next, we show that
\begin{equation}\label{eq:f_n-does_not_preserve_rho_n+1}
\forall n\geq 2 \colon\quad
f_n\notin\Pol(\{\rho_{n+1}\}).
\end{equation}
For $i\in\finset{n}$, let~$\vec{z}_i$ be the element of
$\{0,1\}^{n+1}$
with~$1$ in its first and $(i+1)$-st component, and~$0$ elsewhere.
For each $i\in\{1,\dots, n\}$ we have $\w{\vec{z}_i}=2$, thus
$\vec{z}_i\in \rho_{n+1}$.
Moreover, $f_n$ acting component-wise on these tuples
$\vec{z}_1,\dotsc,\vec{z}_n$ yields
$f_n(\vec{z}_1, \dots, \vec{z}_n)=\vec{e}_1^{n+1}\notin\rho_{n+1}$.
This proves that~$f_n$ does not preserve~$\rho_{n+1}$.

Third, we verify that
\begin{equation} \label{eq:f_n_preserves_rho_k}
\forall n\geq 2\, \forall k\in \N\setminus\{n+1\}\colon\quad
f_n\in\Pol(\{\rho_k\}).
\end{equation}
For this we show that each of the tuples in $B_k\cup\set{\vec{e}_1^k}$
can only be obtained by the component-wise action of~$f_n$ on a
sequence $\vec{z}_1,\dotsc,\vec{z}_n$ of tuples in~$A^k$ that contains
at least one member outside~$\rho_k$, that is, in
$B_k\cup\set{\vec{e}_1^k}$. We assume $n\geq 2$ and split
the proof into two cases according to the value of~$k$.

\textbf{Case}: $1\leq k\leq n$:
Let $\vec{z}_1,\dots, \vec{z}_n\in A^k$ be such that
$f_n(\vec{z}_1,\dots, \vec{z}_n)=\vec{e}_1^k$,
and let~$Z$ be the $(k\times n)$-matrix whose columns are the tuples
$\vec{z}_1,\dotsc,\vec{z}_n$.
As~$1$ has a unique preimage under~$f_n$,
the first row of~$Z$ is~$\vec{1}$.
Moreover, for all $2\leq i\leq k$ there exists $l_i\in\finset{n}$ such
that the $i$-th row of~$Z$ is the tuple~$\vec{e}_{l_i}^n$.
Since $k-1< n$, the matrix~$Z'$ obtained from~$Z$ by removing
the first
row, contains a column whose entries are all~$0$. Thus,
$\vec{e}_1^k\in \{\vec{z}_1,\dots, \vec{z}_n\}$, and so
$\{\vec{z}_1,\dots, \vec{z}_n\}\nsubseteq \rho_k$.
If $k\geq 4$, then we also have to consider any $\vec{r}\in B_k$ with
$\w{\vec{r}}=w$, where $3\leq w\leq k-1$, and we let
$\vec{v}_1,\dots \vec{v}_n\in A^k$ be such that
$f(\vec{v}_1,\dots \vec{v}_n)=\vec{r}$.
Let~$V$ be the $(k\times n)$-matrix whose columns are
$\vec{v}_1,\dots, \vec{v}_n$.
Then~$V$ has exactly~$w$ rows whose entries are all~$1$. Let~$V'$ be
the $((k-w)\times n)$-matrix obtained by removing those rows from~$V$.
For all $i\in\finset{k-w}$ there exists $l_i\in\finset{n}$ such that
the $i$-th row of~$V'$ is
the tuple~$\vec{e}_{l_i}^n$. Since $k-w\leq k-3<n$, there is a
column~$\vec{v}'$ in~$V'$ whose entries are all zero, and therefore,
there exists $j\in\finset{n}$ such that the column
$\vec{v}_j\in\set{0,1}^k$ of~$V$ satisfies $\w{\vec{v}_j}=w$.
Since $3\leq w\leq k-1$, we have $\vec{v}_j\in B_k$, and thus
$\{\vec{v}_1, \dots, \vec{v}_n\}\nsubseteq \rho_k$.

\textbf{Case}: $k\geq n+2$:
Let $\vec{z}_1,\dots, \vec{z}_n\in A^k$ be such that
$f_n(\vec{z}_1,\dots, \vec{z}_n)=\vec{e}_1^k$, and
let~$Z$ be the
$(k\times n)$-matrix whose columns are the tuples
$\vec{z}_1,\dotsc,\vec{z}_n$. As above, the first row of~$Z$
equals~$\vec{1}$.
Moreover, for all $2\leq i\leq k$ there exists $l_i\in\finset{n}$ such
that the $i$-th row of~$Z$ is the tuple~$\vec{e}_{l_i}^n$.
Since $n\leq k-2<k-1$, the matrix~$Z'$ obtained from~$Z$ by removing
the first row, contains a column $\vec{z}'\in\{0,1\}^{k-1}$ with
$\w{\vec{z}'}\geq 2$. If $\w{\vec{z}'}\leq k-2$, then~$Z$ contains a
column $\vec{z}\in B_k$, i.e., $\vec{z}\notin\rho_k$.
If $\w{\vec{z}'}= k-1$, then all the
rows of~$Z'$ are of the form~$\vec{e}_l^n$ for the same
$l\in\finset{n}$, and since $n\geq 2$, $Z'$ has a column whose entries
are all zeros. Thus, $Z$ has a column equal to
$\vec{e}_1^k\notin\rho_k$.
Since $k\geq n+2\geq 4$, we additionally have to
consider any $\vec{r}\in B_k$ with $\w{\vec{r}}=w$,
where $3\leq w\leq k-1$, and we let $\vec{v}_1,\dots \vec{v}_n\in A^n$
be such that $f(\vec{v}_1,\dots \vec{v}_n)=\vec{r}$.
Let~$V$ be the $(k\times n)$-matrix whose
columns are $\vec{v}_1,\dotsc,\vec{v}_n$.
Then~$V$ has~$w$ rows whose entries are all~$1$. Setting~$V'$ to
be the
$((k-w)\times n)$-matrix obtained by removing these rows from~$V$, we
have that for all $i\in\finset{k-w}$ there exists $l_i\in\finset{n}$
such that the $i$-th row of~$V'$ is~$\vec{e}_{l_i}^n$. Thus,
$\vec{v}_1,\dotsc,\vec{v}_n\in\set{0,1}^k$.
Since $k-w\geq 1$ and $n\geq 2$, there is
$j\in\finset{n}\setminus\set{l_1}$, for which the entry of~$\vec{v}_j$
in the first row of~$V$ that is distinct from~$\vec{1}$ equals~$0$.
Hence, $3\leq w\leq \w{\vec{v}_j}\leq k-1$, and so $\vec{v}_j\in
B_k$, i.e.,
$\vec{v}_j\notin\rho_k$.

We are now ready to prove the statement of the theorem.
We abbreviate $N\defeq\N\setminus\set{1,2}$, and denote
by~$\clonelat_f$ the lattice of all constantive clones on~$A$ that
contain~$f$. Then we define $\Phi\colon \potenza{N}\to \clonelat_f$ as
follows. For all $I\in \potenza{N}$ we let
$\Phi(I)=\Pol(\{\rho_i\mid i\in I\})$.
Equation~\eqref{eq:f_and_const_are_poly} ensures that this function is
well defined. We argue that~$\Phi$ induces an order embedding
of the lattice $(\potenza{N},\supseteq)$ into $(\clonelat_f,\subs)$.
Clearly, $\Phi$ is compatible with the inclusion orders. To show that
it also reflects them, take $I,J\subs N$ with $\Phi(J)\subs\Phi(I)$
and
consider any $\iota\in I$. We thus have
$\Phi(\set{\iota})\supseteq \Phi(I)\supseteq \Phi(J)$.
If $\iota\notin J$, equivalently, $J\subs N\setminus\set{\iota}$, then
we would have $\Phi(J)\supseteq \Phi(N\setminus\set{\iota})$,
therefore
$f_{\iota-1}\in\Phi(N\setminus\set{\iota})\subs\Phi(\set{\iota})$
by~\eqref{eq:f_n_preserves_rho_k}, but this would
contradict~\eqref{eq:f_n-does_not_preserve_rho_n+1}.
Hence $\iota\in J$, that is, we have demonstrated $I\subs J$.
As every order embedding is injective, this proves
that $\crd{\clonelat_f}=2^{\aleph_0}$, and the statement follows.
\end{proof}

\begin{theorem}\label{teor:countinously_many_eq_additive_clones_on_the_4_element_set}
On a finite set~$A$ with at least three elements there
are exactly~$2^{\aleph_0}$ distinct equationally additive constantive
clones.
\end{theorem}
\begin{proof}
Let $A=\{0,\dots, n\}$ with $n\geq 2$.
The case $n=2$ is a consequence of
Proposition~\ref{prop:countinously_many_eq_additive_clones_on_the_3_element_set}.
Thus we only consider the case $n\geq 3$.
We define $f\colon A^4\to A$ by
 \[
 f(\vec{x})=\begin{cases}0 & \text{ if } \vec{x}\in\deltarelation,\\
 n & \text{ otherwise.}\end{cases}
 \]
For each $i\in\N\setminus\{1,2\}$ we define $h_i\colon A^i \to A$ by
 \[
 h_i(\vec{x})=
 \begin{cases}
  1& \text{ if }
   \crd{\{j\in\finset{i}\colon x_j=1\}}=1 \text{ and }
   \crd{\{j\in\finset{i}\colon x_j=2\}}=i-1,\\
   & \text{ or }\crd{\{j\colon x_j=2\}}=1 \text{ and }
                \crd{\{j\in\finset{i}\colon x_j=1\}}=i-1;\\
  0&\text{ otherwise.}
 \end{cases}
 \]
For each $I\subseteq \N\setminus\{1,2\}$ we define~$\ab{A}_I$ as
the algebra $\algop{A}{\{f\}\cup \{h_i\mid i\in I\}}$.
Let $Z\defeq A\setminus\{n\}$.
Following~\cite{AgoDemHan83},
for each $i\in \N\setminus\{1,2\}$ we define $g_i\colon Z^i\to Z$ by
 \[
 g_i(\vec{x})=
 \begin{cases}
  1& \text{ if }
   \crd{\{j\in\finset{i}\colon x_j=1\}}=1 \text{ and }
   \crd{\{j\in\finset{i}\colon x_j=2\}}=i-1,\\
   & \text{ or }\crd{\{j\in\finset{i}\colon x_j=2\}}=1\text{ and }
                \crd{\{j\in\finset{i}\colon x_j=1\}}=i-1;\\
  0&\text{ otherwise.}
 \end{cases}
 \]
For each $I\subseteq N\setminus\{1,2\}$ we define~$\ab{Z}_I$ as the
algebra $\algop{Z}{\{g_i\mid i\in I\}}$.
In~\cite{AgoDemHan83} it was proved that
 \begin{equation}\label{eq:the.clones_of_jan_muc_53_are_not_term_equivalent}
 \forall I,J\in\potenza{\N\setminus\{1,2\}}\colon \POL \ab{Z}_I=\POL
 \ab{Z}_J \iff I=J.
 \end{equation}
For each $I\subseteq \N\setminus\{1,2\}$ we define~$\ab{Z}_I^{0}$
as the
algebra $\algop{Z}{\{\cna[4]{0}\}\cup\{g_i\mid i\in I\}}$, where~$\cna[4]{0}$ is the
constant $0$-function of arity~$4$.
Clearly, $\cna[4]{0}\in\POL\ari{4}\ab{Z}_I$ for all $I\subseteq
\N\setminus\{1,2\}$. Thus, we have that for all
$I\subseteq N\setminus\{1,2\}$ the algebras~$\ab{Z}_I$
and~$\ab{Z}_I^0$ are polynomially equivalent. Therefore,
\eqref{eq:the.clones_of_jan_muc_53_are_not_term_equivalent}
yields that
 \begin{equation}\label{eq:the.clones_of_jan_muc_53_are_not_term_equivalent_zero_case}
 \forall I,J\in\potenza{\N\setminus\{1,2\}}\colon \POL \ab{Z}_I^0=\POL
 \ab{Z}_J^0 \iff I=J.
 \end{equation}
 Let $I\in\potenza{\N\setminus\{1,2\}}$.
Since $\deltarelation=\{\vec{x}\in A^4 \mid f(\vec{x})=0\}$,
Theorem~\ref{teor:eq_additive_iff_delta_four_algebraic} yields
that the clone $\POL\ab{A}_I$ is equationally additive. Moreover,
Proposition~\ref{prop:the_TCT_type_of_monolith_when_delta_is_f_equal_constant}
yields that the algebra~$\ab{A}_I$ is subdirectly irreducible and
that $\mu=\Cg[{\ab{A}_I}]{\{(0,n)\}}$ is the monolithic congruence
of~$\ab{A}_I$.
Next, we show that $\mu=\bottom{A}\cup \{(0,n),(n,0)\}$. Clearly,
$S\defeq \bottom{A}\cup\{(0,n),(n,0)\}$ is an equivalence relation on~$A$
that contains the generators of~$\mu$ and that is minimal with
this property. Thus, it suffices to show that~$S$ is a subalgebra
of $\ab{A}_I\times \ab{A}_I$. To this end let $\vec{a},\vec{b}\in A^4$
with $\vec{a}\equiv_S \vec{b}$. Since $f(\vec{a})\in\{0,n\}$
and $f(\vec{b})\in\{0,n\}$ and $\{0,n\}\times\{0,n\}\subseteq S$,
we have $(f(\vec{a}), f(\vec{b}))\in S$. This proves that~$S$
is closed under the component-wise action of~$f$.
Let $i\in I$ and let $\vec{c},\vec{d}\in A^{i}$ with
$\vec{c}\neq \vec{d}$ and $\vec{c}\equiv_S \vec{d}$. We show
that $(h_i(\vec{c}),h_i(\vec{d}))\in S$. Since for all
$z\in Z\setminus\{0\}$ the equivalence class of~$z$ modulo~$S$ is a
singleton, we have that $\vec{c}\neq \vec{d}$ and $\vec{c}\equiv_S
\vec{d}$ together yield that there exists $\ell\in\finset{i}$ such
that $\vec{c}(\ell),\vec{d}(\ell)\in\{0,n\}$ and
$\vec{c}(\ell)\neq\vec{d}(\ell)$. Then the definition
of~$h_i$ yields $h_i(\vec{c})=0=h_i(\vec{d})$, and therefore
$(h_i(\vec{c}),h_i(\vec{d}))=(0,0)\in S$.
Thus, $S$ is closed under the component-wise action of~$f$ and of~$h_i$
for all $i\in I$, and therefore is a subalgebra of $\ab{A}_I\times
\ab{A}_I$.
Hence $\mu=\bottom{A}\cup \{(0,n),(n,0)\}$.

Next, we prove that $\ab{A}_I/\mu\cong \ab{Z}_I^0$. Define $\phi\colon
A\to Z$ by $\phi(x)=x$ for all $x\in Z$ and $\phi(n)=0$. We show
that~$\phi$ is a surjective homomorphism from~$\ab{A}_I$
to~$\ab{Z}_I^0$.
To this end, let $\vec{b}\in A^4$, let $i\in I$, and let $\vec{a}\in
A^i$.
As $f(\vec{b})\in\{0,n\}$, we have
$\phi(f(\vec{b}))=0=\cna[4]{0}(\phi(b_1),\phi(b_2),\phi(b_3),\phi(b_4))$.
Next, we demonstrate that
$\phi(h_i(\vec{a}))=g_i(\phi(a_1),\dots{,}\phi(a_i))$.
We split the proof into two cases.
Assuming $\vec{a}\in Z^i$, we have, as the image of~$h_i$ is a subset
of~$Z$, that~$\phi$ is the identity, and thus
$\phi(h_i(\vec{a}))=h_i(\vec{a})=g_i(\vec{a})=g_i(\phi(a_1),\dots,\phi(a_i))$.
If, otherwise, $\vec{a}\notin Z^i$, then there is $j\in\finset{i}$
such that $a_j=n$ and hence $\phi(a_j)=\phi(n)=0$. Thus, we have
$h_i(\vec{a})=0=g_i(\phi(a_1),\dots, \phi(a_i))$, and~$\phi$ is a
homomorphism.
Since the kernel of~$\phi$ is clearly equal to~$\mu$ and~$\phi$ is
surjective, the first homomorphism theorem
(cf.~\cite[Theorem~6.12]{BurSan81}) yields
$\ab{A}_I/\mu\cong \ab{Z}_I^0$.

Finally, we prove that
\begin{equation}\label{eq:the.clones_of_jan_muc_53_are_not_term_equivalent_the_Q_case}
\forall I,J\in\potenza{\N\setminus\{1,2\}}\colon
                                  \POL\ab{A}_I=\POL\ab{A}_J \iff I=J.
\end{equation}
To this end, let $I,J\in \potenza{\N\setminus\{1,2\}}$. Clearly, if
$I=J$, then $\POL\ab{A}_I=\POL\ab{A}_J$. For the opposite
implication, let us assume that $\POL\ab{A}_I=\POL\ab{A}_J$.
Then~\eqref{eq:the_clone_of_actions_of_function_on_a_quotient_algebra}
yields that
$\POL(\ab{A}_I/\mu)=(\POL\ab{A}_I)/\mu=(\POL\ab{A}_J)/\mu=\POL(\ab{A}_J/\mu)$.
Hence, since for
all $L\in\potenza{\N\setminus\{1,2\}}$, $\phi$ induces an isomorphism
between~$\ab{Z}_L^0$ and $\ab{A}_L/\mu$ that is independent of~$L$, we
have $\POL(\ab{Z}_I^0)=\POL(\ab{Z}_J^0)$.
Thus, we have $I=J$
by~\eqref{eq:the.clones_of_jan_muc_53_are_not_term_equivalent_zero_case}.

By~\eqref{eq:the.clones_of_jan_muc_53_are_not_term_equivalent_the_Q_case},
$\{\POL\ab{A}_I\mid I\in \potenza{\N\setminus\{1,2\}}\}$ is a
set of distinct equationally additive constantive clones on~$A$
of cardinality
$\crd{\potenza{\N\setminus\{1,2\}}}=2^{\aleph_0}$.
\end{proof}
We remark that the clones constructed in
Theorem~\ref{teor:countinously_many_eq_additive_clones_on_the_4_element_set}
all have the same universal algebraic geometry, namely
$\bigcup_{n\in\N}\potenza{A^n}$.
We summarize our knowledge of the number of equationally additive
clones on finite sets in \tablename~\ref{tabella:tabella_numero_cloni}.
\begin{table}
\begin{tabular}{C{5.5mm}C{12mm}C{18mm}C{34mm}C{34mm}}
\toprule
$\crd{A}$ & All & Constantive & Equationally additive & Constantive
equationally additive
\tabularnewline
\midrule
$2$ &$\aleph_0$ \cite{Pos41}&$7$ \cite{Pos41}
&$\aleph_0$ (\figurename~\ref{fig:eqn-add-Boolean-clones})&$2$
(\figurename~\ref{fig:eqn-add-Boolean-clones})
\tabularnewline
$\geq3$ &$2^{\aleph_0}$ \cite{JanMuc59} &$2^{\aleph_0}$
\cite{AgoDemHan83} &$2^{\aleph_0}$
&$2^{\aleph_0}$(Theorem~\ref{teor:countinously_many_eq_additive_clones_on_the_4_element_set})
\tabularnewline
\bottomrule
\end{tabular}
\caption{Numbers of clones on a finite non-trivial
set~$A$.}\label{tabella:tabella_numero_cloni}
\end{table}

\providecommand*{\url}[1]{\texttt{\detokenize{#1}}}
  \ifx\SetBibliographyCyrillicFontfamily\undefined\def\SetBibliographyCyrillicFontfamily{\relax}\fi
  \def\Cyr#1{\bgroup\SetBibliographyCyrillicFontfamily\fontencoding{T2A}\selectfont{#1}\egroup}
  \def\Palatalization#1{\bgroup\fontencoding{T1}\selectfont\v{#1}\egroup}
\providecommand{\bysame}{\leavevmode\hbox to3em{\hrulefill}\thinspace}
\providecommand{\MR}{\relax\ifhmode\unskip\space\fi MR }
\providecommand{\MRhref}[2]{%
  \href{http://www.ams.org/mathscinet-getitem?mr=#1}{#2}
}
\providecommand{\href}[2]{#2}


\begin{thebibliography}{10}

\bibitem{AglBak99}
Paolo Aglian{\`o} and Kirby~A. Baker, \emph{Congruence properties of
  two-generated varieties}, {C}ontributions to General Algebra 12 ({V}ienna,
  1999), Heyn, Klagenfurt, 2000, pp.~71--83. \MR{1777648}

\bibitem{AgoDemHan83}
Istv{\'{a}}n {\'{A}}goston, J{\'{a}}nos Demetrovics, and L{\'{a}}szl{\'{o}}
  Hann{\'{a}}k, \emph{On the number of clones containing all constants (a
  problem of {R}.~{M}c{K}enzie)}, Lectures in universal algebra ({S}zeged,
  1983), Colloq. Math. Soc. J\'{a}nos Bolyai, vol.~43, North-Holland Publishing
  Company, Amsterdam, 1986, pp.~21--25. \MR{860252}

\bibitem{Aic06}
Erhard Aichinger, \emph{The polynomial functions of certain algebras that are
  simple modulo their center}, {C}ontributions to General Algebra, vol.~17,
  Heyn, Klagenfurth, 2006, pp.~9--24. \MR{2237801}

\bibitem{Aic10}
\bysame, \emph{Constantive {M}a{\Palatalization{l}}cev clones on finite sets
  are finitely related}, {P}roc. {A}mer. {M}ath. {S}oc. \textbf{138} (2010),
  no.~10, 3501--3507. \MR{2661550}

\bibitem{AicCanEckKabNeu08}
Erhard Aichinger, G.~Alan Cannon, J{\"u}rgen Ecker, Lucyna Kabza, and Kent
  Neuerburg, \emph{Some near-rings in which all ideals are intersections of
  {N}oetherian quotients}, Rocky Mountain J. Math. \textbf{38} (2008), no.~3,
  713--726. \MR{2426518}

\bibitem{AicMay07}
Erhard Aichinger and Peter Mayr, \emph{Polynomial clones on groups of order
  {$pq$}}, Acta Math. Hungar. \textbf{114} (2007), no.~3, 267--285.
  \MR{2296547}

\bibitem{AicMud09}
Erhard Aichinger and Neboj\v{s}a Mudrinski, \emph{Types of polynomial
  completeness of expanded groups}, {A}lgebra {U}niversalis \textbf{60} (2009),
  no.~3, 309--343. \MR{2495241}

\bibitem{AicMud10}
\bysame, \emph{Some applications of higher commutators in
  {M}a{\Palatalization{l}}cev algebras}, {A}lgebra {U}niversalis \textbf{63}
  (2010), no.~4, 367--403. \MR{2734303}

\bibitem{AicRos20}
Erhard Aichinger and Bernardo Rossi, \emph{A clonoid based approach to some
  finiteness results in universal algebraic geometry}, {A}lgebra {U}niversalis
  \textbf{81} (2020), no.~1, 8:1--7. \MR{4055444}

\bibitem{AicRos23}
\bysame, \emph{On the number of universal algebraic geometries}, {A}lgebra
  {U}niversalis \textbf{84} (2023), no.~1, 1:1--15, see also
  \url{https://doi.org/10.48550/arXiv.2107.11063}. \MR{4518782}

\bibitem{BauMyaRem99}
Gilbert Baumslag, Alexei~Georgievich Myasnikov, and Vladimir~Nikanorovich
  Remeslennikov, \emph{Algebraic geometry over groups. {I}. {A}lgebraic sets
  and ideal theory}, J. Algebra \textbf{219} (1999), no.~1, 16--79.
  \MR{1707663}

\bibitem{AicBehRos23DatasetSystems4EqnAdd}
Mike Behrisch, Bernardo Rossi, and Erhard Aichinger, \emph{Systems for
  equational additivity [dataset]}, Zenodo (2023),
  \url{https://doi.org/10.5281/zenodo.8059121}.

\bibitem{Bod21}
Manuel Bodirsky, \emph{Complexity of infinite-domain constraint satisfaction},
  Lecture Notes in Logic, vol.~52, Cambridge University Press, Cambridge, June
  2021. \MR{4273453}

\bibitem{BodVucZhu23}
Manuel Bodirsky, Albert Vucaj, and Dmitriy Zhuk, \emph{The lattice of clones of
  self-dual operations collapsed}, Internat. J. Algebra Comput. \textbf{33}
  (2023), no.~4, 717--749.

\bibitem{BoehlerReithSchnoorVollmerBasesBooleanCoclones}
Elmar B\"{o}hler, Steffen Reith, Henning Schnoor, and Heribert Vollmer,
  \emph{Bases for {B}oolean co-clones}, Inform. Process. Lett. \textbf{96}
  (2005), no.~2, 59--66. \MR{2166271}

\bibitem{BurSan81}
Stanley Burris and Hanamantagouda~Pandappa Sankappanavar, \emph{A course in
  universal algebra}, Graduate Texts in Mathematics, vol.~78,
  Springer-Ver\-lag, New York, 1981, available on-line from
  \url{http://www.math.uwaterloo.ca/~snburris/htdocs/UALG/univ-algebra.pdf}.
  \MR{648287 (83k:08001)}

\bibitem{CreignouVollmerBooleanCSPsWhenDoesPostsLatticeHelp}
Nadia Creignou and Heribert Vollmer, \emph{{B}oolean constraint satisfaction
  problems: When does {P}ost's lattice help?}, Complexity of Constraints -- An
  Overview of Current Research Themes [{R}esult of a {D}agstuhl Seminar].
  (Nadia Creignou, Phokion~G. Kolaitis, and Heribert Vollmer, eds.), Lecture
  Notes in Comput. Sci., vol. 5250, Springer-Ver\-lag, 2008, pp.~3--37.

\bibitem{DanMyaRem10}
\`Evelina~Yur'evna Daniyarova, Alexei~Georgievich Myasnikov, and
  Vladimir~Nikanorovich Remeslennikov, \emph{Algebraic geometry over algebraic
  structures. {IV}. {E}quational domains and codomains}, Algebra i Logika
  \textbf{49} (2010), no.~6, 715--756, 844, 847, translation in Algebra and
  Logic \textbf{49} (2011), no. 6, 483--508. \MR{2828872}

\bibitem{DanMyaRem12}
\bysame, \emph{Algebraic geometry over algebraic structures. {II}.
  {F}oundations}, Fundam. Prikl. Mat. \textbf{17} (2011/12), no.~1, 65--106,
  translation in J. Math. Sci. (N.Y.) \textbf{185} (2012), no. 3, 389--416.
  \MR{2898219}

\bibitem{BagDem82}
J\'{a}nos Demetrovics and J\'{a}nos Bagyinszki, \emph{The lattice of linear
  classes in prime-valued logics}, Discrete mathematics ({W}arsaw, 1977),
  {B}anach Center Publ., vol.~7, PWN--Polish Scientific Publishers, Warsaw,
  1982, pp.~105--123. \MR{698101}

\bibitem{DemHan83}
J\'{a}nos Demetrovics and L\'{a}szl\'{o} Hann\'{a}k, \emph{The number of
  reducts of a preprimal algebra}, {A}lgebra {U}niversalis \textbf{16} (1983),
  no.~2, 178--185. \MR{692258}

\bibitem{Fio21}
Stefano Fioravanti, \emph{Expansions of abelian square-free groups}, Internat.
  J. Algebra Comput. \textbf{31} (2021), no.~4, 623--638. \MR{4277177}

\bibitem{FreMck87}
Ralph Freese and Ralph Mc{K}enzie, \emph{Commutator theory for congruence
  modular varieties}, London Math. Soc. Lecture Note Ser., vol. 125, Cambridge
  University Press, Cambridge, 1987. \MR{909290}

\bibitem{Her08}
Miki Hermann, \emph{On {B}oolean primitive positive clones}, Discrete Math.
  \textbf{308} (2008), no.~15, 3151--3162. \MR{2423397 (2009d:08006)}

\bibitem{HobbMcK}
David Hobby and Ralph~N. Mc{K}enzie, \emph{The structure of finite algebras},
  Contemporary Mathematics, vol.~76, {A}merican {M}athematical {S}ociety,
  Providence, RI, 1988. \MR{958685 (89m:08001)}

\bibitem{Idz99}
Pawe{\l}~M. Idziak, \emph{Clones containing {M}a{\Palatalization{l}}tsev
  operations}, Internat. J. Algebra Comput. \textbf{9} (1999), no.~2, 213--226.
  \MR{1703074}

\bibitem{IdzSlo01}
Pawe{\l}~M. Idziak and Katarzyna S{\l}omczy{\'{n}}ska, \emph{Polynomially rich
  algebras}, J. Pure Appl. Algebra \textbf{156} (2001), no.~1, 33--68.
  \MR{1807015}

\bibitem{Yab58}
Sergej~Vsevolodovi{\v{c}} Jablonskij,
  \emph{{\Cyr{\CYRF\cyru\cyrn\cyrk\cyrc\cyri\cyro\cyrn\cyra\cyrl\cyrsftsn\cyrn\cyrery\cyre{}
  \cyrp\cyro\cyrs\cyrt\cyrr\cyro\cyre\cyrn\cyri\cyrya{} \cyrv{}
  {$k$}-\cyrz\cyrn\cyra\cyrch\cyrn\cyro\cyrishrt{}
  \cyrl\cyro\cyrg\cyri\cyrk\cyre{}}} [{F}unctional constructions in
  {$k$}-valued logic]}, Trudy Mat. Inst. Steklov. \textbf{51} (1958), 5--142.
  \MR{0104578 (21 \#3331)}

\bibitem{BSAII}
Nathan Jacobson, \emph{Basic algebra. {II}}, second ed., W. H. Freeman and
  Company, New York, 1989. \MR{1009787}

\bibitem{JanMarVel16}
George Janelidze, L\'{a}szl\'{o} M\'{a}rki, and Stefan Veldsman,
  \emph{Commutators for near-rings: {H}uq {$\neq$} {S}mith}, {A}lgebra
  {U}niversalis \textbf{76} (2016), no.~2, 223--229. \MR{3551222}

\bibitem{JanMuc59}
{Ju}rij~Ivanovi{\v{c}} Janov and A{\Palatalization{l}}bert~Abramovi{\v{c}}
  Mu{\v{c}}nik, \emph{{\Cyr{\CYRO{}
  \cyrs\cyru\cyrshch\cyre\cyrs\cyrt\cyrv\cyro\cyrv\cyra\cyrn\cyri\cyri{}
  {$k$}-\cyrz\cyrn\cyra\cyrch\cyrn\cyrery\cyrh{}
  \cyrz\cyra\cyrm\cyrk\cyrn\cyru\cyrt\cyrery\cyrh{}
  \cyrk\cyrl\cyra\cyrs\cyrs\cyro\cyrv{}, \cyrn\cyre{}
  \cyri\cyrm\cyre\cyryu\cyrshch\cyri\cyrh{}
  \cyrk\cyro\cyrn\cyre\cyrch\cyrn\cyro\cyrg\cyro{}
  \cyrb\cyra\cyrz\cyri\cyrs\cyra}} [{O}n the existence of {$k$}-valued closed
  classes having no finite basis]}, Dokl. Akad. Nauk SSSR \textbf{127} (1959),
  no.~1, 44--46. \MR{0108458 (21 \#7174)}

\bibitem{JezKep02}
Jaroslav Je\v{z}ek and Tom\'{a}\v{s} Kepka, \emph{The factor of a subdirectly
  irreducible algebra through its monolith}, {A}lgebra {U}niversalis
  \textbf{47} (2002), no.~3, 319--327. \MR{1918733}

\bibitem{Kis97}
Emil~W. Kiss, \emph{An easy way to minimal algebras}, Internat. J. Algebra
  Comput. \textbf{7} (1997), no.~1, 55--75. \MR{1428329}

\bibitem{KuznecovCentralisers1979}
Aleksandr~Vladimirovi{\v{c}} Kuznecov, \emph{{\Cyr{\CYRO{}
  \cyrs\cyrr\cyre\cyrd\cyrs\cyrt\cyrv\cyra\cyrh{} \cyrd\cyrl\cyrya{}
  \cyro\cyrb\cyrn\cyra\cyrr\cyru\cyrzh\cyre\cyrn\cyri\cyrya{}
  \cyrn\cyre\cyrv\cyrery\cyrv\cyro\cyrd\cyri\cyrm\cyro\cyrs\cyrt\cyri{}
  \cyri\cyrl\cyri{}
  \cyrn\cyre\cyrv\cyrery\cyrr\cyra\cyrz\cyri\cyrm\cyro\cyrs\cyrt\cyri{}}}
  [{M}eans for detection of nondeducibility and inexpressibility]},
  {\Cyr{\CYRL\cyro\cyrg\cyri\cyrch\cyre\cyrs\cyrk\cyri\cyrishrt{}
  \cyrv\cyrery\cyrv\cyro\cyrd{}}} [Logical Inference] ({M}os\-cow, 1974),
  ``{N}auka'', {M}os\-cow, 1979, In Russian, pp.~5--33. \MR{720219 (84k:03036)}

\bibitem{McKMcnTay88}
Ralph~N. Mc{K}enzie, George~F. Mc{N}ulty, and Walter~F. Taylor, \emph{Algebras,
  lattices, varieties. {V}ol. {I}}, The Wadsworth \& Brooks/Cole Mathematics
  Series, vol.~{I}, Wadsworth \& Brooks/Cole Advanced Books \& Software,
  Monterey, CA, 1987. \MR{883644 (88e:08001)}

\bibitem{PilzNR}
G\"{u}nter Pilz, \emph{Near-rings. the theory and its applications},
  North-Holland Mathematics Studies, vol.~23, North-Holland Publishing Company,
  Amsterdam, New York, Oxford, 1977. \MR{0469981}

\bibitem{Pin15a}
Aleksandr~Georgievich Pinus, \emph{On algebras with identical algebraic sets},
  Algebra i Logika \textbf{54} (2015), no.~4, 493--502, 544, 547, translation
  in Algebra and Logic \textbf{54} (2015), no. 4, 316--322. \MR{3468412}

\bibitem{Pin17a}
\bysame, \emph{Algebraically equivalent clones}, Algebra i Logika \textbf{55}
  (2016), no.~6, 760--768, translation in Algebra and Logic \textbf{55} (2017),
  no. 6, 501--506. \MR{3722416}

\bibitem{Pin17c}
\bysame, \emph{Algebraic sets of universal algebras and algebraic closure
  operator}, Lobachevskii J. Math. \textbf{38} (2017), no.~4, 719--723.
  \MR{3673285}

\bibitem{Plo98}
Boris~Isaakovich Plotkin, \emph{Some concepts of algebraic geometry in
  universal algebra}, Algebra i Analiz \textbf{9} (1997), no.~4, 224--248,
  translation in St. Petersburg Math. J. \textbf{9} (1998), no. 4, 859--879.
  \MR{1604318}

\bibitem{PosKal79}
Reinhard P{\"o}schel and Lev~Arka{\Palatalization{d}}evi{\v{c}} Kalu{\v{z}}nin,
  \emph{Funk\-tio\-nen- und {R}e\-la\-tio\-nen\-al\-ge\-bren},
  Ma\-the\-ma\-ti\-sche {M}o\-no\-gra\-phi\-en, vol.~15, {VEB} {D}eut\-scher
  {V}er\-lag der {W}is\-sen\-schaf\-ten, Berlin, 1979. \MR{543839 (81f:03075)}

\bibitem{Pos41}
Emil~Leon Post, \emph{The two-valued iterative systems of mathematical logic},
  Annals of Mathematics Studies, vol.~5, Princeton University Press, Princeton,
  N. J., 1941. \MR{0004195}

\bibitem{Sco97}
Stuart~D. Scott, \emph{The structure of {$\Omega$}-groups}, Nearrings,
  nearfields and {$K$}-loops ({H}amburg, 1995), Math. Appl., vol. 426, Kluwer
  Acad. Publ., Dordrecht, 1997, pp.~47--137. \MR{1492187}

\bibitem{Smith}
Jonathan D.~H. Smith, \emph{{M}a{\Palatalization{l}}cev varieties}, Lecture
  Notes in Mathematics, vol. 554, Springer-Ver\-lag, Berlin, New York, 1976.
  \MR{0432511}

\bibitem{TotWal17}
Endre T{\'{o}}th and Tam{\'{a}}s Waldhauser, \emph{On the shape of solution
  sets of systems of (functional) equations}, Aequationes Math. \textbf{91}
  (2017), no.~5, 837--857. \MR{3697173}

\bibitem{Zhu15}
Dmitriy Zhuk, \emph{The lattice of all clones of self-dual functions in
  three-valued logic}, J. Mult.-Valued Logic Soft Comput. \textbf{24} (2015),
  no.~1--4, 251--316. \MR{3277337}

\end{thebibliography}
\end{document}